\documentclass[a4paper, reqno]{amsart}

\usepackage{tikz} 
\usetikzlibrary{cd}
\usetikzlibrary{patterns,arrows,positioning,decorations.pathmorphing}
\usepackage{xparse} 

\usepackage{amsthm} 
\usepackage{amsmath} 
\usepackage{hyperref} 
\usepackage{amssymb} 
\usepackage{longtable} 
\usepackage{array} 
\usepackage{graphicx} 
\usepackage{mathtools} 
\usepackage{enumitem} 

\newcommand{\La}{\Lambda}
\newcommand{\tn}{\tau_n^{\phantom{0}}}
\newcommand{\tno}{\tau_n^-}

\newcommand{\cA}{\mathcal{A}}
\newcommand{\cC}{\mathcal{C}}
\newcommand{\cD}{\mathcal{D}}
\newcommand{\cF}{\mathcal{F}}
\newcommand{\cP}{\mathcal{P}}
\newcommand{\cT}{\mathcal{T}}
\newcommand{\cX}{\mathcal{X}}

\newcommand{\ZZ}{\mathbb{Z}}

\newcommand{\K}{\mathbf{k}}
\newcommand{\om}{\Omega}
\newcommand{\isom}{\cong}
\newcommand{\divides}{\mid}
\newcommand{\modulo}[1]{\ (\mathrm{mod}\ #1)}

\NewDocumentCommand\glue{ O{} O{}}{\mathbin{\raisebox{0.2ex}{${}^{\mbox{\tiny$#1$}}$\rotatebox[origin=c]{90}{$\triangleright$}$^{\mbox{\tiny$#2$}}$}}}

\newcommand{\nospacepunct}[1]{\makebox[0pt][l]{\,#1}}
\newcommand{\mo}[1]{#1\mathrm{-mod}}
\newcommand{\smo}[1]{#1\mathrm{-\underline{mod}}}
\newcommand{\cmo}[1]{#1\mathrm{-\overline{mod}}}
\newcommand{\rmax}[1]{r_{#1}}
\newcommand{\lmax}[1]{\ell_{#1}}

\DeclareMathOperator{\Hom}{Hom}
\DeclareMathOperator{\Ext}{Ext}
\DeclareMathOperator{\add}{add}
\DeclareMathOperator{\ind}{ind}
\DeclareMathOperator{\soc}{soc}
\DeclareMathOperator{\topp}{top}
\DeclareMathOperator{\gldim}{gl.dim}
\DeclareMathOperator{\pdim}{pr.dim}

\theoremstyle{plain}
\newtheorem*{theorem*}{Theorem}
\newtheorem{theorem}{Theorem}[section] 

\theoremstyle{definition}
\newtheorem{definition}[theorem]{Definition} 
\newtheorem{definition-proposition}[theorem]{Definition-Proposition} 
\newtheorem{example}[theorem]{Example} 
\newtheorem{corollary}[theorem]{Corollary} 
\newtheorem{lemma}[theorem]{Lemma} 
\newtheorem{proposition}[theorem]{Proposition} 
\newtheorem{remark}[theorem]{Remark} 
\newtheorem{question}[theorem]{Question} 
\numberwithin{equation}{section} 

\title{$n\ZZ$-cluster tilting subcategories for Nakayama algebras}

\author{Martin Herschend}
\address{Martin Herschend, Department of Mathematics, Uppsala University, Box 480, 751 06 Uppsala, Sweden}
\email{martin.herschend@math.uu.se}

\author{Sondre Kvamme}
\address{Sondre Kvamme, Department of Mathematical Sciences, Norwegian University of Science and Technology, 7491 Trondheim, Norway}
\email{sondre.kvamme@ntnu.no}

\author{Laertis Vaso}
\address{Laertis Vaso, Department of Mathematical Sciences, Norwegian University of Science and Technology, 7491 Trondheim, Norway}
\email{laertis.vaso@ntnu.no}


\begin{document}

\begin{abstract}
$n\ZZ$-cluster tilting subcategories are an ideal setting for higher dimensional Auslander--Reiten theory. We give a complete classification of $n\ZZ$-cluster tilting subcategories of module categories of Nakayama algebras. In particular, we show that there are three kinds of Nakayama algebras that admit $n\ZZ$-cluster tilting subcategories: finite global dimension, selfinjective and non-Iwanaga--Gorenstein. Only the selfinjective ones can admit more than one $n\ZZ$-cluster tilting subcategory. It has been shown by the second author, that each such $n\ZZ$-cluster tilting subcategory induces an $n\ZZ$-cluster tilting subcategory of the corresponding singularity category. For each Nakayama algebra in our classification, we describe its singularity category, the canonical functor from its module category to its singularity category, and provide a complete comparison of $n\ZZ$-cluster tilting subcategories in the module category and the singularity category. This relies heavily of results by Shen, who described the singularity categories of all Nakayama algebras.
\end{abstract}

\maketitle
\tableofcontents

\section{Introduction}
Auslander--Reiten theory is a fundamental tool to study representation theory from a homological point of view. A higher version of Auslander--Reiten theory was introduced by Iyama \cite{Iya1} for each $n\ge 1$, with $n = 1$ corresponding to the classical version. This theory has connections to algebraic geometry \cite{IW, IW2, IW3, HIMO},  combinatorics \cite{OT,HJ, Wil}, higher category theory and algebraic K-theory \cite{DJW}, representation theory, \cite{HI1,IO1,Miz}, and to wrapped Floer theory in symplectic geometry \cite{DJL}. It is also a crucial ingredient in the recent proof of the Donovan--Wemyss conjecture as announced by Keller \cite{JM} (see Conjecture 1.4 in \cite{DW} for the statement of the conjecture). In higher Auslander--Reiten theory one changes perspective from studying some category $\cA$, say finitely generated modules over a finite dimensional algebra or its bounded derived category, to studying some suitable subcategory $\cC \subseteq \cA$. Commonly $\cC$ is an $n$-cluster tilting subcategory, possibly with additional properties. 

In Definition~\ref{def:nZ-cluster tilting} we recall the definition of $n$-cluster tilting subcategories of module categories, but they can be defined in many different settings (See \cite{Iya1, GKO, Jas, HLN}). Depending on the setting, $n$-cluster tilting subcategories give rise to higher notions of various types of categories of interest in homological algebra. For instance if $\cA$ is abelian and $\cC \subseteq \cA$ is $n$-cluster tilting, then $\cC$ is $n$-abelian in the sense of \cite{Jas}. If $\cA$ is triangulated and $\cC \subseteq \cA$ is $n$-cluster tilting with the additional property that $\cC$ is closed under $n$-fold suspension then $\cC$ is $(n+2)$-angulated in the sense of \cite{GKO}. 

Let $\La$ be a finite dimensional algebra and $\mo \La$ the category of finitely generated left $\La$-modules. In this paper we are concerned with $n$-cluster tilting subcategories $\cC \subseteq \mo \La$. The existence of such a $\cC$ imposes a strong restriction on $\La$ and in the vast majority of known cases $\La$ exhibits very regular homological behaviour. For instance the case when $\gldim \La = n$ has been extensively studied (See for instance \cite{Iya2, IO1, IO2, HI1, HI2}). In this case $\cC$ is unique if it exists. Moreover, $\cC$ gives rise to an $n$-cluster tilting subcategory
\[
\add\{X[in] \mid X \in \cC,\; i\in \ZZ \} \subseteq D^b(\mo \La)
\]
of the bounded derived category $D^b(\mo \La)$, which is $(n+2)$-angulated by \cite{GKO}.

For $\La$ with $\gldim \La > n$ much fewer results are known. One of the cases that has received some attention is when $\La$ is selfinjective (see \cite{EH, IO2, DI}). In this case the stable category $\smo{\La}$ is triangulated and the image of $\cC$ in $\smo{\La}$ is $n$-cluster tilting but not necessarily $(n+2)$-angulated. Some algebras $\La$ with $\gldim \La > n$, which are not selfinjective but admit $n$-cluster tilting subcategories have been found by the third author \cite{Vas1, Vas2, Vas3, Vas4}. Another notable family of algebras admitting $n$-cluster tilting subcategories and exhibiting a variety of homological properties are the higher Nakayama algebras introduced in \cite{JK}.

Iyama and Jasso introduced the notion of $n\ZZ$-cluster tilting in \cite{IJ}, which resolves certain issues that appear in particular when $\gldim \La > n$. This property is significantly stronger than $n$-cluster tilting (see Definition~\ref{def:nZ-cluster tilting} for the definition). In particular, if $\gldim \La < \infty$ and there is $\cC \subseteq \mo \La$, $n\ZZ$-cluster tilting, then $\gldim \La \in n\ZZ$. On the other hand if  $\gldim \La = n$, then any $n$-cluster tilting subcategory is $n\ZZ$-cluster tilting. One of the benefits of this notion was found by the second author, who showed in \cite{Kva} that every $n\ZZ$-cluster tilting subcategory of $\mo \La$ gives rise to an $n\ZZ$-cluster tilting subcategory of the singularity category $D_{\operatorname{sing}}(\La)$ (see Section~\ref{sec:singularity categories} for relevant definitions). In the triangulated setting $n\ZZ$-cluster tilting means precisely closure under $n$-fold suspension and so $(n+2)$-angulated categories are obtained in this way. Note that if $\gldim \La < \infty$, then $D_{\operatorname{sing}}(\La) = 0$ and if $\La$ is selfinjective, then $D_{\operatorname{sing}}(\La) = \smo{\La}$. 

For these reasons given above we are motivated to search for algebras $\La$ admitting $n\ZZ$-cluster tilting subcategories $\cC \subseteq \mo \La$. In particular, we are interested in cases when $\gldim \La = \infty$ and $\La$ is not selfinjective. 

We limit our search to Nakayama algebras (see Section~\ref{sec:nakayama algebras} for relevant definitions) as it is a class of algebras that is not too large and for which computations can be done easily. Still it allows for algebras of the complex homological behaviour we are looking for. Our main result is a classification of all Nakayama algebras that admit an $n\ZZ$-cluster tilting subcategory. Moreover, for each algebra in our classifying list we describe all its $n\ZZ$-cluster tilting subcategories explicitly. We note that although there are many results giving examples of algebras with $n$-cluster tilting subcategories as mentioned above, there are few that give a complete classification for a given family of algebras. Notable exceptions for which classification results have been obtained are selfinjective Nakayama algebras \cite{DI}, radical square zero algebras \cite{Vas4} and gentle algebras \cite{HJS}.

Our classification is subdivided into four cases depending on the quiver with relation describing the Nakayama algebra $\Lambda$. With respect to the quiver we subdivide depending on if it acyclic or not. With respect to the relations we subdivide depending on if the relations are homogeneous (meaning that the relations are given by all paths of some fixed length) or not. The classification for homogeneous relations is given in Theorem~\ref{thrm:homogeneous Nakayama with nZ-ct}. The classification for non-homogeneous relations  is given in Theorem~\ref{thrm:acyclic non-homogeneous case} (acyclic) and Theorem~\ref{thrm:cyclic non-homogeneous case} (cyclic)

The homogeneous acyclic Nakayama algebras contain all the ones of global dimension $n$ that admit an $n\ZZ$-cluster tilting subcategory. Since $n$-cluster tilting is equivalent to $n\ZZ$-cluster tilting for algebras of global dimension n, we can rely on a previous classification result due to the third author \cite[Theorem 3]{Vas1} to determine these algebras. For the non-homogeneous acyclic Nakayama algebras it turns out that we can cover all cases using the method of gluing developed in \cite{Vas3}.

The homogeneous cyclic Nakayama algebras are the same as the selfinjective ones. For these Darp{\"o} and Iyama \cite{DI} have classified which of them admit an $n$-cluster tilting subcategory. Thus our problem is reduced to finding out which among these are $n\ZZ$-cluster tilting. For the non-homogeneous cyclic Nakayama algebras it turns out that we can cover all cases using the self-gluing as developed in \cite{Vas2}. The algebras that appear in this final case turn out to be homologically the least regular, in that they are not Iwanaga--Gorenstein. We recall that an algebra is called \emph{Iwanaga--Gorenstein} if it has finite left and right injective dimensions as a module over itself.

We may also subdivide our classification from a homological point of view. In summary it goes as follows. To obtain a Nakayama algebra $\La$ with an $n\ZZ$-cluster tilting subcategory, decide if you prefer $\gldim \La = rn$ for some $r \ge 1$, for $\La$ to be non-Iwanaga--Gorenstein or for $\La$ to be selfinjective. In the first case pick any $r$ Nakayama algebras of global dimension $n$, each admitting an $n\ZZ$-cluster tilting module, and glue them together in any order. For the second option do the same and then self-glue the result. In these cases the $n\ZZ$-cluster tilting subcategory is unique. For $\La$ selfinjective, make sure that $n$ divides the number of simples and $\La$ has Loewy length $2$ or $n+2$. In that case $\La$ admits precisely $n$ distinct $n\ZZ$-cluster tilting subcategories. These are the only options as every other Nakayama algebra does not admit an $n\ZZ$-cluster tilting subcategory.

In Section~\ref{sec:singularity categories} we let $\La$ be a Nakayama algebra with an $n\ZZ$-cluster tilting subcategory and study its singularity category $D_{\operatorname{sing}}(\La)$. We describe the canonical functor from the module category to the singularity category and the image of its $n\ZZ$-cluster tilting subcategory. If $\La$ is acyclic, then $\gldim \La < \infty$ and its singularity category is zero, so there is nothing to discuss. If $\La$ is selfinjective then the singularity category coincides with the stable module category and there is a bijection between $n\ZZ$-cluster tilting subcategories in the module category and in the singularity category. The final case to consider is when $\La$ is cyclic and non-homogeneous. We use a description of singularity categories of Nakayama algebras due to Shen \cite{Shen}, to show that $D_{\operatorname{sing}}(\La)$ is equivalent to the stable category of a radical square zero cyclic Nakayama algebra $\Gamma$, see Corollary~\ref{cor:SingCatNonHomog}. By our previous results we know that $\Gamma$ admits precisely $n$ distinct $n\ZZ$-cluster tilting subcategories and we determine which of them corresponds to the unique $n\ZZ$-cluster tilting subcategory coming from $\La$.

By our results we know exactly which $n\ZZ$-cluster tilting subcategories appear in singularity categories of Nakayama algebras. Since they all have finitely many indecomposable objects they can equally be described as $n\ZZ$-cluster tilting objects by taking the direct sum of all indecomposables. It has recently been shown by Jasso and Muro \cite{JM} that any algebraic triangulated category with an $n\ZZ$-cluster tilting objects is determined by the endomorphism algebra of that object (under the assumption that this endomorphism algebra is finite dimensional over a perfect ground field). As a consequence any algebraic triangulated category admitting an $n\ZZ$-cluster tilting object with the same endomorphism algebra as one appearing in this paper must be equivalent to a singularity category of a Nakayama algebra.

\section{Preliminaries}

\subsection{Conventions and notation} 
Let us start by setting conventions and introducing notation. Throughout this paper we fix a positive integer $n \ge 2$ and a ground field $\K$. All subcategories are assumed to be full. By algebra we mean associative $\K$-algebra and by module we mean left module. We denote by $D$ the duality $\Hom_\K(-,\K)$. 

Let $\La$ be a finite dimensional algebra. We consider the category $\mo{\La}$ of finitely generated left $\La$-modules. We denote by $\smo{\La}$ the \emph{projectively stable module category} of $\La$, that is the category with objects the same as $\mo{\La}$ and morphisms given by $\underline{\Hom}(M,N)=\Hom(M,N)/\mathcal{P}(M,N)$ where $\mathcal{P}(M,N)$ denotes the subspace of morphisms factoring through projective modules. We denote by $\om:\smo{\La} \to \smo{\La} $ the \emph{syzygy functor} defined by $\om(M)$ being the kernel of a projective cover $P\twoheadrightarrow M$. The \emph{injectively stable module category} $\cmo{\La}$ and the \emph{cosyzygy functor} $\om^{-}:\cmo{\La} \to \cmo{\La}$ are defined dually. We consider the \emph{$n$-Auslander--Reiten translations} $\tn:\smo{\La}\to\cmo{\La}$ and $\tno:\cmo{\La}\to\smo{\La}$ defined by $\tn=\tau\om^{n-1}$ and $\tno=\tau^{-}\om^{-(n-1)}$, where $\tau$ and $\tau^{-}$ denote the usual Auslander--Reiten translations.

For a module $M\in\mo{\La}$ we denote by $\add(M)$ the subcategory of $\mo{\La}$ containing all direct summands of finite direct sums of $M$. For $i=1,\ldots,s$ let $M_i\in \mo{\La}$ be a module and let $\cC_i=\add(M_i)$. We define \[\add(\cC_1,\ldots,\cC_s)\coloneqq\add(M_1,\ldots, M_s)\coloneqq \add\left(\bigoplus_{i=1}^s M_i\right).\]

For a finite quiver $Q$ we denote its path algebra over $\K$ by $\K Q$. If $\alpha$ is an arrow in $Q$ from $i$ to $j$ we write $i \overset{\alpha}\to j$. For each vertex $i$ in $Q$ we denote the path of length $0$ at $i$ by $\epsilon_i$. Multiplication in $\K Q$ is defined by concatenation of paths in such a way that if $i \overset{\alpha}\to j$ is an arrow, then $\epsilon_j\alpha\epsilon_i = \alpha$. We denote the two sided $\K Q$-ideal generated by all arrows by $R_Q = R$. A two sided $\K Q$-ideal $I$ is called \emph{admissible} if $R^2 \supseteq I \supseteq R^N$ for some $N \ge 2$.

\subsection{\texorpdfstring{$n\ZZ$}{nZ}-cluster tilting subcategories}
Let $\cD$ be a subcategory of a category $\cC$ and let $x\in\cC$. A \emph{right $\cD$-approximation} of $x$ is a morphism $f\colon a\to x$ with $a\in \cD$ such that all morphisms $g\colon b\to x$ with $b\in \cD$ factor through $f$. If every $x\in\cC$ admits a right $\cD$-approximation, then we say that $\cD$ is \emph{contravariantly finite (in $\cC$)}. The notions \emph{left $\cD$-approximation} and \emph{covariantly finite} are defined dually. We say that $\cD$ is \emph{functorially finite (in $\cC$)} if $\cD$ is both contravariantly and covariantly finite. Notice that if $M\in\mo{\La}$, then $\add(M)$ is functorially finite.

We recall the following definition from \cite{IJ}.

\begin{definition}\label{def:nZ-cluster tilting}
\begin{enumerate}[label=(\alph*)]
\item We call a subcategory $\cC$ of $\mo \La$ an \emph{$n$-cluster tilting subcategory} if it is functorially finite and
\begin{align*}
\cC&=\{X\in\mo \La \mid \Ext^i_{\La}(\cC,X)=0 \text{ for all $0<i<n$}\}\\
&=\{X\in\mo \La \mid \Ext^i_{\La}(X,\cC)=0 \text{ for all $0<i<n$}\}.
\end{align*}
If moreover $\Ext^i_{\La}(\cC,\cC)\neq 0$ implies that $i\in n\ZZ$, then we call $\cC$ an \emph{$n\ZZ$-cluster tilting subcategory}.
\item Let $M\in\mo \La$. If $\add(M)$ is an $n$-cluster tilting subcategory (respectively $n\ZZ$-cluster tilting subcategory) of $\mo \La$, then we call $M$ an \emph{$n$-cluster tilting module} (respectively \emph{$n\ZZ$-cluster tilting module}).
\end{enumerate}
\end{definition}

We collect some basic properties of $n\ZZ$-cluster tilting subcategories in the following proposition. 

\begin{proposition}\label{prop:basic nZ-cluster tilting results}
Let $\La$ be a finite dimensional algebra and let $\cC\subseteq \mo \La$ be an $n$-cluster tilting subcategory. Then the following statements hold.
\begin{enumerate}
    \item[(a)] $\La\in\cC$ and $D(\La)\in\cC$.
    \item[(b)] Denote by $\cC_P$ and $\cC_I$ the sets of isomorphism classes of indecomposable non-projective respectively non-injective modules in $\cC$. Then $\tn$ and $\tno$ induce mutually inverse bijections
    \[\begin{tikzpicture}
        \node (0) at (0,0) {$\cC_P$};
        \node (1) at (2,0) {$\cC_I$.};

        \draw[-latex] (0) to [bend left=20] node [above] {$\tn$} (1);
        \draw[-latex] (1) to [bend left=20] node [below] {$\tno$} (0);
    \end{tikzpicture}\] 
    \item[(c)] $\om^i M$ is indecomposable for all $M\in \cC_P$ and $0<i<n$.
    \item[(d)] $\om^{-i}N$ is indecomposable for all $N\in \cC_I$ and $0<i<n$.
\end{enumerate}
Moreover, if $\La$ is representation-directed, then any subcategory of $\mo \La$ which is closed under finite direct sums and summands and satisfies (a)-(d) is an $n$-cluster tilting subcategory.
\end{proposition}

\begin{proof}
Part (a) follows immediately by the definition of an $n$-cluster tilting subcategory and the facts that $\Ext_{\La}^i(\La,M)=0$ and $\Ext_{\La}^i(M,D(\La))=0$ for all $M\in \mo \La$. For part (b) we refer to \cite[Section 1.4.1]{Iya1}. For parts (c) and (d) we refer to \cite[Corollary 3.3]{Vas1}. For the reverse implication if $\La$ is representation-directed, we refer to \cite[Theorem 1]{Vas1}. 
\end{proof}

\begin{corollary}\label{cor:unique n-ct for representation-directed}
If $\La$ is representation-directed, then any $n$-cluster tilting subcategory of $\mo \La$ is unique and given by $\cC=\add\left\{\tau_n^{-r}(\La)\mid r\geq 0\right\}$.
\end{corollary}

\begin{proof}
If $\cD\subseteq \mo\La$ is an $n$-cluster tilting subcategory, then $\cC\subseteq \cD$ by Proposition~\ref{prop:basic nZ-cluster tilting results}(a) and (b). Since $\cD$ is an $n$-cluster tilting subcategory, it follows that $\cC$ satisfies Proposition~\ref{prop:basic nZ-cluster tilting results}(b)--(d). Since we also have $\La\in\cC$, it follows that $\cC$ is $n$-cluster tilting by \cite[Theorem 1]{Vas1}. Since $\cC\subseteq \cD$ are both $n$-cluster tilting subcategories, we conclude that $\cC=\cD$.
\end{proof}

\begin{proposition}\label{prop:basic nZ-cluster tilting results2}
	Let $\La$ be a finite dimensional algebra and let $\cC\subseteq \mo \La$ be an $n$-cluster tilting subcategory. Then the following statements are equivalent.
	\begin{enumerate}
		\item[(a)] $\cC$ is $n\ZZ$-cluster tilting.
		\item[(b)] $\om^{n}(\cC)\subseteq \cC$.
		\item[(c)] $\om^{-n}(\cC)\subseteq \cC$.
	\end{enumerate}
	In this case, the following statements hold.
	\begin{enumerate}
		\item[(d)] If $M\in \cC_P$, then $\om^-\tau(M)\in \cC$.
		\item[(e)] If $N\in \cC_I$, then $\om\tau^-(N)\in\cC$.  
	\end{enumerate}
\end{proposition}

\begin{proof}
For the equivalence between (a), (b) and (c) we refer to \cite[Section 2.2]{IJ}. For part (d) notice that if $M\in\cC_P$, then $M\isom \tno(N)$ for some $N\in\cC_I$ by Proposition \ref{prop:basic nZ-cluster tilting results} (b). By  Proposition \ref{prop:basic nZ-cluster tilting results} (d) it follows that $\om^{-(n-1)}(N)$ is indecomposable and not injective so
\[\tau(M)\isom \tau\tno(N) \isom \tau\tau^-\om^{-(n-1)}(N) \isom \om^{-(n-1)}(N).\]
Applying $\om^{-}$ in the above we have
\[\om^-\tau(M) \isom \om^{-n}(N),\]
where $\om^{-n}(N)\in\cC$ by (c). Hence $\om^{-}\tau(M)\in\cC$. Part (e) follows similarly.	
\end{proof}

\subsection{Nakayama algebras}\label{sec:nakayama algebras}
In this section we discuss (connected) Nakayama algebras in terms of quivers with relations as well as the shape of their Auslander--Reiten quivers. Note that everything stated is essentially known, but we still provide some proofs for the reader's convenience.

Let $m \ge 1$ be a positive integer and $Q_m \in \{A_m,\tilde{A}_m\}$, where $A_m$ and $\tilde{A}_m$ are the following quivers:
\[\begin{tikzpicture}
    \node (Q) at (-1,0) {$A_m:$};
    \node (1) at (0,0) {$m$}; 
    \node (2) at (2,0) {$m-1$};
    \node (3) at (4,0) {$m-2$};
    \node (4) at (6,0) {$\cdots$};
    \node (m-1) at (7.5,0) {$2$};
    \node (m) at (9,0) {$1$};
    
    \draw[-{Stealth[scale=0.5]}] (1) -- (2) node[draw=none,midway,above] {$\alpha_{m}$};
    \draw[-{Stealth[scale=0.5]}] (2) -- (3) node[draw=none,midway,above] {$\alpha_{m-1}$};
    \draw[-{Stealth[scale=0.5]}] (3) -- (4) node[draw=none,midway,above] {$\alpha_{m-2}$};
    \draw[-{Stealth[scale=0.5]}] (4) -- (m-1) node[draw=none,midway,above] {$\alpha_{3}$};
    \draw[-{Stealth[scale=0.5]}] (m-1) -- (m) node[draw=none,midway,above] {$\alpha_{2}$};
\end{tikzpicture}\]
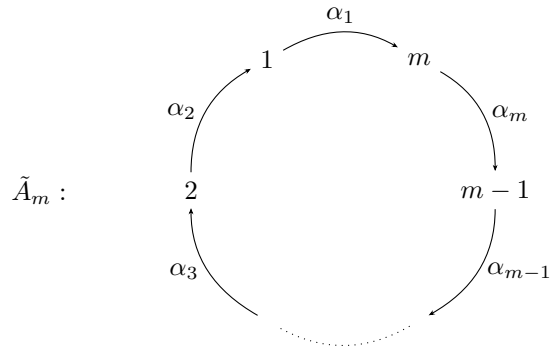
\[\begin{tikzpicture}
    \node (Q) at (-4,0) {$\tilde{A}_m:$};
    \node (1) at (1,1.732050808) {$m$};
    \node (0) at (-1,1.732050808) {$1$};
    \node (m-1) at (-2,0) {$2$};
    \node (m-2) at (-1,-1.732050808) {$ $};
    \node (dots) at (1,-1.732050808) {$ $};
    \node (2) at (2,0) {$m-1$};
    
    \draw[-{Stealth[scale=0.5]}] (0) to [out=30, in=150]  node[draw=none, above] {$\alpha_{1}$} (1);
    \draw[-{Stealth[scale=0.5]}] (1) to [out=-30, in=90] node[draw=none, midway, right] {$\alpha_{m}$} (2);
    \draw[-{Stealth[scale=0.5]}] (m-1) to [out=90, in=-150] node[draw=none, midway, left] {$\alpha_{2}$} (0);
    \draw[-{Stealth[scale=0.5]}] (m-2) to [out=150, in=-90] node[draw=none, midway, left] {$\alpha_{3}$} (m-1);
    \draw[-{Stealth[scale=0.5]}] (2) to [out=-90, in=30] node[draw=none, midway, right] {$\alpha_{m-1}$} (dots);
    \draw[dotted] (dots) to [out=-150,in=-30] (m-2);
\end{tikzpicture}\]
We say that an algebra $\La$ is a \emph{Nakayama algebra} if $\La = KQ_m/I$ for some admissible ideal $I$. We call $\La$ \emph{cyclic} if $Q_m = \tilde{A}_m$ and \emph{acyclic} if $Q_m = A_m$. Moreover, we say that $\La$ is a \emph{Nakayama algebra with homogeneous relations} or simply a \emph{homogeneous Nakayama algebra} if $I = R_{Q_m}^l$ for some $l\ge 2$. Note that in both cases $Q_m$ has $m$ vertices, which is perhaps non-standard but means that $A_m$ is obtained from $\tilde{A}_m$ by removing the arrow $\alpha_1$, which is useful for our purposes. 

Our aim is to study $n\ZZ$-cluster tilting subcategories for Nakayama algebras. The case $\La = \K$ is not very interesting as $\La$ is semisimple and $\cC =\mo \La$ is the unique $n\ZZ$-cluster tilting subcategory for any $n$. To simplify the discussion we therefore exclude this case and assume from now on that the quiver of $\La$ has at least one arrow. Note that the restriction to connected Nakayama algebras is mainly for convenience as it does not affect the existence of $n\ZZ$-cluster tilting subcategories in any essential way.

The representation theory of Nakayama algebras is well-understood, see for example \cite[Chapter V]{ASS}. In particular, indecomposable modules over Nakayama algebras are uniserial, see \cite[Corollary V.3.6]{ASS}. To describe all indecomposable modules over a Nakayama algebra $\La$ we first introduce a $\K \tilde{A}_m$-module $M(i,j)$ for each pair of integers $i \le j$. We define $M(i,j)$ to have a basis $\{b_t \mid i \le t \le j\}$ and $\K \tilde{A}_m$-action defined by
\[
\epsilon_sb_t = 
\begin{cases} 
b_t &\mbox{if } s-t \in m\ZZ,\\ 
0 &\mbox{otherwise,}
\end{cases}
\quad
\mbox{and}
\quad
\alpha_sb_t = 
\begin{cases} 
b_{t-1} &\mbox{if } s-t \in m\ZZ \mbox{ and } i \le t-1,\\ 
0 &\mbox{otherwise.}
\end{cases}
\]
Note that $M(i+m,j+m) \isom M(i,j)$ by renaming basis vectors in the obvious way. For simplicity we will consider this isomorphism as an identity. Similarly, if $i \le k \le j$, then there is a monomorphism $M(i,k) \to M(i,j)$ and an epimorphism $M(i,j) \to M(k,j)$ obtained by forgetting some of the basis vectors. Hence we consider $M(i,k)$ and $M(k,j)$ as a submodule respectively quotient module of $M(i,j)$ in this case. 

To deal with acyclic Nakayama algebras we will consider $M(i,j)$ as a $\K A_m$-module whenever $\alpha_1M(i,j) = 0$. In this situation we may as well assume $1 \le i \le j \le m$.

Now let $\La$ be a Nakayama algebra. If $\La= \K \tilde{A}_m/I$ and $IM(i,j)  = 0$, then $M(i,j)$ defines a $\La$-module. Similarly, if $\La= \K A_m/I$ and $\alpha_1M(i,j) = 0$ as well as $IM(i,j)  = 0$, then $M(i,j)$ again defines a $\La$-module. We will indicate either of these situations simply by writing $M(i,j) \in \mo{\La}$ as a condition on the pair $(i,j)$.

The modules $M(i,j) \in \mo{\La}$ classify all indecomposable $\La$-modules up to isomorphism. Moreover, almost split sequences in $\mo{\La}$ are straightforward to compute for instance as in \cite[Theorem V.4.1]{ASS}. Thus we obtain the following description of the Auslander--Reiten quiver of $\La$.

\begin{proposition}\label{prop:AR quiver of Nakayama}
Let $\La=\K Q_m/I$ be a Nakayama algebra where $Q_m\in\{A_m,\tilde{A}_m\}$. Then the Auslander--Reiten quiver $\Gamma = \Gamma(\La)$ of $\La$ can be described as follows.
\begin{itemize}
    \item[$\bullet$] The set of vertices $\Gamma_0 = \{(i,j) \in \ZZ^2/\ZZ(m,m) \mid i \le j \mbox{ and } M(i,j) \in \mo{\La}\}$.
    \item[$\bullet$] For each $(i,j) \in \Gamma_0$ there is an arrow $(i,j) \to (i,j+1)$ if $(i,j+1) \in \Gamma_0$ and an arrow $(i,j) \to (i+1,j)$ if $(i+1,j) \in \Gamma_0$.
\end{itemize}
Moreover, all arrows $(i,j)\to (i,j+1)$ correspond to monomorphisms, all arrows $(i,j)\to (i+1,j)$ correspond to epimorphisms and the Auslander--Reiten translation is given by $\tau(i,j) = (i-1,j-1)$, whenever it is defined.
\end{proposition}

To navigate the Auslander--Reiten quiver it is convenient to encode its shape as follows.

\begin{definition}
Let $\Gamma$ be the Auslander--Reiten quiver of a Nakayama algebra $\La$.

\begin{enumerate}[label=(\alph*)]
\item For $i \in \ZZ$ define $\rmax i \in \ZZ$ such that $(i,\rmax i) \in \Gamma_0$ with $\rmax i - i$ maximal.
\item For $j\in \ZZ$ define $\lmax j \in \ZZ$ such that $(\lmax j,j) \in \Gamma_0$ with $j- \lmax j$ maximal.
\end{enumerate}
\end{definition}

As an illustration we consider the following examples

\begin{example}\label{ex:AR quiver generic}
\begin{enumerate}
    \item[(a)] Let $\La=\K A_7/\langle \alpha_2\alpha_3\alpha_4, \alpha_5\alpha_6 \rangle$. The Auslander--Reiten quiver $\Gamma(\La)$ of $\La$ is
    
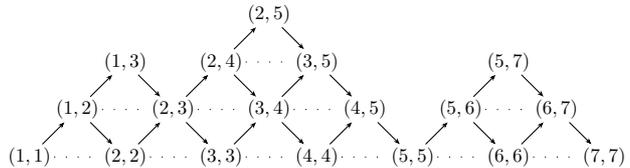
\[
    \begin{tikzpicture}[scale=0.9, transform shape, baseline={(current bounding box.center)}]
    
    \tikzstyle{mod}=[rectangle, minimum width=6pt, draw=none, inner sep=1.5pt, scale=0.8]
    
    \node[mod] (11) at (0,0) {$(1,1)$};
    \node[mod] (22) at (1.4,0) {$(2,2)$};
    \node[mod] (33) at (2.8,0) {$(3,3)$};
    \node[mod] (44) at (4.2,0) {$(4,4)$};
    \node[mod] (55) at (5.6,0) {$(5,5)$};
    \node[mod] (66) at (7,0) {$(6,6)$};
    \node[mod] (77) at (8.4,0) {$(7,7)$\nospacepunct{.}};
    
    \draw[loosely dotted] (11.east) -- (22);
    \draw[loosely dotted] (22.east) -- (33);
    \draw[loosely dotted] (33.east) -- (44);
    \draw[loosely dotted] (44.east) -- (55);
    \draw[loosely dotted] (55.east) -- (66);
    \draw[loosely dotted] (66.east) -- (77);
        
    \node[mod] (12) at (0.7,0.7) {$(1,2)$};    
    \node[mod] (23) at (2.1,0.7) {$(2,3)$};    
    \node[mod] (34) at (3.5,0.7) {$(3,4)$};    
    \node[mod] (45) at (4.9,0.7) {$(4,5)$};    
    \node[mod] (56) at (6.3,0.7) {$(5,6)$};    
    \node[mod] (67) at (7.7,0.7) {$(6,7)$};  

    \draw[-{Stealth[scale=0.5]}] (11) -- (12);
    \draw[-{Stealth[scale=0.5]}] (22) -- (23);
    \draw[-{Stealth[scale=0.5]}] (33) -- (34);
    \draw[-{Stealth[scale=0.5]}] (44) -- (45);
    \draw[-{Stealth[scale=0.5]}] (55) -- (56);
    \draw[-{Stealth[scale=0.5]}] (66) -- (67);
    
    \draw[-{Stealth[scale=0.5]}] (12) -- (22);
    \draw[-{Stealth[scale=0.5]}] (23) -- (33);
    \draw[-{Stealth[scale=0.5]}] (34) -- (44);
    \draw[-{Stealth[scale=0.5]}] (45) -- (55);
    \draw[-{Stealth[scale=0.5]}] (56) -- (66);
    \draw[-{Stealth[scale=0.5]}] (67) -- (77);

    \draw[loosely dotted] (12.east) -- (23);
    \draw[loosely dotted] (23.east) -- (34);
    \draw[loosely dotted] (34.east) -- (45);
    \draw[loosely dotted] (56.east) -- (67);
    
    \node[mod] (13) at (1.4,1.4) {$(1,3)$};
    \node[mod] (24) at (2.8,1.4) {$(2,4)$};
    \node[mod] (35) at (4.2,1.4) {$(3,5)$};
    \node[mod] (57) at (7,1.4) {$(5,7)$};

    \draw[loosely dotted] (24.east) -- (35);    
    
    \draw[-{Stealth[scale=0.5]}] (12) -- (13);
    \draw[-{Stealth[scale=0.5]}] (23) -- (24);
    \draw[-{Stealth[scale=0.5]}] (34) -- (35);
    \draw[-{Stealth[scale=0.5]}] (56) -- (57);
    
    \draw[-{Stealth[scale=0.5]}] (13) -- (23);
    \draw[-{Stealth[scale=0.5]}] (24) -- (34);
    \draw[-{Stealth[scale=0.5]}] (35) -- (45);
    \draw[-{Stealth[scale=0.5]}] (57) -- (67);
    
    \node[mod] (25) at (3.5,2.1) {$(2,5)$};

    \draw[-{Stealth[scale=0.5]}] (24) -- (25);
    \draw[-{Stealth[scale=0.5]}] (25) -- (35);
    
    \end{tikzpicture}
    \]
    Moreover, 
    \[
    \rmax 1 = 3, \quad  \rmax 2 = 5, \quad  \rmax 3 = 5, \quad  \rmax 4 = 5, \quad  \rmax 5 = 7, \quad  \rmax 6 = 7, \quad  \rmax 7 = 7
    \]
    and
    \[
    \lmax 1 = 1, \quad  \lmax 2 = 1, \quad  \lmax 3 = 1, \quad  \lmax 4 = 2, \quad  \lmax 5 = 2, \quad  \lmax 6 = 5, \quad  \lmax 7 = 5.
    \]
    
    \item[(b)] Let $\tilde{\La}=\K \tilde{A}_7/\langle \alpha_2\alpha_3\alpha_4, \alpha_5\alpha_6,\alpha_7\alpha_1, \alpha_1\alpha_2\alpha_3 \rangle$ be a cyclic Nakayama algebra. The Auslander--Reiten quiver $\Gamma(\tilde{\La})$ of $\tilde{\La}$ is
   
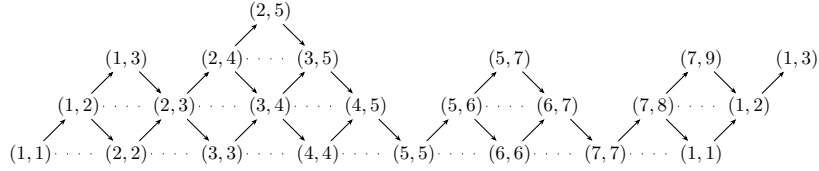
\[
    \begin{tikzpicture}[scale=0.9, transform shape, baseline={(current bounding box.center)}]
    
    \tikzstyle{mod}=[rectangle, minimum width=6pt, draw=none, inner sep=1.5pt, scale=0.8]
    
    \node[mod] (11) at (0,0) {$(1,1)$};
    \node[mod] (22) at (1.4,0) {$(2,2)$};
    \node[mod] (33) at (2.8,0) {$(3,3)$};
    \node[mod] (44) at (4.2,0) {$(4,4)$};
    \node[mod] (55) at (5.6,0) {$(5,5)$};
    \node[mod] (66) at (7,0) {$(6,6)$};
    \node[mod] (77) at (8.4,0) {$(7,7)$};
    \node[mod] (88) at (9.8,0) {$(1,1)$};
    
    \draw[loosely dotted] (11.east) -- (22);
    \draw[loosely dotted] (22.east) -- (33);
    \draw[loosely dotted] (33.east) -- (44);
    \draw[loosely dotted] (44.east) -- (55);
    \draw[loosely dotted] (55.east) -- (66);
    \draw[loosely dotted] (66.east) -- (77);
    \draw[loosely dotted] (77.east) -- (88);
        
    \node[mod] (12) at (0.7,0.7) {$(1,2)$};    
    \node[mod] (23) at (2.1,0.7) {$(2,3)$};    
    \node[mod] (34) at (3.5,0.7) {$(3,4)$};    
    \node[mod] (45) at (4.9,0.7) {$(4,5)$};    
    \node[mod] (56) at (6.3,0.7) {$(5,6)$};    
    \node[mod] (67) at (7.7,0.7) {$(6,7)$};  
    \node[mod] (78) at (9.1,0.7) {$(7,8)$};
    \node[mod] (89) at (10.5,0.7) {$(1,2)$};

    \draw[-{Stealth[scale=0.5]}] (11) -- (12);
    \draw[-{Stealth[scale=0.5]}] (22) -- (23);
    \draw[-{Stealth[scale=0.5]}] (33) -- (34);
    \draw[-{Stealth[scale=0.5]}] (44) -- (45);
    \draw[-{Stealth[scale=0.5]}] (55) -- (56);
    \draw[-{Stealth[scale=0.5]}] (66) -- (67);
    \draw[-{Stealth[scale=0.5]}] (77) -- (78);
    \draw[-{Stealth[scale=0.5]}] (88) -- (89);
    
    \draw[-{Stealth[scale=0.5]}] (12) -- (22);
    \draw[-{Stealth[scale=0.5]}] (23) -- (33);
    \draw[-{Stealth[scale=0.5]}] (34) -- (44);
    \draw[-{Stealth[scale=0.5]}] (45) -- (55);
    \draw[-{Stealth[scale=0.5]}] (56) -- (66);
    \draw[-{Stealth[scale=0.5]}] (67) -- (77);
    \draw[-{Stealth[scale=0.5]}] (78) -- (88);

    \draw[loosely dotted] (12.east) -- (23);
    \draw[loosely dotted] (23.east) -- (34);
    \draw[loosely dotted] (34.east) -- (45);
    \draw[loosely dotted] (56.east) -- (67);
    \draw[loosely dotted] (78.east) -- (89);    
    
    \node[mod] (13) at (1.4,1.4) {$(1,3)$};
    \node[mod] (24) at (2.8,1.4) {$(2,4)$};
    \node[mod] (35) at (4.2,1.4) {$(3,5)$};
    \node[mod] (57) at (7,1.4) {$(5,7)$};
    \node[mod] (79) at (9.8,1.4) {$(7,9)$};
    \node[mod] (810) at (11.2,1.4) {$(1,3)$};

    \draw[loosely dotted] (24.east) -- (35);    
    
    \draw[-{Stealth[scale=0.5]}] (12) -- (13);
    \draw[-{Stealth[scale=0.5]}] (23) -- (24);
    \draw[-{Stealth[scale=0.5]}] (34) -- (35);
    \draw[-{Stealth[scale=0.5]}] (56) -- (57);
    \draw[-{Stealth[scale=0.5]}] (78) -- (79);
    \draw[-{Stealth[scale=0.5]}] (89) -- (810);
        
    \draw[-{Stealth[scale=0.5]}] (13) -- (23);
    \draw[-{Stealth[scale=0.5]}] (24) -- (34);
    \draw[-{Stealth[scale=0.5]}] (35) -- (45);
    \draw[-{Stealth[scale=0.5]}] (57) -- (67);
    \draw[-{Stealth[scale=0.5]}] (79) -- (89);
    
    \node[mod] (25) at (3.5,2.1) {$(2,5)$};

    \draw[-{Stealth[scale=0.5]}] (24) -- (25);
    \draw[-{Stealth[scale=0.5]}] (25) -- (35);
    
    \end{tikzpicture}
    \]
    where $(1,1)$, $(1,2)$ and $(1,3)$ have been drawn twice. Note that by our labelling convention $(7,7) = (0,0)$, $(7,8) = (0,1)$ and $(7,9) = (0,2)$.
    Moreover, 
    \[
    \rmax 1 = 3, \quad  \rmax 2 = 5, \quad  \rmax 3 = 5, \quad  \rmax 4 = 5, \quad  \rmax 5 = 7, \quad  \rmax 6 = 7, \quad  \rmax 7 = 9
    \]
    and
    \[
    \lmax 1 = 0, \quad  \lmax 2 = 0, \quad  \lmax 3 = 1, \quad  \lmax 4 = 2, \quad  \lmax 5 = 2, \quad  \lmax 6 = 5, \quad  \lmax 7 = 5.
    \]

\end{enumerate}
\end{example}

The following results gives a general description of the shape of the Auslander--Reiten quiver.

\begin{proposition}\label{prop:basic Nakayama results1}
Let $\La$ be a Nakayama algebra with Auslander--Reiten quiver $\Gamma$.
\begin{enumerate}[label=(\alph*)]
\item If $(i,j) \in \Gamma_0$, then $(i',j') \in \Gamma_0$ for all $i \le i' \le j' \le j$.
\item If $i \le i'$, then $\rmax i \le \rmax {i'}$. Similarly, if $j' \le j$, then $\lmax {j'} \le \lmax j$.
\item We have $\rmax{i+m} = \rmax{i}+m$ and $\lmax{j+m} = \lmax{j}+m$ for all $i,j \in \ZZ$.
\item If $\La$ is cyclic, then $\rmax i \ge i+1$ and $\lmax j \le j-1$ for all $i,j \in \ZZ$.
\item If $\La$ is acyclic, then $\rmax m = m$ and $i+1 \le \rmax i \le m$ for all $1 \le i \le m-1$. Similarly $\lmax 1 = 1 $ and $1 \le \lmax j \le j-1$ for all $2 \le j \le m$.
\end{enumerate}
\end{proposition}
\begin{proof}
Part (a) follows from the fact that $M(i',j')$ is a subquotient of $M(i,j)$ if $i \le i' \le j' \le j$. Part (b) follows from (a). Part (c) follows from $M(i,j) = M(i+m,j+m)$. Part (d) follows from the admissibility of the ideal $I$ in $\La = \K \tilde{A}_m/I$. Part (e) similarly follows from admissibility and the fact that $\alpha_1M(i,j) = 0$ for all $M(i,j) \in \mo{\La}$ if $\La$ is acyclic.
\end{proof}

Note that as a consequence of Proposition~\ref{prop:basic Nakayama results1}(a), it is enough to know either the numbers $\rmax i$ or the numbers $\lmax j$ to recover the shape of the Auslander--Reiten quiver $\Gamma$.

Next we observe that $\Gamma$ has a particularly nice shape in case the Nakayama algebra is homogeneous.

\begin{proposition}\label{prop:homogeneous Nakayama AR-quiver}
\begin{enumerate}[label=(\alph*)]
    \item Let $\La=\K A_m/I$ be an acyclic Nakayama algebra. Then $I=R_{A_m}^l$ if and only if $\rmax i = \min\{i+l-1,m\}$ for all $1 \le i \le m$, if and only if $\lmax j = \max\{j-l+1,1\}$ for all $1 \le j \le m$.
    \item Let $\La=\K\tilde{A}_m/I$ be a cyclic Nakayama algebra. Then $I=R_{\tilde{A}_m}^l$ if and only if $\rmax i = i+l-1$ for all $i \in \ZZ$, if and only if $\lmax j = j-l+1$ for all $j \in \ZZ$.
\end{enumerate}
\end{proposition}
\begin{proof}
The claims follow from the fact that $R_{\tilde{A}_m}^l M(i,j) = 0$ if and only if $j-i+1 \le l$.
\end{proof}

The following proposition shows how a number of computations can be done simply by considering the shape of the Auslander--Reiten quiver.

\begin{proposition}\label{prop:basic Nakayama results2} 
Let $\La$ be a Nakayama algebra and $M(i,j) \in \mo\La$.
\begin{enumerate}[label=(\alph*)]
\item We have $\topp(M(i,j)) = M(j,j)$ and $\soc(M(i,j)) = M(i,i)$. In particular, $M(i,j)$ is simple if and only if $i = j$.
\item The projective cover of $M(i,j)$ is $M(\lmax j,j)$. In particular, $M(i,j)$ is projective if and only if $i = \lmax j$. Otherwise we have $\om(M(i,j))\isom M(\lmax j,i-1)$.
\item The injective hull of $M(i,j)$ is $M(i,\rmax i)$. In particular, $M(i,j)$ is injective if and only if $j = \rmax i$. Otherwise we have $\om^-(M(i,j))\isom M(j+1,\rmax i)$.
\end{enumerate}
\end{proposition}

\begin{proof}
For part (a) note that $M(i,j)$ is uniserial and has $M(j,j)$ as a quotient module of dimension $1$. Hence $\topp(M(i,j)) = M(j,j)$. Similarly $\soc(M(i,j)) = M(i,i)$. 

For part (b) one may compute that $M(\lmax j,j) \isom \La \epsilon_j$, which shows that $M(\lmax j,j)$ is projective and indecomposable. Moreover there is an epimorphism $M(\lmax j,j) \to M(i,j)$ whose kernel is $M(\lmax j,i-1)$ if $i > \lmax j$. The claim follows. 

Part (c) follows similarly to (b).
\end{proof}
Note that by Proposition~\ref{prop:basic Nakayama results2}, the sequences $(j-\lmax j +1)_j$ and $(\rmax i - i+1)_i$ are just the Kupisch series of $\La$ and $\La^{\rm op}$ originally introduced in \cite{Kup}.

Finally we observe which Nakayama algebras are selfinjective.

\begin{corollary}\label{cor:SelfinjNakayama}
Let $\La$ be a Nakayama algebra. Then $\La$ is selfinjective if and only if it is cyclic and homogeneous.
\end{corollary}
\begin{proof}
If $\La$ is cyclic and homogeneous then $\La$ is selfinjective by Proposition~\ref{prop:homogeneous Nakayama AR-quiver}(b) and Proposition~\ref{prop:basic Nakayama results2}(b) and (c). If $\La$ is selfinjective, then $\La$ is clearly cyclic since otherwise $\La$ has finite global dimension. Since $\La$ is selfinjective, we have for all $i\in \ZZ$ that $\rmax{i+1}=\rmax i+1$ (otherwise either $M(i+1,\rmax i)$ is injective but not projective or $M(i+1,\rmax i+1)$ is projective but not injective). Hence for all $i\in\ZZ$ we have $\rmax i=i+\rmax 1-1$ and $\La$ is homogeneous by Proposition~\ref{prop:homogeneous Nakayama AR-quiver}(b).
\end{proof}

\begin{example}\label{ex:AR quiver homoheneous}
\begin{enumerate}
    \item[(a)] Let $\La=\K A_7/R^3$ be a homogeneous acyclic Nakayama algebra. Then the Auslander--Reiten quiver $\Gamma(\La)$ of $\La$ is
    
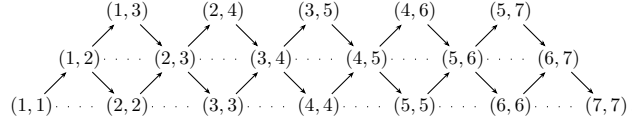
\[
    \begin{tikzpicture}[scale=0.9, transform shape, baseline={(current bounding box.center)}]
    
    \tikzstyle{mod}=[rectangle, minimum width=6pt, draw=none, inner sep=1.5pt, scale=0.8]
    
    \node[mod] (11) at (0,0) {$(1,1)$};
    \node[mod] (22) at (1.4,0) {$(2,2)$};
    \node[mod] (33) at (2.8,0) {$(3,3)$};
    \node[mod] (44) at (4.2,0) {$(4,4)$};
    \node[mod] (55) at (5.6,0) {$(5,5)$};
    \node[mod] (66) at (7,0) {$(6,6)$};
    \node[mod] (77) at (8.4,0) {$(7,7)$\nospacepunct{.}};
    
    \draw[loosely dotted] (11.east) -- (22);
    \draw[loosely dotted] (22.east) -- (33);
    \draw[loosely dotted] (33.east) -- (44);
    \draw[loosely dotted] (44.east) -- (55);
    \draw[loosely dotted] (55.east) -- (66);
    \draw[loosely dotted] (66.east) -- (77);
        
    \node[mod] (12) at (0.7,0.7) {$(1,2)$};    
    \node[mod] (23) at (2.1,0.7) {$(2,3)$};    
    \node[mod] (34) at (3.5,0.7) {$(3,4)$};    
    \node[mod] (45) at (4.9,0.7) {$(4,5)$};    
    \node[mod] (56) at (6.3,0.7) {$(5,6)$};    
    \node[mod] (67) at (7.7,0.7) {$(6,7)$};  

    \draw[-{Stealth[scale=0.5]}] (11) -- (12);
    \draw[-{Stealth[scale=0.5]}] (22) -- (23);
    \draw[-{Stealth[scale=0.5]}] (33) -- (34);
    \draw[-{Stealth[scale=0.5]}] (44) -- (45);
    \draw[-{Stealth[scale=0.5]}] (55) -- (56);
    \draw[-{Stealth[scale=0.5]}] (66) -- (67);
    
    \draw[-{Stealth[scale=0.5]}] (12) -- (22);
    \draw[-{Stealth[scale=0.5]}] (23) -- (33);
    \draw[-{Stealth[scale=0.5]}] (34) -- (44);
    \draw[-{Stealth[scale=0.5]}] (45) -- (55);
    \draw[-{Stealth[scale=0.5]}] (56) -- (66);
    \draw[-{Stealth[scale=0.5]}] (67) -- (77);

    \draw[loosely dotted] (12.east) -- (23);
    \draw[loosely dotted] (23.east) -- (34);
    \draw[loosely dotted] (34.east) -- (45);
    \draw[loosely dotted] (45.east) -- (56);
    \draw[loosely dotted] (56.east) -- (67);
    
    \node[mod] (13) at (1.4,1.4) {$(1,3)$};
    \node[mod] (24) at (2.8,1.4) {$(2,4)$};
    \node[mod] (35) at (4.2,1.4) {$(3,5)$};
    \node[mod] (46) at (5.6,1.4) {$(4,6)$};
    \node[mod] (57) at (7,1.4) {$(5,7)$};
    
    \draw[-{Stealth[scale=0.5]}] (12) -- (13);
    \draw[-{Stealth[scale=0.5]}] (23) -- (24);
    \draw[-{Stealth[scale=0.5]}] (34) -- (35);
    \draw[-{Stealth[scale=0.5]}] (45) -- (46);
    \draw[-{Stealth[scale=0.5]}] (56) -- (57);
    
    \draw[-{Stealth[scale=0.5]}] (13) -- (23);
    \draw[-{Stealth[scale=0.5]}] (24) -- (34);
    \draw[-{Stealth[scale=0.5]}] (35) -- (45);
    \draw[-{Stealth[scale=0.5]}] (46) -- (56);
    \draw[-{Stealth[scale=0.5]}] (57) -- (67);
    
    \end{tikzpicture}
    \]
    
    \item[(b)] Let $\tilde{\La}=\K \tilde{A}_6/R^3$ be a homogeneous cyclic Nakayama algebra. Then the Auslander--Reiten quiver $\Gamma(\tilde{\La})$ of $\tilde{\La}$ is
    
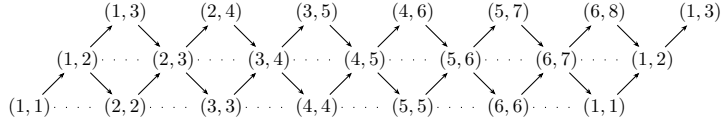
\[
    \begin{tikzpicture}[scale=0.9, transform shape, baseline={(current bounding box.center)}]
    
    \tikzstyle{mod}=[rectangle, minimum width=6pt, draw=none, inner sep=1.5pt, scale=0.8]
    
    \node[mod] (11) at (0,0) {$(1,1)$};
    \node[mod] (22) at (1.4,0) {$(2,2)$};
    \node[mod] (33) at (2.8,0) {$(3,3)$};
    \node[mod] (44) at (4.2,0) {$(4,4)$};
    \node[mod] (55) at (5.6,0) {$(5,5)$};
    \node[mod] (66) at (7,0) {$(6,6)$};
    \node[mod] (77) at (8.4,0) {$(1,1)$};
    
    \draw[loosely dotted] (11.east) -- (22);
    \draw[loosely dotted] (22.east) -- (33);
    \draw[loosely dotted] (33.east) -- (44);
    \draw[loosely dotted] (44.east) -- (55);
    \draw[loosely dotted] (55.east) -- (66);
    \draw[loosely dotted] (66.east) -- (77);
        
    \node[mod] (12) at (0.7,0.7) {$(1,2)$};    
    \node[mod] (23) at (2.1,0.7) {$(2,3)$};    
    \node[mod] (34) at (3.5,0.7) {$(3,4)$};    
    \node[mod] (45) at (4.9,0.7) {$(4,5)$};    
    \node[mod] (56) at (6.3,0.7) {$(5,6)$};    
    \node[mod] (67) at (7.7,0.7) {$(6,7)$};  
    \node[mod] (78) at (9.1,0.7) {$(1,2)$};

    \draw[-{Stealth[scale=0.5]}] (11) -- (12);
    \draw[-{Stealth[scale=0.5]}] (22) -- (23);
    \draw[-{Stealth[scale=0.5]}] (33) -- (34);
    \draw[-{Stealth[scale=0.5]}] (44) -- (45);
    \draw[-{Stealth[scale=0.5]}] (55) -- (56);
    \draw[-{Stealth[scale=0.5]}] (66) -- (67);
    \draw[-{Stealth[scale=0.5]}] (77) -- (78);
    
    \draw[-{Stealth[scale=0.5]}] (12) -- (22);
    \draw[-{Stealth[scale=0.5]}] (23) -- (33);
    \draw[-{Stealth[scale=0.5]}] (34) -- (44);
    \draw[-{Stealth[scale=0.5]}] (45) -- (55);
    \draw[-{Stealth[scale=0.5]}] (56) -- (66);
    \draw[-{Stealth[scale=0.5]}] (67) -- (77);

    \draw[loosely dotted] (12.east) -- (23);
    \draw[loosely dotted] (23.east) -- (34);
    \draw[loosely dotted] (34.east) -- (45);
    \draw[loosely dotted] (45.east) -- (56);
    \draw[loosely dotted] (56.east) -- (67);
    \draw[loosely dotted] (67.east) -- (78);
    
    \node[mod] (13) at (1.4,1.4) {$(1,3)$};
    \node[mod] (24) at (2.8,1.4) {$(2,4)$};
    \node[mod] (35) at (4.2,1.4) {$(3,5)$};
    \node[mod] (46) at (5.6,1.4) {$(4,6)$};
    \node[mod] (57) at (7,1.4) {$(5,7)$};
    \node[mod] (68) at (8.4,1.4) {$(6,8)$};
    \node[mod] (79) at (9.8,1.4) {$(1,3)$};
    
    \draw[-{Stealth[scale=0.5]}] (12) -- (13);
    \draw[-{Stealth[scale=0.5]}] (23) -- (24);
    \draw[-{Stealth[scale=0.5]}] (34) -- (35);
    \draw[-{Stealth[scale=0.5]}] (45) -- (46);
    \draw[-{Stealth[scale=0.5]}] (56) -- (57);
    \draw[-{Stealth[scale=0.5]}] (67) -- (68);
    \draw[-{Stealth[scale=0.5]}] (78) -- (79);
    
    \draw[-{Stealth[scale=0.5]}] (13) -- (23);
    \draw[-{Stealth[scale=0.5]}] (24) -- (34);
    \draw[-{Stealth[scale=0.5]}] (35) -- (45);
    \draw[-{Stealth[scale=0.5]}] (46) -- (56);
    \draw[-{Stealth[scale=0.5]}] (57) -- (67);
    \draw[-{Stealth[scale=0.5]}] (68) -- (78);
    
    \end{tikzpicture}
    \]
    where $(1,1)$, $(1,2)$ and $(1,3)$ have been drawn twice.
    Notice that $\tilde{\La}$ is selfinjective as claimed in Corollary \ref{cor:SelfinjNakayama}.
\end{enumerate}
\end{example}

\section{\texorpdfstring{$n\ZZ$}{nZ}-cluster tilting subcategories for Nakayama algebras}

\subsection{Computations} The aim of this paper is to classify all Nakayama algebras that admit an $n\ZZ$-cluster tilting subcategory for some $n$. In this section we perform some computations that will be useful to achieve this aim.

Since $n\ZZ$-cluster tilting subcategories are closed under $\om^n$, $\om^{-n}$, $\tn$ and $\tno$ it is crucial to describe the action of these functors on the Auslander--Reiten quiver. We start by computing the action of iterated syzygies and cosyzygies.

\begin{lemma}\label{lem:two in the diagonal, syzygies}
Let $\La$ be Nakayama algebra and $k\in \ZZ$. Let $i_1 \le j_1$ and $i_2 \le j_2$ be such that $M(i_1,j_1), M(i_2,j_2) \in \mo \La$ and $\om^{k}(M(i_1,j_1)), \om^{k}(M(i_2,j_2))$ are non-zero. Then there are $i_1' \le j_1'$ and $i_2' \le j_2'$ such that $\om^{k}(M(i_1,j_1)) \isom M(i'_1,j'_1)$, $\om^{k}(M(i_2,j_2)) \isom M(i'_2,j'_2)$ and 
\begin{enumerate}[label=(\alph*)]
\item if $k$ is even, the following implications hold
\[\begin{array}{cc}
i_1 = i_2 \Rightarrow i'_1 = i'_2, & i_1 \le i_2 \Rightarrow i'_1 \le i'_2,
\\
j_1 = j_2 \Rightarrow j'_1 = j'_2, & j_1 \le j_2 \Rightarrow j'_1 \le j'_2,
\end{array}\]
\item if $k$ is odd, the following implications hold
\[\begin{array}{cc}
i_1 = i_2 \Rightarrow j'_1 = j'_2, & i_1 \le i_2 \Rightarrow j'_1 \le j'_2,
\\
j_1 = j_2 \Rightarrow i'_1 = i'_2, & j_1 \le j_2 \Rightarrow i'_1 \le i'_2.
\end{array}\]
\end{enumerate}
\end{lemma}
\begin{proof}
We assume $k \ge 0$ as the case $k \le 0$ is similar. 

(a) The statement is trivial for $k = 0$ as we can choose $(i'_1,j'_1) = (i_1,j_1)$ and $(i'_2,j'_2)=(i_2,j_2)$. By induction it is enough to consider the case $k = 2$. By Proposition~\ref{prop:basic Nakayama results2}(b) we have $\om(M(i,j))\isom M(\lmax j,i-1)$ and $\om^2(M(i,j)) \isom \om M(\lmax j,i-1) \isom M(\lmax {i-1},\lmax j-1)$. Now by Proposition~\ref{prop:basic Nakayama results1}(b) the maps $i \mapsto \lmax {i-1}$ and $j \mapsto \lmax {j}-1$ are weakly increasing so the claim follows.

(b) By (a) it is enough to consider the case $k = 1$. Again by Proposition~\ref{prop:basic Nakayama results1}(b) we have $\om(M(i,j))\isom M(\lmax j,i-1)$ and since the maps $i \mapsto i-1$ and $j \mapsto \lmax {j}$ are weakly increasing the claim follows.
\end{proof}

As a corollary we get a similar result for the $n$-Auslander--Reiten translations.

\begin{corollary}\label{cor:two in the diagonal, n-AR-translations}
Let $\La$ be Nakayama algebra. Let $i_1 \le j_1$ and $i_2 \le j_2$ be such that $M(i_1,j_1), M(i_2,j_2) \in \mo \La$ and $\tn(M(i_1,j_1)), \tn(M(i_2,j_2))$ are non-zero. Then there are $i_1' \le j_1'$ and $i_2' \le j_2'$ such that $\tn(M(i_1,j_1)) \isom M(i'_1,j'_1)$, $\tn(M(i_2,j_2)) \isom M(i'_2,j'_2)$ and 
\begin{enumerate}[label=(\alph*)]
\item if $n$ is even, the following implications hold
\[\begin{array}{cc}
i_1 = i_2 \Rightarrow j'_1 = j'_2, & i_1 \le i_2 \Rightarrow j'_1 \le j'_2,
\\
j_1 = j_2 \Rightarrow i'_1 = i'_2, & j_1 \le j_2 \Rightarrow i'_1 \le i'_2,
\end{array}\]
\item if $n$ is odd, the following implications hold
\[\begin{array}{cc}
i_1 = i_2 \Rightarrow i'_1 = i'_2, & i_1 \le i_2 \Rightarrow i'_1 \le i'_2,
\\
j_1 = j_2 \Rightarrow j'_1 = j'_2, & j_1 \le j_2 \Rightarrow j'_1 \le j'_2.
\end{array}\]
\end{enumerate}
The same is true replacing $\tn$ by $\tno$ everywhere.
\end{corollary}
\begin{proof}
By Lemma~\ref{lem:two in the diagonal, syzygies} it is enough to note that $\tau M(i,j) = M(i-1,j-1)$ and $\tau^{-}M(i,j) = M(i+1,j+1)$ if $M(i,j)$ is not projective respectively not injective.
\end{proof}

Finally we give a sufficient condition for non-vanishing of $\Ext^1_\La$ that is useful to exclude possible $n\ZZ$-cluster tilting subcategories.

\begin{proposition}\label{prop:tetragon corollary}
Let $\La$ be a Nakayama algebra and $M(i,j) \in \mo \La$. Further let $i',j' \in \ZZ$ be such that $i+1 \le i' \le j+1 \le j' \le \rmax i$. Then $M(i',j')\in \mo \La$ and $\Ext_{\La}^1(M(i',j'),M(i,j)) \neq 0$.
\end{proposition}
\begin{proof}
First note that $j' \le \rmax i$ implies $M(i,j') \in \mo \La$. Since $i < i'$ we have that $M(i',j')$ is a proper quotient module of $M(i,j')$ which gives that $M(i',j')$ is a non-projective $\La$-module. Next we apply the Auslander--Reiten formula to obtain
\[
\Ext_{\La}^1(M(i',j'),M(i,j)) \isom D \overline{\Hom}_\La (M(i,j), \tau M(i',j')) = D \overline{\Hom}_\La (M(i,j), M(i'-1,j'-1)). 
\]
Now $i \le i'-1$ and $j \le j'-1$ tell us that there is a non-zero morphism $M(i,j) \to M(i'-1,j'-1)$ obtained as the composition of the quotient $M(i,j) \to  M(i'-1,j)$ and the inclusion $M(i'-1,j) \to M(i'-1,j'-1)$. Assume towards a contradiction that this morphism factors through an injective $\La$-module. Then it must factor through the inclusion of $M(i,j)$ in its injective hull, which is $M(i,\rmax i)$ by Proposition~\ref{prop:basic Nakayama results2}(c). This contradicts $j' -1 < \rmax i$ since $M(i,j) = R^{\rmax i-j} M(i,\rmax i)$, but the image of $M(i,j)$ in $M(i'-1,j'-1)$ is $M(i'-1,j) = R^{j' -1-j}M(i'-1,j'-1)$, which is strictly larger than $R^{\rmax i-j}M(i'-1,j'-1)$.
\end{proof}

\subsection{Necessary conditions} In this section we introduce some necessary conditions for the existence of an $n\ZZ$-cluster tilting subcategory for a Nakayama algebra. 

First we observe the following rather strong condition that holds assuming just the existence of an $n$-cluster tilting subcategory.

\begin{lemma}\label{lem:no peaks in AR-quiver}
Let $\La$ be a Nakayama algebra that admits an $n$-cluster tilting subcategory. If $M(i,j) \in \mo \La$, then $M(i-1,j-1) \in \mo \La$ or $M(i+1,j+1) \in \mo \La$.
\end{lemma}

\begin{proof}
If $i = j$, then the claim follows from Proposition~\ref{prop:basic Nakayama results1}(d),(e). Assume towards a contradiction that $i < j$, $M(i-1,j-1) \not \in \mo \La$ and $M(i+1,j+1) \not\in \mo \La$. By Proposition~\ref{prop:basic Nakayama results1}(a) $M(i+1,j) \in \mo \La$ and $M(i,j-1) \in \mo \La$. Hence $\rmax{i+1} = j$ and $\lmax{j-1} = i$. By Proposition~\ref{prop:basic Nakayama results2}(b)(c) we get that $M(i+1,j)$ is injective and $M(i,j-1)$ is projective, and so both belong to the $n$-cluster tilting subcategory. This implies $\Ext_\La^1(M(i+1,j),M(i,j-1)) = 0$ contradicting $\tau M(i+1,j) \isom M(i,j-1)$.
\end{proof}

In the rest of this section we assume that $\La$ is a Nakayama algebra and $\cC\subseteq \mo \La$ is an $n\ZZ$-cluster tilting subcategory.

\begin{lemma}\label{lem:n-th syzygy and n-th cosyzygy are indecomposable}
Let $\La$ be a Nakayama algebra and let $\cC\subseteq \mo \La$ be an $n$-cluster tilting subcategory. 
\begin{enumerate}[label=(\alph*)]
    \item $\Omega^i M$ is indecomposable for all $M\in\cC_{P}$ and $0\leq i\leq n$.
    \item $\Omega^{-i} N$ is indecomposable for all $N\in\cC_{I}$ and $0\leq i \leq n$.
\end{enumerate}
\end{lemma}

\begin{proof}
We only prove (a) as (b) is similar. For $i=0$ the result is clear by definition and for $0<i<n$ the result follows by Proposition~\ref{prop:basic nZ-cluster tilting results}(c). Since $\La$ is a Nakayama algebra and $\Omega^{n-1}M$ is indecomposable, it follows that $\Omega^{n}M$ is indecomposable or zero. Hence it is enough to show that $\Omega^{n-1}M$ is not projective. Assume towards a contradiction that $\Omega^{n-1}M$ is projective. Then $\Ext^{n-1}_{\La}(M,\Omega^{n-1}M)\neq 0$, contradicting that $\cC$ is $n$-cluster tilting. 
\end{proof}

\begin{lemma}\label{lem:necessary modules for projective non-injective}
Let $\La$ be a Nakayama algebra and let $\cC\subseteq \mo \La$ be an $n\ZZ$-cluster tilting subcategory. 
\begin{enumerate}[label=(\alph*)]
\item Assume for some $i<j$ that $M(i,j+1)\in\mo{\La}$ and $M(i-1,j)\not\in\mo\La$. Then
\[
\{M(i,j),M(i,j+1),M(i,i),M(i,j-1)\} \subseteq \cC.
\]
If moreover $M(i-1,j-1)\in\mo\La$, then $M(i-1,j-1)\in\cC$.
\item Assume for some $i<j$ that $M(i-1,j)\in\mo\La$ and $M(i,j+1)\not\in\mo\La$. Then
\[
\{M(i,j),M(i-1,j),M(j,j),M(i+1,j)\} \subseteq \cC.
\]
If moreover $M(i+1,j+1)\in\mo\La$, then $M(i+1,j+1)\in\cC$.
\end{enumerate}
\end{lemma}

\begin{proof}
We only prove (a) as (b) is similar. We will use Proposition~\ref{prop:basic Nakayama results2}(b)(c) repeatedly to identify projective and injective modules as well as compute syzygies and cosyzygies. Recall also that the Auslander--Reiten translation can be computed using Proposition~\ref{prop:AR quiver of Nakayama}.

By Proposition~\ref{prop:basic Nakayama results1}(a) $M(i,j+1) \in \mo \La$ implies $M(i,j) \in \mo \La$. Also, by the same proposition $M(i-1,j) \not \in \mo \La$ implies $M(i-1,j+1) \not \in \mo \La$. Hence $M(i,j+1)$ is projective and $M(i,j)$ is projective non-injective. In particular, \[\{M(i,j),M(i,j+1)\} \subseteq \cC.\]
Since $M(i,j)$ is not injective, we have $\om\tau^{-}\left(M(i,j)\right)\in\cC$ by Proposition~\ref{prop:basic nZ-cluster tilting results2}(e). We compute that
\[
\cC \ni \om\tau^{-}(M(i,j)) \isom \om(M(i+1,j+1)) \isom M(\lmax{j+1},i+1-1) = M(i,i)
\]
since $M(i,j+1)$ is projective.

It remains to show $M(i,j-1) \in \cC$. If $M(i-1,j-1)\not\in\mo\La$, then $M(i,j-1)$ is projective and so $M(i,j-1)\in\cC$. Otherwise, assume $M(i-1,j-1)\in\mo\La$. This implies $M(i-1,i) \in \mo \La$ by Proposition~\ref{prop:basic Nakayama results1}(a) and so $\lmax i \le i-1$. In particular $M(i,i)$ is not projective. Thus Proposition~\ref{prop:basic nZ-cluster tilting results2}(d) gives $\om^-\tau M(i,i) \in \cC$. Again we compute that
\[
\cC \ni \om^-\tau M(i,i) \isom \om^- M(i-1,i-1) \isom M(i-1+1,\rmax{i-1}) = M(i,j-1)
\]
since $M(i-1,j)\not\in\mo\La$ implies that $M(i-1,j-1)$ is injective. Hence in this case it also follows that $M(i-1,j-1)\in\cC$.
\end{proof}

The strength of Lemma~\ref{lem:necessary modules for projective non-injective}(a) is that for each indecomposable projective non-injective module $M(i,j)$, we get two more indecomposable modules in $\cC$, namely $M(i,i)$ and $M(i,j-1)$. Similarly Lemma~\ref{lem:necessary modules for projective non-injective}(b) can be used for any indecomposable injective non-projective module. Later we will show that this drastically reduces the possible $n\ZZ$-cluster tilting subcategories of $\mo \La$ and we only need to consider some special cases.

\begin{lemma}\label{lem:two connected in the diagonal implies projective or injective}
Let $\La$ be a Nakayama algebra, $\cC\subseteq \mo \La$ an $n\ZZ$-cluster tilting subcategory and $M(i,j) \in \cC$. 
\begin{enumerate}[label=(\alph*)]
\item If $M(i,j+1)\in \cC$, then $M(i,j+1)$ is projective.
\item If $M(i-1,j)\in \cC$, then $M(i-1,j)$ is injective.
\end{enumerate}
\end{lemma}

\begin{proof}
We only prove (a) as (b) is similar. Assume towards a contradiction that $M(i,j+1)$ is not projective. By Proposition~\ref{prop:basic nZ-cluster tilting results2}(d) we have that $\om^{-}\tau(M(i,j+1))\in\cC$. Using Proposition~\ref{prop:basic Nakayama results2}(c) we compute
\[
\om^{-}\tau(M(i,j+1)) \isom \om^{-}(M(i-1,j)) \isom  M(j+1,\rmax{i-1}).
\]
We claim that $\rmax{i-1} = j+1$. Indeed, if $\rmax{i-1} > j+1$, then Proposition~\ref{prop:tetragon corollary} gives $$\Ext_{\La}^1(M(j+1,\rmax{i-1}),M(i,j+1)) \neq 0$$ as $\rmax{i-1} \le \rmax{i}$ by Proposition~\ref{prop:basic Nakayama results1}(b), which contradicts that $\cC$ is $n$-cluster tilting.

Hence $M(j+1,j+1) \in \cC$. But $\rmax i \ge j+1$ so Proposition~\ref{prop:tetragon corollary} gives $\Ext_{\La}^1(M(j+1,j+1),M(i,j))\neq 0$ contradicting that $\cC$ is $n$-cluster tilting.
\end{proof}

\begin{lemma}\label{lem:connected pairs stay connected}
Let $\La$ be a Nakayama algebra, $\cC\subseteq \mo \La$ an $n\ZZ$-cluster tilting subcategory and $M(i,j) \in \cC$. 
\begin{enumerate}[label=(\alph*)]
\item Assume $M(i,j)$ is not projective and $\tn(M(i,j)) \isom M(i',j')$. If $M(i+1,j) \in \cC$, then
\[
\tn(M(i+1,j)) \isom
\begin{cases} 
M(i',j'+1) &\mbox{if $n$ is even},\\ 
M(i'+1,j')&\mbox{if $n$ is odd.}
\end{cases}\]
\item Assume $M(i,j)$ is not injective and $\tno(M(i,j)) \isom M(i',j')$. If $M(i,j-1) \in \cC$, then
\[
\tno(M(i,j-1)) \isom
\begin{cases} 
M(i'-1,j') &\mbox{if $n$ is even},\\ 
M(i',j'-1)&\mbox{if $n$ is odd.}
\end{cases}\]
\end{enumerate}
\end{lemma}

\begin{proof}
We only prove (a) as (b) is similar. Since $M(i,j) \in \mo \La$ we get that $M(i+1,j)$ is not projective by Proposition~\ref{prop:basic Nakayama results2}(b). Now Proposition~\ref{prop:basic nZ-cluster tilting results}(c) implies that $\om^k M(i,j) = M(i_k,j_k)$ and $\om^k M(i+1,j) = M(i_k',j_k')$ for all $0 \le k \le n-1$ and some $i_k \le j_k$, $i_k' \le j_k'$. We show by induction on $k$, that
\[
M(i_k',j_k') = \begin{cases} 
M(i_k+1,j_k) &\mbox{if $k$ is even},\\ 
M(i_k,j_k+1) &\mbox{if $k$ is odd.}
\end{cases}
\]
Then (a) follows from the case $k = n-1$.

Note that the case $k=0$ is trivial. Now assume the statement holds for some $0 \le k < n-1$. Proposition~\ref{prop:basic Nakayama results2}(b) gives $M(i_{k+1},j_{k+1}) = M(\lmax{j_k}, i_k-1)$ and $M(i_{k+1}',j_{k+1}') = M(\lmax{j_k'}, i_k'-1)$. Hence by induction hypothesis
\[
M(i_{k+1}',j_{k+1}') = 
\begin{cases} 
M(\lmax{j_k},i_k) &\mbox{if $k$ is even},\\ 
M(\lmax{j_k+1},i_k-1) &\mbox{if $k$ is odd.}
\end{cases}
\]
On the other hand we need to show
\[\begin{split}
M(i_{k+1}',j_{k+1}')
&= \begin{cases} 
M(i_{k+1}+1,j_{k+1}) &\mbox{if $k+1$ is even},\\ 
M(i_{k+1},j_{k+1}+1) &\mbox{if $k+1$ is odd},
\end{cases}\\
&=
\begin{cases} 
M(\lmax{j_k}+1,i_k-1) &\mbox{if $k+1$ is even},\\ 
M(\lmax{j_k},i_k) &\mbox{if $k+1$ is odd.}
\end{cases}
\end{split}\]
Hence it suffices to show that $\lmax{j_k+1} = \lmax{j_k}+1$ if $k$ is odd.

First assume towards a contradiction that $\lmax{j_k+1} < \lmax{j_k}+1$. Then $\lmax{j_k+1} = \lmax{j_k}$ by  Proposition~\ref{prop:basic Nakayama results1}(b). Hence $M(i_{k+1}',j_{k+1}') = M(\lmax{j_k+1},i_{k}-1) = M(i_{k+1},j_{k+1})$. But then
\[
\tn M(i,j)= \tau \om^{n-k-2}M(i_{k+1},j_{k+1}) = \tau \om^{n-k-2}M(i_{k+1}',j_{k+1}')= \tn M(i+1,j)
\]
contradicting Proposition~\ref{prop:basic nZ-cluster tilting results}(b).

Next assume towards a contradiction that $\lmax{j_k+1} > \lmax{j_k}+1$. Then 
\[
\{M(\lmax{j_k+1},j_k+1),M(\lmax{j_k+1}-2,j_k)\} \subseteq \mo \La
\quad\mbox{and}\quad
M(\lmax{j_k+1}-1,j_k+1) \not \in \mo \La,
\]
so by Lemma~\ref{lem:necessary modules for projective non-injective}(b) we get $M(\lmax{j_k+1},j_k) \in \cC$. Now recall $M(i_k',j_k') = M(i_k,j_k+1)$, which is not projective implying that $\lmax{j_k+1} < i_k$. This allows us to apply Proposition~\ref{prop:tetragon corollary} to obtain $$\Ext_{\La}^1(M(i_k,j_k+1),M(\lmax{j_k+1},j_k))\neq 0$$ (note that $j_k+1 \le \rmax{\lmax{j_k+1}}$ holds since $t \le \rmax{\lmax{t}}$ is true for all $t \in \ZZ$).
But then
\[\begin{split}
\Ext_{\La}^{k+1}(M(i+1,j),M(\lmax{j_k+1},j_k)) &\isom
\Ext_{\La}^{1}(\om^kM(i+1,j),M(\lmax{j_k+1},j_k)) \isom \\
\Ext_{\La}^{1}(M(i_k',j_k'),M(\lmax{j_k+1},j_k)) &\isom
\Ext_{\La}^{1}(M(i_k,j_k+1),M(\lmax{j_k+1},j_k)) \neq 0,
\end{split}\]
which contradicts that $\cC$ is $n\ZZ$-cluster tilting.
\end{proof}

\begin{lemma}\label{lem:tau of projective is injective}
Let $\La$ be a Nakayama algebra and $\cC\subseteq \mo \La$ an $n\ZZ$-cluster tilting subcategory. Assume for some $i < j$ that $\{M(i,j),M(i-1,j-1)\} \subseteq \mo \La$ and $M(i-1,j) \not \in \mo \La$. Then both $M(i,j)$ and $M(i-1,j-1)$ are both projective and injective unless $n$ is even and $j=i+1$.
\end{lemma}

\begin{proof}
Since $M(i-1,j) \not \in \mo \La$ we have that $M(i,j)$ is projective and $M(i-1,j-1)$ is injective. We only show that $M(i,j)$ injective. The proof that $M(i-1,j-1)$ is projective is similar. Note that it suffices to show that $M(i,j+1) \not \in \mo \La$. Thus we assume towards a contradiction that $M(i,j+1) \in \mo \La$.

Then by Lemma~\ref{lem:necessary modules for projective non-injective}(a)
\[
\{M(i,j),M(i-1,j-1),M(i,j+1),M(i,i),M(i,j-1)\} \subseteq \cC.
\]
Moreover, since $M(i,j+1) \in \mo \La$, we get $\{M(i,j),M(i,i),M(i,j-1)\}\subseteq \cC_I$. Recall from Proposition~\ref{prop:basic nZ-cluster tilting results}(b), that there are mutually inverse bijections
\[\begin{tikzpicture}
        \node (0) at (0,0) {$\cC_P$};
        \node (1) at (2,0) {$\cC_I$.};
        
        \draw[-latex] (0) to [bend left=20] node [above] {$\tn$} (1);
        \draw[-latex] (1) to [bend left=20] node [below] {$\tno$} (0);
\end{tikzpicture}\] 
We use these repeatedly in the rest of the proof. In particular note that $$\{\tno(M(i,j)),\tno(M(i,i)),\tno(M(i,j-1))\} \subseteq \cC_P.$$

Now assume that $n$ is odd. By Lemma~\ref{lem:connected pairs stay connected}(b) we get $\tno(M(i,j)) \isom M(i',j')$ and $\tno(M(i,j-1)) \isom M(i',j'-1)$ for some $i'<j'$. But then Lemma~\ref{lem:two connected in the diagonal implies projective or injective}(a) implies $M(i',j')$ is projective contradicting $M(i',j') \in \cC_P$.

Next assume that $n$ is even and $j > i+1$ so that $M(i,i) \neq M(i,j-1)$. By Lemma~\ref{lem:connected pairs stay connected}(b) and Corollary~\ref{cor:two in the diagonal, n-AR-translations}(a) we get
\[\begin{split}
&\tno(M(i,j)) \isom M(i',j')\\
&\tno(M(i,j-1)) \isom M(i'-1,j') \\
&\tno(M(i,i)) \isom M(k,j')
\end{split}\]
for some $i', j', k$ satisfying $i' < j'$ and $k \le i'-1$. Moreover, the above modules are non-projective and distinct so in fact $k < i'-1$. In particular, $M(i'-2,j') \in \mo \La$ is not projective. Since $\{M(i',j'),M(i'-1,j')\}\subseteq \cC$ it follows that $M(i'-1,j')$ is injective by Lemma~\ref{lem:two connected in the diagonal implies projective or injective}(b), and therefore so is also $M(i'-2,j')$. In particular $M(i'-2,j') \in \cC_P$. Now $\tn(M(i'-1,j')) \isom \tn(\tno(M(i,j-1))) \isom M(i,j-1)$, so by Lemma~\ref{lem:connected pairs stay connected}(a) we have that $\tn(M(i'-2,j')) \isom M(i,j-2) \in \cC$.

Hence we have $M(i,j-2), M(i,j-1) \in \cC$ and $M(i-1,j-1)\in\mo\La$. By Lemma \ref{lem:two connected in the diagonal implies projective or injective}(a) we have that $M(i,j-1)$ is projective. But this contradicts $M(i-1,j-1)\in\mo\La$.
\end{proof}

Lemma~\ref{lem:tau of projective is injective} provides strong restrictions on which $\La$ admit an $n\ZZ$-cluster tilting subcategory. For instance we can now show the following.

\begin{corollary}\label{cor:n odd implies homogeneous Nakayama}
Let $\La$ be a Nakayama algebra and let $\cC\subseteq \mo \La$ be an $n\ZZ$-cluster tilting subcategory. If $n$ is odd, then the algebra $\La$ is a homogeneous Nakayama algebra.
\end{corollary}

\begin{proof}
Assume towards a contradiction that $\La$ is not homogeneous. Then by Proposition~\ref{prop:basic Nakayama results1} and Proposition~\ref{prop:homogeneous Nakayama AR-quiver} there exists $i<j$ such that $\{M(i-1,j-1),M(i,j)\} \subseteq \mo \La$, $M(i-1,j) \not \in \mo \La$ and such that $M(i-2,j-1)\in \mo \La$ or $M(i,j+1)\in \mo \La$. Hence $M(i-1,j-1)$ is not projective or $M(i,j)$ is not injective. In either case Lemma~\ref{lem:tau of projective is injective} is contradicted.
\end{proof}

\begin{corollary}\label{cor:not homogeneous implies ungluing simple}
Let $\La$ be a Nakayama algebra and assume $\mo{\La}$ has at least one $n\ZZ$-cluster tilting subcategory. If $\La$ is not homogeneous, then there exists a simple module $M(i,i)\in \mo{\La}$ which is neither projective nor injective, such that $M(i-1,i+1)\not\in\mo{\La}$ and such that any $n\ZZ$-cluster tilting subcategory of $\La$ contains $M(i,i)$.
\end{corollary}

\begin{proof}
Since $\La$ is not homogeneous, there exist $i,j\in \ZZ$ with $i<j$ and such that 
\[
\{M(i-1,j-1),M(i,j)\}\subseteq \mo{\La}, \; M(i-1,j)\not\in\mo{\La}
\]
and at least one of the conditions 
\[
M(i,j+1)\in\mo{\La}, \;M(i-2,j-1)\in\mo{\La}
\]
holds. Assume that $M(i,j+1)\in\mo{\La}$; the other case is similar. Then it follows by Lemma~\ref{lem:tau of projective is injective} that $j=i+1$ and so $M(i-1,i+1)=M(i-1,j)\not\in\mo{\La}$. It follows by Lemma~\ref{lem:necessary modules for projective non-injective}(a) that $M(i,i)$ belongs to any $n\ZZ$-cluster tilting subcategory of $\mo{\La}$. Finally, since $\{M(i-1,i),M(i,i+1)\}=\{M(i-1,j-1),M(i,j)\}\subseteq\mo{\La}$, it follows that $M(i,i)$ is neither projective nor injective.
\end{proof}

\section{Classification}

In this section we classify all Nakayama algebras admitting an $n\ZZ$-cluster tilting subcategory.

\subsection{Homogeneous relations} 
We begin by restricting our attention to homogeneous Nakayama algebras. Homogeneous acyclic Nakayama algebras admitting an $n$-cluster tilting subcategory were classified in \cite{Vas1}. Homogeneous cyclic (i.e. self-injective) Nakayama algebras admitting an $n$-cluster tilting subcategory were classified in \cite{DI}. We recall these classifications.

\begin{theorem}\label{thrm:homogeneous Nakayama with n-ct}
Let $\La=\K Q_m/R^l$ be a homogeneous Nakayama algebra where $Q_m\in\{A_m,\tilde{A}_m\}$ and $l\geq 2$. Then $\La$ admits an $n$-cluster tilting subcategory if and only if at least one of the following conditions is satisfied.
\begin{enumerate}[label=(\alph*)]
    \item $Q_m=A_m$, $l=2$ and $n\divides m-1$.
    \item $Q_m=A_m$, $n$ is even and $(l(n-1)+2)\divides m-1-\frac{n}{2}l$.
    \item $Q_m=\tilde{A}_m$ and $(l(n-1)+2)\divides 2m$.
    \item $Q_m=\tilde{A}_m$ and $(l(n-1)+2)\divides tm$, where $t=\gcd(n+1,2(l-1))$.
\end{enumerate}
\end{theorem}

\begin{proof}
Parts (a) and (b) are \cite[Theorem 2]{Vas1}. Parts (c) and (d) are \cite[Theorem 5.1]{DI}.
\end{proof}

Our aim is now to determine which of the above algebras $\La$ admit not only an $n$-cluster tilting subcategory but an $n\ZZ$-cluster tilting subcategory. Since $\La$ is homogeneous Proposition~\ref{prop:homogeneous Nakayama AR-quiver} gives explicit formulas for $\rmax{i}$ and $\lmax{j}$, which we can use to determine projectives and injectives, as well as compute syzygies and co-syzygies using Proposition~\ref{prop:basic Nakayama results2}. We will use these results freely in this section without further reference.

We start with the acyclic case. Note that then $\La$ is representation-directed and so there is at most one $n$-cluster tilting subcategory in $\mo \La$ by Corollary~\ref{cor:unique n-ct for representation-directed}.

\begin{proposition}\label{prop:homogeneous acyclic Nakayama with nZ-ct}
Let $\La=\K A_m/R^l$ be a homogeneous acyclic Nakayama algebra. There exists an $n\ZZ$-cluster tilting subcategory $\cC\subseteq \mo \La$ if and only if one of the following conditions is satisfied.
\begin{enumerate}[label=(\alph*)]
    \item $l=2$ and $n\divides m-1$.
    \item $l \ge 3$, $l\divides m-1$ and $n=2\frac{m-1}{l}$.
\end{enumerate}
More explicitly we have in case (a) that
\[
\cC = \add(\{\La\}\cup\{\tau_n^{-k}(M(1,1))\mid 1\leq k\leq \tfrac{m-1}{n}\})=\add(\{\La\}\cup\{M(1+kn,1+kn)\mid 1\leq k \leq \tfrac{m-1}{n}\}
\]
and in case (b) that $\cC = \add(\La\oplus D(\La))$.
\end{proposition}

\begin{proof}
Assume first that $\cC\subseteq \mo \La$ is $n\ZZ$-cluster tilting. In particular, $\cC$ is $n$-cluster tilting. 

If $l=2$, then by Theorem~\ref{thrm:homogeneous Nakayama with n-ct}(a)(b) we have that at least one of the conditions
\begin{enumerate}[label=(\roman*)]
    \item $n\divides m-1$, or
    \item $2n\divides m-1-n$
\end{enumerate} 
is satisfied. If (ii) is satisfied, then $n\divides m-1-n$ and so $n\divides m-1$. Either way condition (a) is satisfied.

Now assume that $l \ge 3$. From Theorem~\ref{thrm:homogeneous Nakayama with n-ct} it follows that $n$ is even. Since $\La$ is homogeneous we have that $\rmax{1} = l$ and $\lmax{j} = 1$ for $1 \le j \le l$. Thus $M(1,l)$ is projective and injective, while $M(1,j)$ is projective and non-injective for $1 \le j \le l-1$. In particular we have $M(1,j) \in \cC_I$ and $\tno (M(1,j)) \in \cC_P$ for $1 \le j \le l-1$ by Proposition~\ref{prop:basic nZ-cluster tilting results}(a)(b). Write $\tno(M(1,l-1)) \isom M(i',j')$ for some $i' \le j'$. By Lemma~\ref{lem:connected pairs stay connected}(b) we get that 
\[
\tno (M(1,l-t)) \isom M(i'-t+1,j') \quad \mbox{for } 1 \le t \le l-1.
\]
Next we use this equation to show $i' = j' = m$. For $t = l-1$ we obtain $\tno (M(1,1)) \isom M(i'-l+2,j')$ which is not projective. Thus $i'-l+2 > \lmax{j'} =\max\{j'-l+1,1\}$. Since $i' \le j'$ the only possibility is $i' = j'$. For $t = 2$ we obtain $\tno (M(1,l-2)) \isom M(i'-1,i')\in \cC$ which is injective by Lemma~\ref{lem:two connected in the diagonal implies projective or injective}(b). Hence $i' = \rmax{i'-1} = \min\{i'+l-2,m\}$. Since $l \ge 3$ we get $i'=m$. We now obtain
\[
\tno (M(1,l-t)) \isom M(m-t+1,m) \quad \mbox{for } 1 \le t \le l-1.
\]
On the other hand, for homogeneous acyclic Nakayama algebras there is a formula for computing $\tno$ given in \cite[Lemma 4.8(b)]{Vas1} which translated to our notation reads
\[
\tno(M(i,j)) \isom M(j+\tfrac{n-2}{2}l+2,i +\tfrac{n}{2}l).
\]
Comparing this formula to the one above gives $n=2\frac{m-1}{l}$, as required.

For the converse implication assume that one of the conditions (a) or (b) is satisfied.

If (a) is satisfied, then there exists a unique $n$-cluster tilting subcategory $\cC \subseteq \mo \La$ by Theorem~\ref{thrm:homogeneous Nakayama with n-ct}(a). We show that $\cC$ is actually $n\ZZ$-cluster tilting. Let $2 \le j \le m$. Since $l = 2$ we get $\lmax{j} = j-1$ and
\[
\om(M(j,j)) = M(\lmax j, j-1) = M(j-1, j-1) = \tau M(j,j).
\]
Note that the above formula applies to every non-projective indecomposable since $l = 2$. Hence $\om^n(M)\isom \tn(M)$ for every $M\in\cC_P$. By Proposition~\ref{prop:basic nZ-cluster tilting results}(b) and Proposition~\ref{prop:basic nZ-cluster tilting results2} (b) it follows that  $\cC$ is $n\ZZ$-cluster tilting. The explicit form follows from Corollary \ref{cor:unique n-ct for representation-directed} by noticing that $M(1,1)$ is the unique indecomposable projective non-injective $\La$-module.

Finally, assume that (b) is satisfied. By Theorem 3 in Section 5.2 in \cite{Vas1} there exists a unique $n$-cluster tilting subcategory $\cC=\add(\La\oplus D(\La)) \subseteq \mo \La$ and $\gldim(\La)=n$. In particular, $\cC$ is $n\ZZ$-cluster tilting.
\end{proof}

We now turn our attention to the cyclic homogeneous case. Thus for the rest of this section let $\La=\K \tilde{A}_m/R^l$ for some $l \ge 2$. Then $\La$ is selfinjective and $M(i,j)$ is projective (or injective) if and only if $j-i = l-1$. In other words, $\rmax i = i+l-1$ and $\lmax j = j-l+1$. Using this we get the following convenient formulas for $\om$, $\om^-$, $\tau$ and $\tau^-$.

\begin{lemma}\label{lem:homogeneous cyclic syzygy and tau}
Let $i \le j \le i+l-2$. Then
\begin{enumerate}[label=(\alph*)]
\item $\om (M(i,j)) \isom M(j-l+1,i-1)$.
\item $\om^- (M(i,j)) \isom M(j+1,i+l-1)$.
\item $\om \om^- (M(i,j)) \isom \om^- \om (M(i,j))\isom M(i,j)$.
\item $\tau \tau^- (M(i,j)) \isom \tau^- \tau (M(i,j)) \isom M(i,j)$.
\item $\om \tau^- (M(i,j)) \isom \tau^- \om (M(i,j) \isom M(j-l+2,i)$.
\item $\om^- \tau (M(i,j)) \isom \tau \om^- (M(i,j)\isom M(j,i+l-2)$.
\end{enumerate}
\end{lemma}

Our aim is to find all $n\ZZ$-cluster tilting subcategories $\cC \subseteq \mo \La$. Since  $M(i,i+l-1)$ is projective (and injective) we have $ M(i,i+l-1) \in\cC$ for all $i$. Hence to describe $\cC$ we need to determine for which $i \le j \le i+l-2$ we have 
$M(i,j) \in \cC_P=\cC_I$. As we shall see next, it turns out that only the extreme cases $i = j$ and $j = i+l-2$ (which coincide for $l = 2$) can occur. In other words $\cC_P$ will only contain simples and (co-)syzygies of simples. We also remark that $\cC_P \neq \emptyset$ as $\La$  is not semi-simple.

\begin{lemma}\label{lem:homogeneous cyclic extreme}
Let $\cC\subseteq \mo \La$ be $n\ZZ$-cluster tilting.
\begin{enumerate}[label=(\alph*)]
\item If $M(i,j) \in \cC_P$, then $j = i$ or $j = i+l-2$.
\item $M(i,i) \in \cC_P$ if and only if $M(i,i+l-2) \in \cC_P$ if and only if $M(i+l-2,i+l-2) \in \cC_P$.
\end{enumerate}
\end{lemma}
\begin{proof}
(a) Assume towards a contradiction that $i < j < i+l-2$. By Proposition~\ref{prop:basic nZ-cluster tilting results2}(d) and Proposition~\ref{lem:homogeneous cyclic syzygy and tau}(f) we have $\cC_P \ni  \om^- \tau(M(i,j)) \isom M(j,i+l-2)$. But then $\Ext^1_\La(M(j,i+l-2),M(i,j)) \neq 0$ by Proposition~\ref{prop:tetragon corollary} contradicting that $\cC$ is $n$-cluster tilting.

(b) We apply Proposition~\ref{lem:homogeneous cyclic syzygy and tau} to compute 
\[
\om^- \tau (M(i,i)) \isom M(i,i+l-2),  \quad \om^- \tau (M(i,i+l-2)) \isom M(i+l-2,i+l-2)\] 
\[\mbox{and} \quad \om\tau^{-}\om\tau^{-}(M(i+l-2,i+l-2))\isom M(i,i).
\]
The claim follows by Proposition~\ref{prop:basic nZ-cluster tilting results2}(d)(e).
\end{proof}

Note that Lemma~\ref{lem:homogeneous cyclic extreme} is independent of the number $n$. We now involve $n$ to get another constraint.

\begin{lemma}\label{lem:homogeneous cyclic is closed under n-th translation}
Let $\cC\subseteq \mo \La$ be $n\ZZ$-cluster tilting, let $k\in\ZZ$ be an integer, and assume $M(i,j)\in\cC_P$. Then $M(i+k,j+k)\in\cC_P$ if and only if $n\divides k$.
\end{lemma}

\begin{proof}
By Proposition~\ref{lem:homogeneous cyclic syzygy and tau},
\[
M(i+k,j+k) \isom \tau^{-k} (M(i,j)) \isom \om^{-k}(\om \tau^-)^k (M(i,j))
\]
where $(\om \tau^-)^k (M(i,j)) \in \cC_P$ by Proposition~\ref{prop:basic nZ-cluster tilting results2}(e).

For $k = n$ we get $M(i+n,j+n) \in \cC_P$ by Proposition~\ref{prop:basic nZ-cluster tilting results2}(c). 

For $1 \le k \le n-1$ we get
\[
\Ext^k_{\La}(M(i+k,j+k),(\om \tau^-)^k M(i,j)) \neq 0
\]
which implies $M(i+k,j+k) \not \in \cC_P$. 

For $k \ge 0$ the claim now follows by induction. The case $k \le 0$ is similar.
\end{proof}

Together Lemma~\ref{lem:homogeneous cyclic extreme} and Lemma~\ref{lem:homogeneous cyclic is closed under n-th translation} pose very strong restrictions on $l$, $m$ and $n$.

\begin{corollary}\label{cor:homogeneous cyclic implies division}
If $\mo \La$ admits an $n\ZZ$-cluster tilting subcategory then $n \divides m$ and $n \divides l-2$.
\end{corollary}
\begin{proof}
Let $\cC \subseteq \mo \La$ be $n\ZZ$-cluster tilting. Since $\cC_P \neq \emptyset$ there are $i \le j \le i+l-2$ such that $M(i,j) \in \cC_P$. Then $M(i+m,j+m) = M(i,j) \in \cC_p$ together with Lemma~\ref{lem:homogeneous cyclic is closed under n-th translation} implies $n \divides m$. 

By Lemma~\ref{lem:homogeneous cyclic extreme} we have that $j = i$ or $j = i+l-2$ and regardless we get $\{M(i,i), M(i+l-2,i+l-2)\} \subseteq \cC_P$. Again Lemma~\ref{lem:homogeneous cyclic is closed under n-th translation} implies $n \divides l-2$.
\end{proof}

Note that the condition $n \divides l-2$ is trivially satisfied for $l=2$ and implies $l \neq 3$. It turns out that for $l \ge 4$ we must have $n = l-2$. To show this we need to get better control over $\Ext$-groups involving $M(i,i)$ and $M(i,i+l-2)$.

\begin{lemma}\label{lem:homogeneous cyclic nonzero extensions}
Let $i \in \ZZ$ and $X \in \mo \La$ be indecomposable.
\begin{enumerate}[label=(\alph*)]
\item $\Ext^1_{\La}(X,M(i,i)) \neq 0$ if and only if $X \isom M(i+1,i+t)$ for some $1 \le t \le l-1$.
\item $\Ext^1_{\La}(X,M(i,i+l-2)) \neq 0$ if and only if $X \isom M(i+t,i+l-1)$ for some $1 \le t \le l-1$.
\item $\Ext^1_{\La}(M(i,i),X) \neq 0$ if and only if $X \isom M(i-t,i-1)$ for some $1 \le t \le l-1$.
\item $\Ext^1_{\La}(M(i,i+l-2),X) \neq 0$ if and only if $X \isom M(i-1,i+l-2-t)$ for some $1 \le t \le l-1$.
\end{enumerate}
\end{lemma}
\begin{proof}
We only prove (a) and (b) as (c) and (d) are similar.

(a) By the Auslander--Reiten formula
\[
\Ext^1_{\La}(X,M(i,i)) \isom D\underline{\Hom}_{\La}(\tau^-(M(i,i)),X) \isom D\underline{\Hom}_{\La}(M(i+1,i+1),X).
\]
By Proposition~\ref{prop:basic Nakayama results2}(a), $M(i+1,i+1)$ is simple. Moreover, $X$ is indecomposable and hence uniserial. It follows that $D\underline{\Hom}_{\La}(M(i+1,i+1),X) \neq 0$ if and only if $\soc X \isom M(i+1,i+1)$ and $X$ is not projective. By Proposition~\ref{prop:basic Nakayama results2} this is equivalent to $X \isom M(i+1,j)$ for some $i+1 \le j \le i+l-1$. The claim follows.

(b) By Proposition~\ref{lem:homogeneous cyclic syzygy and tau}(b) we have that $\om^{-} M(i-1,i-1) \isom M(i,i+l-2)$. Hence 
\[
\Ext^1_{\La}(X,M(i,i+l-2)) \isom \Ext^1_{\La}(X,\om^{-} M(i-1,i-1)) \isom \Ext^{1}_{\La}(\om X,M(i-1,i-1))
\]
and the claim follows from (a) and Proposition~\ref{lem:homogeneous cyclic syzygy and tau}.
\end{proof}

\begin{lemma}\label{lem:homogeneous cyclic implies n=l-2}
Assume that $l \ge 4$ and $\cC\subseteq \mo{\La}$ is an $n\ZZ$-cluster tilting subcategory. Then $n=l-2$.
\end{lemma}

\begin{proof}
Since $\cC_P \neq \emptyset$, Lemma~\ref{lem:homogeneous cyclic extreme} gives $M(i,i) \in \cC_P$ for some $1 \le i \le m$ and by Lemma~\ref{lem:homogeneous cyclic is closed under n-th translation} we have that $M(i+n,i+n) \in \cC_P$ as well. By Corollary~\ref{cor:homogeneous cyclic implies division} we have that $n \divides l-2$ and so $n \le l-2$. Assume towards a contradiction that $n < l-2$. Then $M(i,i+n) \not \in \cC_P$ by Lemma~\ref{lem:homogeneous cyclic extreme} and so there is $C \in \cC_P$ such that $0 \neq \Ext^k(M(i,i+n),C) \isom \Ext^1(M(i,i+n),\om^{-(k-1)}C)$ for some $1 \le k \le n-1$. By Lemma~\ref{lem:homogeneous cyclic extreme} it follows that $C$ is simple or a syzygy of a simple and by Lemma~\ref{lem:homogeneous cyclic syzygy and tau}, so is $\om^{-(k-1)}C$. 

Assume first that $\om^{-(k-1)}C$ is simple. Then there is $0 \le i' \le m-1$ such that $\om^{-(k-1)}C \isom M(i',i')$. By Lemma~\ref{lem:homogeneous cyclic nonzero extensions}(a) we get that $i'+1=i$ and so $\om^{-(k-1)}C \isom M(i-1,i-1)$. But then 
\[
\Ext_\La^k(M(i,i),C) \isom \Ext_\La^1(M(i,i),\om^{-(k-1)}C) \isom \Ext_\La^1(M(i,i),M(i-1,i-1)) \neq 0
\]
(again by Lemma~\ref{lem:homogeneous cyclic nonzero extensions}(a)), which contradicts that $\cC$ is $n\ZZ$-cluster tilting.

Secondly assume that $\om^{-(k-1)}C$ is a syzygy of a simple. Then we can use Lemma~\ref{lem:homogeneous cyclic nonzero extensions}(b) to show in a similar way that $\om^{-(k-1)}C \isom M(i+n-l+1,i+n-1)$ and $\Ext_\La^k(M(i+n,i+n),C) \isom \Ext_\La^1(M(i+n,i+n),M(i+n-l+1,i+n-1) \neq 0$. Again this contradicts that $\cC$ is $n\ZZ$-cluster tilting.
\end{proof}

Now we are ready to deal with homogeneous cyclic Nakayama algebras.

\begin{proposition}\label{prop:homogeneous cyclic Nakayama with nZ-ct}
Let $\La=\K \tilde{A}_m/R^l$ be a homogeneous cyclic Nakayama algebra. There exists an $n\ZZ$-cluster tilting subcategory $\cC\subseteq \mo{\La}$ if and only if one of the following conditions is satisfied.
\begin{enumerate}[label=(\alph*)]
    \item $l=2$ and $n\divides m$.
    \item $l \ge 4$, $n=l-2$ and $n\divides m$.
\end{enumerate}
Then there is a unique $1 \le i \le n$ such that $M(i,i) \in \cC$. More explicitly we have in case (a) that
\[
\cC = \add(\{\La\} \cup \{\tau_n^{k} \left(M(i,i)\right) \mid k \in \ZZ\}) = \add(\{\La\} \cup \{M(i+kn,i+kn)\mid k \in \ZZ\})
\]
and in case (b) that
\[\begin{split}
\cC &= \add(\{\La\} \cup \{(\om^{-}\tau)^k (M(i,i)) \mid k \in \ZZ\}) \\
&= \add(\{\La\} \cup \{M(i+kn,i+kn)\oplus M(i+kn,i+(k+1)n) \mid k \in \ZZ\}).
\end{split}\]
\end{proposition}

\begin{proof}
First note that if there exists an $n\ZZ$-cluster tilting subcategory $\cC \subseteq\mo \La$, then (a) or (b) hold by Corollary~\ref{cor:homogeneous cyclic implies division} and Lemma~\ref{lem:homogeneous cyclic implies n=l-2}. Moreover, Lemma~\ref{lem:homogeneous cyclic extreme} and Lemma~\ref{lem:homogeneous cyclic is closed under n-th translation} imply that there is unique $1 \le i \le n$ such that $M(i,i) \in \cC$. Thus it remains to show that if (a) or (b) holds, then there exists an $n\ZZ$-cluster tilting subcategory $\cC \subseteq\mo \La$, and it has the desired form.

Assume that condition (a) is satisfied. Then
\[l(n-1)+2 = 2n \divides 2m\]
and hence there exists an $n$-cluster tilting subcategory $\cC \subseteq\mo \La$ by Theorem~\ref{thrm:homogeneous Nakayama with n-ct}. Note that since $l = 2$ every indecomposable $\La$-module is either simple or projective-injective. Let $M(i,i) \in \cC_P$. Using Proposition~\ref{lem:homogeneous cyclic syzygy and tau} we compute
\[
\tn M(i,i) = \tau \om^{n-1}M(i,i) \isom \tau M(i-n+1,i-n+1) \isom M(i-n,i-n) \isom \om^{n}M(i,i)
\]
Hence $\cC$ is $n\ZZ$-cluster tilting by Proposition~\ref{prop:basic nZ-cluster tilting results}(b) and Proposition~\ref{prop:basic nZ-cluster tilting results2} (b). The second part of the claim also follows.

Next assume that condition (b) is satisfied. To simplify notation we define
\[
F(M(i,j)) := M(j,i+n) = M(j,i+l-2)
\]
for all $i \le j \le i+n$ so that $F(M(i,j)) \isom \om^- \tau (M(i,j))$ by Lemma~\ref{lem:homogeneous cyclic syzygy and tau}.

Let $1 \le i \le n$ and set
\[
\cC = \add(\{\La\} \cup (\om^{-}\tau)^k (M(i,i)) \mid k \in \ZZ\})
\]
Using the notation just introduced we get 
\[\begin{split}
\cC &= \add(\{\La\} \cup \{F^k(M(i,i)) \mid k \in \ZZ\}) \\
&= \add(\{\La\} \cup \{M(i+kn,i+kn)\oplus M(i+kn,i+(k+1)n) \mid k \in \ZZ\}).
\end{split}\]
In particular, we note that $F^{\tfrac{2m}{n}}(M(i,i)) = M(i+m,i+m) = M(i,i)$. Hence $\cC_P$ has exactly $\tfrac{2m}{n}$ elements, half of which are simple.

To complete the proof we show that $\cC$ is $n$-cluster tilting. Then it follows that $\cC$ is $n\ZZ$-cluster tilting as for $M(i,j) \in \cC_P$ we have by Lemma~\ref{lem:homogeneous cyclic syzygy and tau} that
\[
\om^{n}(M(i,j)) = \tau \om^{n-1}(\tau \om^{-})^{-1}(M(i,j)) = \tn F^{-1}(M(i,j)) \in \cC_P.
\]
Moreover, any $n\ZZ$-cluster tilting subcategory containing $M(i,i)$ must contain $\cC$ (and thus be equal to $\cC$) by Proposition~\ref{prop:basic nZ-cluster tilting results2}(d) and (e). 

We show 
\[
\cC = \{X\in \mo\La \mid \Ext^{s}_{\La}(X,\cC) = 0 \mbox{ for } 1 \le s \le n-1\}.
\]
The equality 
\[
\cC = \{X\in \mo\La \mid \Ext^{s}_{\La}(\cC,X) = 0 \mbox{ for } 1 \le s \le n-1\}
\] 
is similar. Set 
\[
\mathcal{X} = \{M(i+s,i+s+t) \mid 1 \le s \le n-1, 0 \le t \le n\}.
\]
and let $X = M(i',j')$ for some $1 \le i' \le j' \le i'+n$. It is straightforward to show that $X \in F^k\mathcal{X}$ for some $k \in \ZZ$ if and only if $X \not \in \cC_P$. On the other hand, we claim that $X \in \mathcal{X}$ if and only if $\Ext^s_{\La}(X,F^{-s+1}(M(i,i)) \neq 0$ for some $1 \le s \le n-1$. Indeed, this follows from the calculation
\[
\Ext^s_{\La}(X,F^{-s+1}(M(i,i))) \isom \Ext^1_{\La}(\om^{s-1}(X),\om^{s-1} \tau^{-s+1}(M(i,i))) \isom \Ext^1_{\La}(X,M(i+s-1,i+s-1)) 
\]
and Lemma~\ref{lem:homogeneous cyclic nonzero extensions}(a). Hence $X \in F^k\mathcal{X}$ for some $k \in \ZZ$ if and only if $\Ext^s_{\La}(X,F^{k-s+1}(M(i,i))) \neq 0$ for some $k \in \ZZ$ and $1 \le s \le n-1$, which is equivalent to 
\[
X \not \in \{Y\in \mo\La \mid \Ext^{s}_{\La}(Y,\cC) = 0 \mbox{ for } 1 \le s \le n-1\}.
\]
The claim follows.
\end{proof}

\begin{remark}\label{rmk:homogeneous cyclic Nakayama with nZ-ct}
Note that in Proposition~\ref{prop:homogeneous cyclic Nakayama with nZ-ct} any $n\ZZ$-cluster tilting subcategory $\cC \subseteq \mo \La$ is determined by $1 \le i \le n$ such that $M(i,i) \in \cC$. By symmetry any such $i$ can appear. Thus for a homogeneous cyclic Nakayama algebras there are precisely $n$ distinct $n\ZZ$-cluster tilting subcategories if any at all.
\end{remark}

We now combine our results to achieve the classification of all homogeneous Nakayama algebras admitting $n\ZZ$-cluster tilting subcategories.

\begin{theorem}\label{thrm:homogeneous Nakayama with nZ-ct}
Let $\La=\K Q_m/R^l$ be a homogeneous Nakayama algebra where $Q_m\in\{A_m,\tilde{A}_m\}$ and $l\geq 2$. Then $\La$ admits an $n\ZZ$-cluster tilting subcategory if and only if one of the following conditions is satisfied.
\begin{enumerate}[label=(\alph*)]
    \item $Q_m=A_m$, $l=2$ and $n\divides m-1$.
    \item $Q_m=A_m$, $l \ge 3$, $l\divides m-1$ and $n=2\frac{m-1}{l}$.
    \item $Q_m=\tilde{A}_m$, $l=2$ and $n\divides m$.
    \item $Q_m=\tilde{A}_m$, $l \ge 4$, $n=l-2$ and $n\divides m$.
\end{enumerate}
\end{theorem}

\begin{proof}
This follows immediately by Proposition~\ref{prop:homogeneous acyclic Nakayama with nZ-ct} and Proposition~\ref{prop:homogeneous cyclic Nakayama with nZ-ct}.
\end{proof}

\begin{example}\label{ex:homogeneous case}
We give examples of Theorem~\ref{thrm:homogeneous Nakayama with nZ-ct} below. In each case the additive closure of the modules corresponding to vertices in $\Gamma(\La)$ marked by rectangles forms an $n\ZZ$-cluster tilting subcategory.
\begin{enumerate}
 \item[(a)] Let $\La=\K A_9/R^2$ and $n = 4$. Then $\Gamma(\La)$ is
    
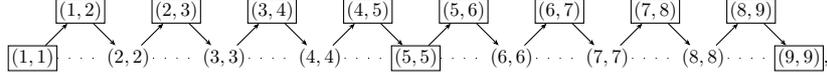
\[
    \begin{tikzpicture}[scale=0.9, transform shape, baseline={(current bounding box.center)}]
    
    \tikzstyle{mod}=[rectangle, minimum width=6pt, draw=none, inner sep=1.5pt, scale=0.8]
    \tikzstyle{nct}=[rectangle, minimum width=3pt, draw, inner sep=1.5pt, scale=0.8]
    
    \node[nct] (11) at (0,0) {$(1,1)$};
    \node[mod] (22) at (1.4,0) {$(2,2)$};
    \node[mod] (33) at (2.8,0) {$(3,3)$};
    \node[mod] (44) at (4.2,0) {$(4,4)$};
    \node[nct] (55) at (5.6,0) {$(5,5)$};
    \node[mod] (66) at (7,0) {$(6,6)$};
    \node[mod] (77) at (8.4,0) {$(7,7)$};
    \node[mod] (88) at (9.8,0) {$(8,8)$};
    \node[nct] (99) at (11.2,0) {$(9,9)$\nospacepunct{,}};
    
    \draw[loosely dotted] (11.east) -- (22);
    \draw[loosely dotted] (22.east) -- (33);
    \draw[loosely dotted] (33.east) -- (44);
    \draw[loosely dotted] (44.east) -- (55);
    \draw[loosely dotted] (55.east) -- (66);
    \draw[loosely dotted] (66.east) -- (77);
    \draw[loosely dotted] (77.east) -- (88);
    \draw[loosely dotted] (88.east) -- (99);
        
    \node[nct] (12) at (0.7,0.7) {$(1,2)$};    
    \node[nct] (23) at (2.1,0.7) {$(2,3)$};    
    \node[nct] (34) at (3.5,0.7) {$(3,4)$};    
    \node[nct] (45) at (4.9,0.7) {$(4,5)$};    
    \node[nct] (56) at (6.3,0.7) {$(5,6)$};    
    \node[nct] (67) at (7.7,0.7) {$(6,7)$};    
    \node[nct] (78) at (9.1,0.7) {$(7,8)$};    
    \node[nct] (89) at (10.5,0.7) {$(8,9)$};  

    \draw[-{Stealth[scale=0.5]}] (11) -- (12);
    \draw[-{Stealth[scale=0.5]}] (22) -- (23);
    \draw[-{Stealth[scale=0.5]}] (33) -- (34);
    \draw[-{Stealth[scale=0.5]}] (44) -- (45);
    \draw[-{Stealth[scale=0.5]}] (55) -- (56);
    \draw[-{Stealth[scale=0.5]}] (66) -- (67);
    \draw[-{Stealth[scale=0.5]}] (77) -- (78);
    \draw[-{Stealth[scale=0.5]}] (88) -- (89);
    
    \draw[-{Stealth[scale=0.5]}] (12) -- (22);
    \draw[-{Stealth[scale=0.5]}] (23) -- (33);
    \draw[-{Stealth[scale=0.5]}] (34) -- (44);
    \draw[-{Stealth[scale=0.5]}] (45) -- (55);
    \draw[-{Stealth[scale=0.5]}] (56) -- (66);
    \draw[-{Stealth[scale=0.5]}] (67) -- (77);
    \draw[-{Stealth[scale=0.5]}] (78) -- (88);
    \draw[-{Stealth[scale=0.5]}] (89) -- (99);

    \end{tikzpicture}
    \]
    
    \item[(b)] Let $\La=\K A_7/R^3$ and $n = 4$. Then $\Gamma(\La)$ is
    
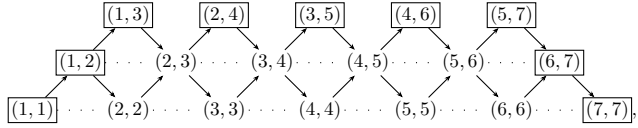
\[
    \begin{tikzpicture}[scale=0.9, transform shape, baseline={(current bounding box.center)}]
    
    \tikzstyle{mod}=[rectangle, minimum width=6pt, draw=none, inner sep=1.5pt, scale=0.8]
    \tikzstyle{nct}=[rectangle, minimum width=3pt, draw, inner sep=1.5pt, scale=0.8]
    
    \node[nct] (11) at (0,0) {$(1,1)$};
    \node[mod] (22) at (1.4,0) {$(2,2)$};
    \node[mod] (33) at (2.8,0) {$(3,3)$};
    \node[mod] (44) at (4.2,0) {$(4,4)$};
    \node[mod] (55) at (5.6,0) {$(5,5)$};
    \node[mod] (66) at (7,0) {$(6,6)$};
    \node[nct] (77) at (8.4,0) {$(7,7)$\nospacepunct{,}};
    
    \draw[loosely dotted] (11.east) -- (22);
    \draw[loosely dotted] (22.east) -- (33);
    \draw[loosely dotted] (33.east) -- (44);
    \draw[loosely dotted] (44.east) -- (55);
    \draw[loosely dotted] (55.east) -- (66);
    \draw[loosely dotted] (66.east) -- (77);
        
    \node[nct] (12) at (0.7,0.7) {$(1,2)$};    
    \node[mod] (23) at (2.1,0.7) {$(2,3)$};    
    \node[mod] (34) at (3.5,0.7) {$(3,4)$};    
    \node[mod] (45) at (4.9,0.7) {$(4,5)$};    
    \node[mod] (56) at (6.3,0.7) {$(5,6)$};    
    \node[nct] (67) at (7.7,0.7) {$(6,7)$};  

    \draw[-{Stealth[scale=0.5]}] (11) -- (12);
    \draw[-{Stealth[scale=0.5]}] (22) -- (23);
    \draw[-{Stealth[scale=0.5]}] (33) -- (34);
    \draw[-{Stealth[scale=0.5]}] (44) -- (45);
    \draw[-{Stealth[scale=0.5]}] (55) -- (56);
    \draw[-{Stealth[scale=0.5]}] (66) -- (67);
    
    \draw[-{Stealth[scale=0.5]}] (12) -- (22);
    \draw[-{Stealth[scale=0.5]}] (23) -- (33);
    \draw[-{Stealth[scale=0.5]}] (34) -- (44);
    \draw[-{Stealth[scale=0.5]}] (45) -- (55);
    \draw[-{Stealth[scale=0.5]}] (56) -- (66);
    \draw[-{Stealth[scale=0.5]}] (67) -- (77);

    \draw[loosely dotted] (12.east) -- (23);
    \draw[loosely dotted] (23.east) -- (34);
    \draw[loosely dotted] (34.east) -- (45);
    \draw[loosely dotted] (45.east) -- (56);
    \draw[loosely dotted] (56.east) -- (67);
    
    \node[nct] (13) at (1.4,1.4) {$(1,3)$};
    \node[nct] (24) at (2.8,1.4) {$(2,4)$};
    \node[nct] (35) at (4.2,1.4) {$(3,5)$};
    \node[nct] (46) at (5.6,1.4) {$(4,6)$};
    \node[nct] (57) at (7,1.4) {$(5,7)$};
    
    \draw[-{Stealth[scale=0.5]}] (12) -- (13);
    \draw[-{Stealth[scale=0.5]}] (23) -- (24);
    \draw[-{Stealth[scale=0.5]}] (34) -- (35);
    \draw[-{Stealth[scale=0.5]}] (45) -- (46);
    \draw[-{Stealth[scale=0.5]}] (56) -- (57);
    
    \draw[-{Stealth[scale=0.5]}] (13) -- (23);
    \draw[-{Stealth[scale=0.5]}] (24) -- (34);
    \draw[-{Stealth[scale=0.5]}] (35) -- (45);
    \draw[-{Stealth[scale=0.5]}] (46) -- (56);
    \draw[-{Stealth[scale=0.5]}] (57) -- (67);
    
    \end{tikzpicture}
    \]
    
    \item[(c)] Let $\tilde{\La}=\K \tilde{A}_8/R^2$ and $n = 4$. Then $\Gamma(\La)$ is
    
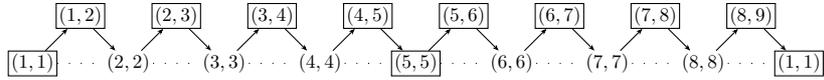
\[
    \begin{tikzpicture}[scale=0.9, transform shape, baseline={(current bounding box.center)}]
    
    \tikzstyle{mod}=[rectangle, minimum width=6pt, draw=none, inner sep=1.5pt, scale=0.8]
    \tikzstyle{nct}=[rectangle, minimum width=3pt, draw, inner sep=1.5pt, scale=0.8]
    
    \node[nct] (11) at (0,0) {$(1,1)$};
    \node[mod] (22) at (1.4,0) {$(2,2)$};
    \node[mod] (33) at (2.8,0) {$(3,3)$};
    \node[mod] (44) at (4.2,0) {$(4,4)$};
    \node[nct] (55) at (5.6,0) {$(5,5)$};
    \node[mod] (66) at (7,0) {$(6,6)$};
    \node[mod] (77) at (8.4,0) {$(7,7)$};
    \node[mod] (88) at (9.8,0) {$(8,8)$};
    \node[nct] (99) at (11.2,0) {$(1,1)$};
    
    \draw[loosely dotted] (11.east) -- (22);
    \draw[loosely dotted] (22.east) -- (33);
    \draw[loosely dotted] (33.east) -- (44);
    \draw[loosely dotted] (44.east) -- (55);
    \draw[loosely dotted] (55.east) -- (66);
    \draw[loosely dotted] (66.east) -- (77);
    \draw[loosely dotted] (77.east) -- (88);
    \draw[loosely dotted] (88.east) -- (99);
        
    \node[nct] (12) at (0.7,0.7) {$(1,2)$};    
    \node[nct] (23) at (2.1,0.7) {$(2,3)$};    
    \node[nct] (34) at (3.5,0.7) {$(3,4)$};    
    \node[nct] (45) at (4.9,0.7) {$(4,5)$};    
    \node[nct] (56) at (6.3,0.7) {$(5,6)$};    
    \node[nct] (67) at (7.7,0.7) {$(6,7)$};    
    \node[nct] (78) at (9.1,0.7) {$(7,8)$};    
    \node[nct] (89) at (10.5,0.7) {$(8,9)$};  

    \draw[-{Stealth[scale=0.5]}] (11) -- (12);
    \draw[-{Stealth[scale=0.5]}] (22) -- (23);
    \draw[-{Stealth[scale=0.5]}] (33) -- (34);
    \draw[-{Stealth[scale=0.5]}] (44) -- (45);
    \draw[-{Stealth[scale=0.5]}] (55) -- (56);
    \draw[-{Stealth[scale=0.5]}] (66) -- (67);
    \draw[-{Stealth[scale=0.5]}] (77) -- (78);
    \draw[-{Stealth[scale=0.5]}] (88) -- (89);
    
    \draw[-{Stealth[scale=0.5]}] (12) -- (22);
    \draw[-{Stealth[scale=0.5]}] (23) -- (33);
    \draw[-{Stealth[scale=0.5]}] (34) -- (44);
    \draw[-{Stealth[scale=0.5]}] (45) -- (55);
    \draw[-{Stealth[scale=0.5]}] (56) -- (66);
    \draw[-{Stealth[scale=0.5]}] (67) -- (77);
    \draw[-{Stealth[scale=0.5]}] (78) -- (88);
    \draw[-{Stealth[scale=0.5]}] (89) -- (99);

    \end{tikzpicture}
    \]
    where $(1,1)$ has been drawn twice.
    \item[(d)] Let $\La=\K \tilde{A}_6/R^5$and $n = 3$. Then $\Gamma(\La)$ is
    
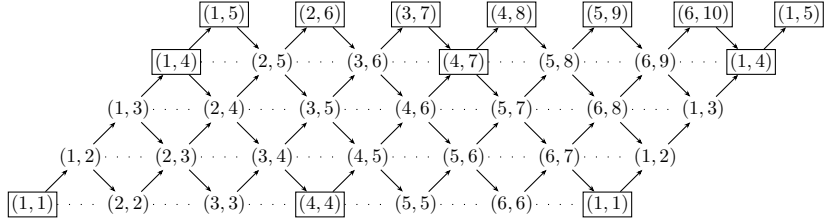
\[
    \begin{tikzpicture}[scale=0.9, transform shape, baseline={(current bounding box.center)}]
    
    \tikzstyle{mod}=[rectangle, minimum width=6pt, draw=none, inner sep=1.5pt, scale=0.8]
    \tikzstyle{nct}=[rectangle, minimum width=3pt, draw, inner sep=1.5pt, scale=0.8]
    
    \node[nct] (11) at (0,0) {$(1,1)$};
    \node[mod] (22) at (1.4,0) {$(2,2)$};
    \node[mod] (33) at (2.8,0) {$(3,3)$};
    \node[nct] (44) at (4.2,0) {$(4,4)$};
    \node[mod] (55) at (5.6,0) {$(5,5)$};
    \node[mod] (66) at (7,0) {$(6,6)$};
    \node[nct] (11b) at (8.4,0) {$(1,1)$};
    
    \draw[loosely dotted] (11.east) -- (22);
    \draw[loosely dotted] (22.east) -- (33);
    \draw[loosely dotted] (33.east) -- (44);
    \draw[loosely dotted] (44.east) -- (55);
    \draw[loosely dotted] (55.east) -- (66);
    \draw[loosely dotted] (66.east) -- (11b);
        
    \node[mod] (12) at (0.7,0.7) {$(1,2)$};    
    \node[mod] (23) at (2.1,0.7) {$(2,3)$};    
    \node[mod] (34) at (3.5,0.7) {$(3,4)$};    
    \node[mod] (45) at (4.9,0.7) {$(4,5)$};    
    \node[mod] (56) at (6.3,0.7) {$(5,6)$};    
    \node[mod] (67) at (7.7,0.7) {$(6,7)$};  
    \node[mod] (12b) at (9.1,0.7) {$(1,2)$};

    \draw[-{Stealth[scale=0.5]}] (11) -- (12);
    \draw[-{Stealth[scale=0.5]}] (22) -- (23);
    \draw[-{Stealth[scale=0.5]}] (33) -- (34);
    \draw[-{Stealth[scale=0.5]}] (44) -- (45);
    \draw[-{Stealth[scale=0.5]}] (55) -- (56);
    \draw[-{Stealth[scale=0.5]}] (66) -- (67);
    \draw[-{Stealth[scale=0.5]}] (11b) -- (12b);
    
    \draw[-{Stealth[scale=0.5]}] (12) -- (22);
    \draw[-{Stealth[scale=0.5]}] (23) -- (33);
    \draw[-{Stealth[scale=0.5]}] (34) -- (44);
    \draw[-{Stealth[scale=0.5]}] (45) -- (55);
    \draw[-{Stealth[scale=0.5]}] (56) -- (66);
    \draw[-{Stealth[scale=0.5]}] (67) -- (11b);

    \draw[loosely dotted] (12.east) -- (23);
    \draw[loosely dotted] (23.east) -- (34);
    \draw[loosely dotted] (34.east) -- (45);
    \draw[loosely dotted] (45.east) -- (56);
    \draw[loosely dotted] (56.east) -- (67);
    \draw[loosely dotted] (67.east) -- (12b);
    
    \node[mod] (13) at (1.4,1.4) {$(1,3)$};
    \node[mod] (24) at (2.8,1.4) {$(2,4)$};
    \node[mod] (35) at (4.2,1.4) {$(3,5)$};
    \node[mod] (46) at (5.6,1.4) {$(4,6)$};
    \node[mod] (57) at (7,1.4) {$(5,7)$};
    \node[mod] (68) at (8.4,1.4) {$(6,8)$};
    \node[mod] (13b) at (9.8,1.4) {$(1,3)$};
    
    \draw[-{Stealth[scale=0.5]}] (12) -- (13);
    \draw[-{Stealth[scale=0.5]}] (23) -- (24);
    \draw[-{Stealth[scale=0.5]}] (34) -- (35);
    \draw[-{Stealth[scale=0.5]}] (45) -- (46);
    \draw[-{Stealth[scale=0.5]}] (56) -- (57);
    \draw[-{Stealth[scale=0.5]}] (67) -- (68);
    \draw[-{Stealth[scale=0.5]}] (12b) -- (13b);
    
    \draw[-{Stealth[scale=0.5]}] (13) -- (23);
    \draw[-{Stealth[scale=0.5]}] (24) -- (34);
    \draw[-{Stealth[scale=0.5]}] (35) -- (45);
    \draw[-{Stealth[scale=0.5]}] (46) -- (56);
    \draw[-{Stealth[scale=0.5]}] (57) -- (67);
    \draw[-{Stealth[scale=0.5]}] (68) -- (12b);
    
    \draw[loosely dotted] (13.east) -- (24);
    \draw[loosely dotted] (24.east) -- (35);
    \draw[loosely dotted] (35.east) -- (46);
    \draw[loosely dotted] (46.east) -- (57);
    \draw[loosely dotted] (57.east) -- (68);
    \draw[loosely dotted] (68.east) -- (13b);
    
    \node[nct] (14) at (2.1,2.1) {$(1,4)$};
    \node[mod] (25) at (3.5,2.1) {$(2,5)$};
    \node[mod] (36) at (4.9,2.1) {$(3,6)$};
    \node[nct] (47) at (6.3,2.1) {$(4,7)$};
    \node[mod] (58) at (7.7,2.1) {$(5,8)$};
    \node[mod] (69) at (9.1,2.1) {$(6,9)$};
    \node[nct] (14b) at (10.5,2.1) {$(1,4)$};
    
    \draw[-{Stealth[scale=0.5]}] (13) -- (14);
    \draw[-{Stealth[scale=0.5]}] (24) -- (25);
    \draw[-{Stealth[scale=0.5]}] (35) -- (36);
    \draw[-{Stealth[scale=0.5]}] (46) -- (47);
    \draw[-{Stealth[scale=0.5]}] (57) -- (58);
    \draw[-{Stealth[scale=0.5]}] (68) -- (69);
    \draw[-{Stealth[scale=0.5]}] (13b) -- (14b);
    
    \draw[-{Stealth[scale=0.5]}] (14) -- (24);
    \draw[-{Stealth[scale=0.5]}] (25) -- (35);
    \draw[-{Stealth[scale=0.5]}] (36) -- (46);
    \draw[-{Stealth[scale=0.5]}] (47) -- (57);
    \draw[-{Stealth[scale=0.5]}] (58) -- (68);
    \draw[-{Stealth[scale=0.5]}] (69) -- (13b);
    
    \draw[loosely dotted] (14.east) -- (25);
    \draw[loosely dotted] (25.east) -- (36);
    \draw[loosely dotted] (36.east) -- (47);
    \draw[loosely dotted] (47.east) -- (58);
    \draw[loosely dotted] (58.east) -- (69);
    \draw[loosely dotted] (69.east) -- (14b);
    
    \node[nct] (15) at (2.8,2.8) {$(1,5)$};
    \node[nct] (26) at (4.2,2.8) {$(2,6)$};
    \node[nct] (37) at (5.6,2.8) {$(3,7)$};
    \node[nct] (48) at (7,2.8) {$(4,8)$};
    \node[nct] (59) at (8.4,2.8) {$(5,9)$};
    \node[nct] (610) at (9.8,2.8) {$(6,10)$};
    \node[nct] (15b) at (11.2,2.8) {$(1,5)$};
    
    \draw[-{Stealth[scale=0.5]}] (14) -- (15);
    \draw[-{Stealth[scale=0.5]}] (25) -- (26);
    \draw[-{Stealth[scale=0.5]}] (36) -- (37);
    \draw[-{Stealth[scale=0.5]}] (47) -- (48);
    \draw[-{Stealth[scale=0.5]}] (58) -- (59);
    \draw[-{Stealth[scale=0.5]}] (69) -- (610);
    \draw[-{Stealth[scale=0.5]}] (14b) -- (15b);
    
    \draw[-{Stealth[scale=0.5]}] (15) -- (25);
    \draw[-{Stealth[scale=0.5]}] (26) -- (36);
    \draw[-{Stealth[scale=0.5]}] (37) -- (47);
    \draw[-{Stealth[scale=0.5]}] (48) -- (58);
    \draw[-{Stealth[scale=0.5]}] (59) -- (69);
    \draw[-{Stealth[scale=0.5]}] (610) -- (14b);
    \end{tikzpicture}
    \]
    where $(1,j)$, has been drawn twice for $1 \le j \le 5$.
\end{enumerate}
\end{example}

\subsection{Non-homogeneous relations: acyclic case}

To give the classification in this case, we first need to recall the notion of gluing from \cite{Vas3}. Since we shall deal with several Nakayama algebras at the same time, it is convenient to use a different notation for their quivers. To this end, for $m_1,m_2\in\ZZ$ with $m_1<m_2$ we denote by $A_{[m_1,m_2]}$ the quiver
\[\begin{tikzpicture}
\node (Q) at (-1,0) {$A_{[m_1,m_2]}:$};
\node (1) at (0,0) {$m_2$}; 
\node (2) at (2,0) {$m_2-1$};
\node (3) at (4.5,0) {$m_2-2$};
\node (4) at (6.75,0) {$\cdots$};
\node (m-1) at (9,0) {$m_1+1$};
\node (m) at (11.25,0) {$m_1$.};

\draw[-{Stealth[scale=0.5]}] (1) -- (2) node[draw=none,midway,above] {$\alpha_{m_2}$};
\draw[-{Stealth[scale=0.5]}] (2) -- (3) node[draw=none,midway,above] {$\alpha_{m_2-1}$};
\draw[-{Stealth[scale=0.5]}] (3) -- (4) node[draw=none,midway,above] {$\alpha_{m_2-2}$};
\draw[-{Stealth[scale=0.5]}] (4) -- (m-1) node[draw=none,midway,above] {$\alpha_{m_1+2}$};
\draw[-{Stealth[scale=0.5]}] (m-1) -- (m) node[draw=none,midway,above] {$\alpha_{m_1+1}$};
\end{tikzpicture}\]

If $\La=\K A_{[m_1,m_2]}/I$ is an acyclic Nakayama algebra and $M(i,j)\in\mo{\La}$, we then have that $M(i,j)\isom M(i+m_2-m_1+1,j+m_2-m_1+1)$ where we consider this isomorphism as the identity. To avoid certain technicalities arising from this identification, from now on we drop our assumption that $i\in\ZZ$ and we only allow $m_1\leq i\leq m_2$. We also denote by $\ind(\La)$ the set
\[
\ind(\La) \coloneqq \{ M(i,j)\in\mo{\La} \mid m_1 \leq i \leq m_2\},
\]
which is then a complete and irredundant set of representatives of isomorphism classes of indecomposable $\La$-modules. If $\cC\subseteq \mo{\La}$ is a subcategory, we set $\ind(\cC)=\{M(i,j)\in\cC \mid m_1\leq i\leq m_2\}$.

\begin{definition}\label{def:gluing of acyclic Nakayama}
	Let $m_1,m_2,m_3\in\ZZ$ with $m_1<m_2<m_3$ and let
	$\La_1=\K A_{[m_1,m_2]}/I_1$ and
	$\La_2=\K A_{[m_2,m_3]}/I_2$ be two acyclic Nakayama algebras. We define the \emph{gluing of $\La_2$ and $\La_1$} to be the acyclic Nakayama algebra $\La= \La_1\glue \La_2$ given by $\La=\K A_{[m_1,m_3]}/ I_{\La}$ where $I_{\La}$ is the ideal generated by $I_1\cup I_2\cup \{\alpha_{m_2}\alpha_{m_2+1}\}$. 
\end{definition}

We immediately have the following lemma.

\begin{lemma}\label{lem:ungluing at a simple, acyclic}
Let $m_1,m_2,m_3\in\ZZ$ with $m_1<m_2<m_3$ and let $\La=\K A_{[m_1,m_3]}/I_{\La}$ be an acyclic Nakayama algebra. Assume that $M(m_2-1,m_2+1)\not\in\mo{\La}$. Let $I_1 = \K A_{[m_1,m_2]} \cap I_\La$ and $I_2 = \K A_{[m_2,m_3]} \cap I_\La$. Set $\La_1=\K A_{[m_1,m_2]}/I_1$ and $\La_2=\K A_{[m_2,m_3]}/I_2$. Then $\La=\La_1\glue \La_2$.
\end{lemma}

\begin{proof}
The condition $M(m_2-1,m_2+1)\not\in\mo{\La}$ means precisely $\alpha_{m_2}\alpha_{m_2+1} \in I_\La$, and so the claim follows directly from Definition~\ref{def:gluing of acyclic Nakayama}.
\end{proof}

Definition~\ref{def:gluing of acyclic Nakayama} is a special case of \cite[Lemma 3.21]{Vas3}. Since both $\La_1$ and $\La_2$ can be viewed as quotient algebras of $\La$ there are full and faithful embeddings of $\mo{\La_1}$ and $\mo{\La_2}$ in $\mo{\La}$. Specifically, an indecomposable module $M(i,j)$ over $\La_1$ satisfies $m_1\leq i \leq m_2$ and is mapped to the indecomposable $\La$-module $M(i,j)$ with $m_1\leq i \leq m_3$ and similarly for $\La_2$. For more details we refer to \cite[Section 3]{Vas3}. It also follows directly from the definition that
\[(\La_1 \glue \La_2) \glue \La_3 = \La_1 \glue (\La_2 \glue \La_3),\]
and hence gluing is associative. 

In the following proposition we collect some basic properties of gluing.

\begin{proposition}\label{prop:basic gluing results}
	Let $m_1,m_2,m_3\in\ZZ$ with $m_1<m_2<m_3$ and let $\La_1=\K A_{[m_1,m_2]}/I_1$ and $\La_2=\K A_{[m_2,m_3]}/I_2$ be two acyclic Nakayama algebras. Let $\La=\La_1\glue \La_2$.
	\begin{enumerate}[label=(\alph*)]
		\item $\ind(\La)=\ind(\La_1)\cup\ind(\La_2)$ and $\ind(\La_1)\cap\ind(\La_2) = \{M(m_2,m_2)\}$.
		\item $M(i,j)$ is a projective $\La$-module if and only if exactly one of the following conditions hold
		\begin{itemize}
			\item[(i)] either $M(i,j)$ is a projective $\La_1$-module, or
			\item[(ii)] $M(i,j)$ is a projective $\La_2$-module different from $M(m_2,m_2)$.
		\end{itemize}
		\item $M(i,j)$ is an injective $\La$-module if and only if exactly one of the following conditions hold
		\begin{itemize}
			\item[(i)] either $M(i,j)$ an injective $\La_2$-module, or
			\item[(ii)] $M(i,j)$ is an injective $\La_1$-module different from $M(m_2,m_2)$.
		\end{itemize}
		\item If $M(i,j)\in\ind(\La_1)$, then $\tau_{\La}\left(M(i,j)\right)= \tau_{\La_1}\left(M(i,j)\right)$ and $\om_{\La}\left(M(i,j)\right)= \om_{\La_1}\left(M(i,j)\right)$. If moreover $M(i,j)\neq M(m_2,m_2)$, then $\tau_{\La}^-\left(M(i,j)\right)= \tau_{\La_1}^-\left(M(i,j)\right)$ and  $\om_{\La}^-\left(M(i,j)\right)= \om_{\La_1}^-\left(M(i,j)\right)$.
		\item If $M(i,j)\in\ind(\La_2)$, then $\tau_{\La}^-\left(M(i,j)\right)= \tau_{\La_2}^-\left(M(i,j)\right)$ and $\om_{\La}^-\left(M(i,j)\right)= \om_{\La_2}^-\left(M(i,j)\right)$. If moreover $M(i,j)\neq M(m_2,m_2)$, then $\tau_{\La}\left(M(i,j)\right)= \tau_{\La_2}\left(M(i,j)\right)$ and  $\om_{\La}\left(M(i,j)\right)= \om_{\La_2}\left(M(i,j)\right)$.
	\end{enumerate}
\end{proposition}

\begin{proof}
Part (a) follows from definition. Parts (b) and (c) follow from part (a) and  Proposition~\ref{prop:basic Nakayama results2}(b) and (c). Parts (d) and (e) follow from part (a), Proposition~\ref{prop:AR quiver of Nakayama} and Proposition~\ref{prop:basic Nakayama results2}(b) and (c). For a more detailed proof of a more general version of this proposition see \cite[Corollary 3.37]{Vas3} and \cite[Corollary 3.39]{Vas3}.
\end{proof}

With this we can show the following facts about $n\ZZ$-cluster tilting subcategories of glued acyclic Nakayama algebras.

\begin{proposition}\label{prop:gluing of nZ-cluster tilting is nZ-cluster tilting}
	Let $m_1,m_2,m_3\in\ZZ$ with $m_1<m_2<m_3$ and let
	$\La_1=\K A_{[m_1,m_2]}/I_1$ and 
	$\La_2=\K A_{[m_2,m_3]}/I_2$ be two acyclic Nakayama algebras. Let $\La=\La_1\glue \La_2$. If $\La_i$ admits an $n\ZZ$-cluster tilting subcategory $\cC_{\La_i}$ for $i=1,2$, then $\cC_{\La}\coloneqq\add\{\cC_{\La_1},\cC_{\La_2}\}$ is an $n\ZZ$-cluster tilting subcategory of $\mo{\La}$.
\end{proposition}

\begin{proof}
	That $\cC_{\La}$ is an $n$-cluster tilting subcategory of $\mo{\La}$ follows by \cite[Corollary 4.17]{Vas3}. To show that $\cC_{\La}$ is $n\ZZ$-cluster tilting, let $M(i,j)\in\ind(\cC_{\La})$ and we show that $\Omega_{\La}^n(M(i,j))\in\cC_{\La}$. Clearly we may assume that $M(i,j)$ is not projective and so $\Omega^n_{\La}(M(i,j))$ is indecomposable by Lemma~\ref{lem:n-th syzygy and n-th cosyzygy are indecomposable}(a). We consider the cases $M(i,j)\in\cC_{\La_1}$ and $M(i,j)\in\cC_{\La_2}\setminus\{M(m_2,m_2)\}$ separately (notice that $M(m_2,m_2)\in\cC_{\La_1}$ since it is injective as a $\La_1$-module).
	
	If $M(i,j)\in\cC_{\La_1}$, then $M(i,j)$ is not projective as a $\La_1$-module by Proposition~\ref{prop:basic gluing results}(b). Hence $\Omega^{k}_{\La_1}(M(i,j))\in\ind(\La_1)$ for all $0\leq k \leq n$ by Lemma~\ref{lem:n-th syzygy and n-th cosyzygy are indecomposable}(a). By Proposition~\ref{prop:basic gluing results}(d) and since $\cC_{\La_1}$ is $n\ZZ$-cluster tilting, it follows that 
	\[
	\Omega^n_{\La}(M(i,j)) = \Omega^n_{\La_1}(M(i,j))\in\cC_{\La_1}\subseteq \cC_{\La},
	\]
	as required.
	
	If $M(i,j)\in\cC_{\La_2}$ and $M(i,j)\neq M(m_2,m_2)$, then $M(i,j)$ is not projective as a $\La_2$-module by Proposition~\ref{prop:basic gluing results}(b). By Lemma~\ref{lem:n-th syzygy and n-th cosyzygy are indecomposable}(a) and since $M(m_2,m_2)$ is a projective $\La_2$-module it follows that $\Omega^k_{\La_2}(M(i,j))\in\ind(\La_2)\setminus \{M(m_2,m_2)\}$ for all $0\leq k\leq n-1$. By Proposition~\ref{prop:basic gluing results}(d) and since $\cC_{\La_2}$ is $n\ZZ$-cluster tilting, we conclude that 
	\[
	\Omega^n_{\La}(M(i,j)) = \Omega^n_{\La_2}(M(i,j))\in\cC_{\La_2}\subseteq \cC,
	\]
	as required.
\end{proof}

The converse of Proposition~\ref{prop:gluing of nZ-cluster tilting is nZ-cluster tilting} does not hold, but we have the following partial converse instead.

\begin{proposition}\label{prop:ungluing of nZ-cluster tilting at simple, acyclic}
Let $m_1,m_2,m_3\in\ZZ$ with $m_1<m_2<m_3$ and let $\La_1=\K A_{[m_1,m_2]}/I_1$ and $\La_2=\K A_{[m_2,m_3]}/I_2$ be two acyclic Nakayama algebras. Let $\La=\La_1\glue \La_2$. If $\La$ admits an $n\ZZ$-cluster tilting subcategory $\cC_{\La}$ such that $M(m_2,m_2)\in\cC_{\La}$, then
\[
\cC_{\La_1} \coloneqq \mo{\La_1}\cap \cC_{\La} \text{ and } \cC_{\La_2} \coloneqq \mo{\La_2}\cap \cC_{\La}
\]
are $n\ZZ$-cluster tilting subcategories of $\mo{\La_1}$ and $\mo{\La_2}$, respectively.
\end{proposition}

\begin{proof}
	We only show that $\cC_{\La_1}=\mo{\La_1}\cap \cC_{\La}$ is an $n\ZZ$-cluster tilting subcategory of $\mo{\La_1}$ as the claim about $\cC_{\La_2}$ can be proved dually. We first claim that for all $M(i,j)\in\ind(\cC_{\La_1})$ such that $M(i,j)$ is not injective as a $\La_1$-module and for all $1\leq k\leq n-1$ we have
	\begin{equation}\label{eq:k-th cosyzygy}
	\Omega^{-k}_{\La_1}(M(i,j))=\Omega^{-k}_{\La}(M(i,j))\neq M(m_2,m_2).
	\end{equation}
	By Proposition~\ref{prop:basic gluing results}(d) it is enough to show that $\Omega^{-k}_{\La}(M(i,j))\neq M(m_2,m_2)$ for all $1\leq k\leq n-1$. But this follows since $\{M(i,j),M(m_2,m_2)\}\subseteq\cC_{\La}$ and so $\Ext^{k}_{\La}(M(m_2,m_2),M(i,j))=0$ for $1\leq k\leq n-1$.
	
	To show that $\cC_{\La_1}$ is $n\ZZ$-cluster tilting, it is enough to show that the statements (a)-(d) in Proposition~\ref{prop:basic nZ-cluster tilting results} and statement (b) in Proposition~\ref{prop:basic nZ-cluster tilting results2} hold. That Proposition~\ref{prop:basic nZ-cluster tilting results} (a) holds follows from Proposition~\ref{prop:basic gluing results}(b) and (c) and since $M(m_2,m_2)\in\cC_{\La}$. That Proposition~\ref{prop:basic nZ-cluster tilting results} (c) and Proposition~\ref{prop:basic nZ-cluster tilting results2} (b) hold follows from Proposition~\ref{prop:basic gluing results}(d). That Proposition~\ref{prop:basic nZ-cluster tilting results} (d) holds follows from (\ref{eq:k-th cosyzygy}). 
	
	Finally it remains to show that Proposition~\ref{prop:basic nZ-cluster tilting results}(b) holds for $\cC_{\La_1}$. For any non-projective $\La_1$-module $M(i,j)\in\ind(\La_1)$ we have that $(\tn)_{\La_1}(M(i,j))=(\tn)_{\La}(M(i,j))$ by Proposition~\ref{prop:basic gluing results}(d). For any non-injective $\La_1$-module $M(i,j)\in\ind(\La_1)$ we have that $(\tno)_{\La_1}(M(i,j))=(\tno)_{\La}(M(i,j))$ by (\ref{eq:k-th cosyzygy}) and Proposition~\ref{prop:basic gluing results}(d). Then Proposition~\ref{prop:basic nZ-cluster tilting results}(b) holds for $\cC_{\La_1}$ since it holds for $\cC_{\La}$.
\end{proof}

Using Proposition~\ref{prop:basic gluing results} we can also get some control of how global dimension behaves under gluing.

\begin{proposition}\label{prop:global dimension of gluing}
Let $\La_1$ and $\La_2$ be two acyclic Nakayama algebras. Then $$\gldim(\La_1\glue \La_2) \le \gldim(\La_1) + \gldim(\La_2).$$
\end{proposition}
\begin{proof}
This result follows from \cite[Corollary 2.38]{Vas3}. We include a proof for the reader's convenience.

Set $\gldim(\La_1) = d_1$, $\gldim(\La_2) = d_2$ and $\La = \La_1\glue \La_2$. As before let $M(m_2,m_2)$ be the simple that satisfies $M(m_2,m_2) \in \ind(\La_1)\cap\ind(\La_2)$.

Let $M(i,j) \in \ind(\La)$. If $M(i,j) \in \ind(\La_1)$, then $\pdim_\La M(i,j) = \pdim_{\La_1} M(i,j) \le d_1$ by Proposition~\ref{prop:basic gluing results}(d). If $M(i,j) \in \ind{\La_2}$, then by Proposition~\ref{prop:basic gluing results}(e) there is $0 \le k \le d_2$ such that $\Omega^{k}_\La M(i,j) = \Omega^{k}_{\La_2} M(i,j)$ is either projective or equal to $M(m_2,m_2)$. I the first case $\pdim_\La M(i,j) = k \le d_2$. In the second case $\pdim_\La M(i,j) = k + \pdim_\La M(m_2,m_2) = k + \pdim_{\La_1} M(m_2,m_2) \le d_2+d_1$.
\end{proof}

To describe acyclic Nakayama algebras with non-homogeneous relations which admit $n\ZZ$-cluster tilting subcategories we introduce the following notion.

\begin{definition}
	We call an acyclic Nakayama algebra $\La$ \emph{piecewise homogeneous} if $\La=\La_1\glue \cdots \glue\La_r$, where each $\La_i$ is a homogeneous acyclic Nakayama algebra.
\end{definition}

\begin{theorem}\label{thrm:acyclic non-homogeneous case}
	Let $\La$ be an acyclic Nakayama algebra. Then $\La$ admits an $n\ZZ$-cluster tilting subcategory $\cC_{\La}$ if and only if $\La$ is piecewise homogeneous 
	\[\La = \La_1 \glue \cdots \glue\La_r,\]
	where $\La_k$ for $1\leq k\leq r$ is a homogeneous acyclic Nakayama algebra which admits an $n\ZZ$-cluster tilting subcategory $\cC_{\La_k}$. In this case we have $\cC_{\La}=\add\{\cC_{\La_k}\mid 1\leq k\leq r\}$.
\end{theorem}

\begin{proof}
	Assume first that $\La= \La_1\glue \cdots \glue \La_r$ and that each $\La_k$ for $1\leq k\leq r$ is a homogeneous acyclic Nakayama algebra which admits an $n\ZZ$-cluster tilting subcategory $\cC_{\La_k}$. Set $\cC_{\La}=\add\{\cC_{\La_k}\mid 1\leq k\leq r\}$. Write $\La = \left(\La_1\glue\cdots\glue\La_{r-1}\right)\glue\La_r$. By an induction on $r$ and by using Proposition~\ref{prop:gluing of nZ-cluster tilting is nZ-cluster tilting} it follows that $\cC_{\La}$ is an $n\ZZ$-cluster tilting subcategory of $\mo{\La}$.
	
	Next we assume that $\La$ admits an $n\ZZ$-cluster tilting subcategory $\cC_{\La}$ and show that there exist homogeneous acyclic Nakayama algebras $\La_1,\ldots,\La_r$ such that $\La=\La_1\glue\cdots\glue\La_r$ and each $\La_k$ admits an $n\ZZ$-cluster tilting subcategory $\cC_{\La_k}$. We prove the claim by induction on the number of simple modules of $\La$. If $\La$ is homogeneous there is nothing to show. Assume that $\La$ is not homogeneous and write $\La=\K A_{[m_1,m_2]}/I_{\La}$ where $m_1,m_2\in\ZZ$ with $m_1<m_2$. By Corollary~\ref{cor:not homogeneous implies ungluing simple} there exists $i\in\ZZ$ with $m_1<i<m_2$ and such that $M(i-1,i+1)\not\in\mo{\La}$ and $M(i,i)\in\cC_{\La}$. By Lemma~\ref{lem:ungluing at a simple, acyclic} we have that $\La=A\glue B$ where $A=\K A_{[m_1,i]}/I_A$ and $B=\K A_{[i,m_2]}/I_B$ for suitable ideals $I_A$ and $I_B$. By Proposition~\ref{prop:ungluing of nZ-cluster tilting at simple, acyclic} we conclude that $\mo{A}$ and $\mo{B}$ admit $n\ZZ$-cluster tilting subcategories $\cC_A$ and $\cC_B$. The result follows by the induction hypothesis applied to $A$ and $B$.
\end{proof}

\begin{remark}\label{rem:uniqueness of decomposition in acyclic}
Let $\La$ be an acyclic Nakayama algebra which admits an $n\ZZ$-cluster tilting subcategory. Then $\La=\La_1\glue\cdots\glue\La_r$ where $\La_k=\K A_{[m_k,m_{k+1}]}/R^{l_k}$ for $1\leq k\leq r$ is a homogeneous acyclic Nakayama algebra which admits an $n\ZZ$-cluster tilting subcategory $\cC_{\La_k}$.
	\begin{enumerate}[label=(\alph*)]
		\item By Proposition~\ref{prop:homogeneous acyclic Nakayama with nZ-ct} we have that either $l_k=2$ and $n\divides m_{k+1}-m_k$ or $l_k\geq 3$, $l_k\divides m_{k+1}-m_k$ and $n=2\frac{m_{k+1}-m_k}{l_k}$. By \cite[Theorem 3]{Vas1} we have that $\gldim(\La_k)=n$ if and only if $n=2\frac{m_{k+1}-m_k}{l_k}$ if and only if $\cC_{\La_k}=\add(\La_k\oplus D(\La_k))$.
		
		\item The algebras $\La_1,\ldots,\La_r$ are not necessarily unique. Indeed, let $\La=\K A_{[1,m]}/R^2$ and $m-1=rn$ for some $r\in\ZZ_{\geq 2}$. Then $\La$ is a homogeneous acyclic Nakayama algebra which admits an $n\ZZ$-cluster tilting subcategory by Proposition~\ref{prop:homogeneous acyclic Nakayama with nZ-ct}. Moreover, for $1\leq k\leq r$, set $\La_k=\K A_{[1+(k-1)n,1+kn]}/R^2$. Then $\La_k$ admits an $n\ZZ$-cluster tilting subcategory $\cC_{\La_k}=\add(\La_k\oplus M(1+kn,1+kn))$ and $\La=\La_1\glue\cdots\glue\La_r$. Notice also that in this case $\gldim(\La_k)=n$.
		
		\item By (a) the only case where a simple module which is neither projective nor injective belongs to $\cC_{\La_k}$ is the case $l_k=2$ and $m_{k+1}-m_k=r'n$ for some $r'\in\ZZ_{\geq 2}$, which is equivalent to $\gldim(\La_k)>n$. By (b), in this case, we can write $\La_k$ as the gluing of $r'$ homogeneous acyclic Nakayama algebras of global dimension equal to $n$ in a unique way. Hence if we require that the algebras $\La_1,\ldots,\La_r$ have global dimension equal to $n$, then the decomposition $\La=\La_1\glue\cdots\glue\La_r$ is unique. Moreover, in this case $\gldim(\La) = rn$. Indeed, by Proposition~\ref{prop:global dimension of gluing}, $\gldim(\La) \le rn$. Furthermore, one may compute that $\Omega^{-kn}M(m_1,m_1) = M(m_{k+1},m_{k+1})$ for $0 \le k \le r$ by Proposition~\ref{prop:basic gluing results} and induction. In particular $\Omega^{-rn}M(m_1,m_1) \neq 0$ and $\gldim(\La) = rn$.
	\end{enumerate}
\end{remark}

\begin{example}\label{ex:non-homogeneous acyclic case}
Let $\La_1=\K A_{[1,7]}/R^3$ and $\La_2=\K A_{[7,15]}/R^2$. Then $\La_1$ and $\La_2$ admit $4\ZZ$-cluster subcategories $\cC_{\La_1}$ and $\cC_{\La_2}$ respectively by Theorem \ref{thrm:homogeneous Nakayama with nZ-ct} (see also Example \ref{ex:homogeneous case}). Let $\La=\La_1\glue \La_2$. Then $\La$ admits a $4\ZZ$-cluster tilting subcategory by Theorem \ref{thrm:acyclic non-homogeneous case}. Indeed, the Auslander--Reiten quiver $\Gamma(\La)$ of $\La$ is
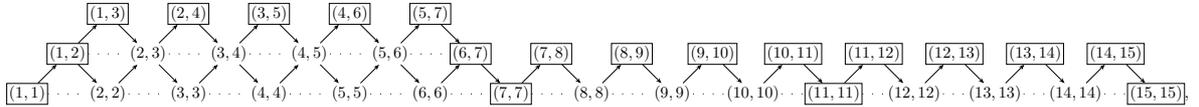
\[\resizebox {\columnwidth} {!} {
    \begin{tikzpicture}[scale=0.9, transform shape, baseline={(current bounding box.center)}]
    
    \tikzstyle{mod}=[rectangle, minimum width=6pt, draw=none, inner sep=1.5pt, scale=0.8]
    \tikzstyle{nct}=[rectangle, minimum width=3pt, draw, inner sep=1.5pt, scale=0.8]
    
    \node[nct] (11) at (0,0) {$(1,1)$};
    \node[mod] (22) at (1.4,0) {$(2,2)$};
    \node[mod] (33) at (2.8,0) {$(3,3)$};
    \node[mod] (44) at (4.2,0) {$(4,4)$};
    \node[mod] (55) at (5.6,0) {$(5,5)$};
    \node[mod] (66) at (7,0) {$(6,6)$};
    \node[nct] (77) at (8.4,0) {$(7,7)$};
    
    \draw[loosely dotted] (11.east) -- (22);
    \draw[loosely dotted] (22.east) -- (33);
    \draw[loosely dotted] (33.east) -- (44);
    \draw[loosely dotted] (44.east) -- (55);
    \draw[loosely dotted] (55.east) -- (66);
    \draw[loosely dotted] (66.east) -- (77);
        
    \node[nct] (12) at (0.7,0.7) {$(1,2)$};    
    \node[mod] (23) at (2.1,0.7) {$(2,3)$};    
    \node[mod] (34) at (3.5,0.7) {$(3,4)$};    
    \node[mod] (45) at (4.9,0.7) {$(4,5)$};    
    \node[mod] (56) at (6.3,0.7) {$(5,6)$};    
    \node[nct] (67) at (7.7,0.7) {$(6,7)$};  

    \draw[-{Stealth[scale=0.5]}] (11) -- (12);
    \draw[-{Stealth[scale=0.5]}] (22) -- (23);
    \draw[-{Stealth[scale=0.5]}] (33) -- (34);
    \draw[-{Stealth[scale=0.5]}] (44) -- (45);
    \draw[-{Stealth[scale=0.5]}] (55) -- (56);
    \draw[-{Stealth[scale=0.5]}] (66) -- (67);
    
    \draw[-{Stealth[scale=0.5]}] (12) -- (22);
    \draw[-{Stealth[scale=0.5]}] (23) -- (33);
    \draw[-{Stealth[scale=0.5]}] (34) -- (44);
    \draw[-{Stealth[scale=0.5]}] (45) -- (55);
    \draw[-{Stealth[scale=0.5]}] (56) -- (66);
    \draw[-{Stealth[scale=0.5]}] (67) -- (77);

    \draw[loosely dotted] (12.east) -- (23);
    \draw[loosely dotted] (23.east) -- (34);
    \draw[loosely dotted] (34.east) -- (45);
    \draw[loosely dotted] (45.east) -- (56);
    \draw[loosely dotted] (56.east) -- (67);
    
    \node[nct] (13) at (1.4,1.4) {$(1,3)$};
    \node[nct] (24) at (2.8,1.4) {$(2,4)$};
    \node[nct] (35) at (4.2,1.4) {$(3,5)$};
    \node[nct] (46) at (5.6,1.4) {$(4,6)$};
    \node[nct] (57) at (7,1.4) {$(5,7)$};
    
    \draw[-{Stealth[scale=0.5]}] (12) -- (13);
    \draw[-{Stealth[scale=0.5]}] (23) -- (24);
    \draw[-{Stealth[scale=0.5]}] (34) -- (35);
    \draw[-{Stealth[scale=0.5]}] (45) -- (46);
    \draw[-{Stealth[scale=0.5]}] (56) -- (57);
    
    \draw[-{Stealth[scale=0.5]}] (13) -- (23);
    \draw[-{Stealth[scale=0.5]}] (24) -- (34);
    \draw[-{Stealth[scale=0.5]}] (35) -- (45);
    \draw[-{Stealth[scale=0.5]}] (46) -- (56);
    \draw[-{Stealth[scale=0.5]}] (57) -- (67);
    
    \node[mod] (88) at (9.8,0) {$(8,8)$};
    \node[mod] (99) at (11.2,0) {$(9,9)$};
    \node[mod] (1010) at (12.6,0) {$(10,10)$};
    \node[nct] (1111) at (14,0) {$(11,11)$};
    \node[mod] (1212) at (15.4,0) {$(12,12)$};
    \node[mod] (1313) at (16.8,0) {$(13,13)$};
    \node[mod] (1414) at (18.2,0) {$(14,14)$};
    \node[nct] (1515) at (19.6,0) {$(15,15)$\nospacepunct{,}};
    
    \draw[loosely dotted] (77.east) -- (88);
    \draw[loosely dotted] (88.east) -- (99);
    \draw[loosely dotted] (99.east) -- (1010);
    \draw[loosely dotted] (1010.east) -- (1111);
    \draw[loosely dotted] (1111.east) -- (1212);
    \draw[loosely dotted] (1212.east) -- (1313);
    \draw[loosely dotted] (1313.east) -- (1414);
    \draw[loosely dotted] (1414.east) -- (1515);
    
    \node[nct] (78) at (9.1,0.7) {$(7,8)$};
    \node[nct] (89) at (10.5,0.7) {$(8,9)$};    
    \node[nct] (910) at (11.9,0.7) {$(9,10)$};    
    \node[nct] (1011) at (13.3,0.7) {$(10,11)$};    
    \node[nct] (1112) at (14.7,0.7) {$(11,12)$};    
    \node[nct] (1213) at (16.1,0.7) {$(12,13)$};    
    \node[nct] (1314) at (17.5,0.7) {$(13,14)$};    
    \node[nct] (1415) at (18.9,0.7) {$(14,15)$};   

    \draw[-{Stealth[scale=0.5]}] (77) -- (78);
    \draw[-{Stealth[scale=0.5]}] (88) -- (89);
    \draw[-{Stealth[scale=0.5]}] (99) -- (910);
    \draw[-{Stealth[scale=0.5]}] (1010) -- (1011);
    \draw[-{Stealth[scale=0.5]}] (1111) -- (1112);
    \draw[-{Stealth[scale=0.5]}] (1212) -- (1213);
    \draw[-{Stealth[scale=0.5]}] (1313) -- (1314);
    \draw[-{Stealth[scale=0.5]}] (1414) -- (1415);
    
    \draw[-{Stealth[scale=0.5]}] (78) -- (88);
    \draw[-{Stealth[scale=0.5]}] (89) -- (99);
    \draw[-{Stealth[scale=0.5]}] (910) -- (1010);
    \draw[-{Stealth[scale=0.5]}] (1011) -- (1111);
    \draw[-{Stealth[scale=0.5]}] (1112) -- (1212);
    \draw[-{Stealth[scale=0.5]}] (1213) -- (1313);
    \draw[-{Stealth[scale=0.5]}] (1314) -- (1414);
    \draw[-{Stealth[scale=0.5]}] (1415) -- (1515);
    \end{tikzpicture}}
    \]
where the rectangles indicate the $4\ZZ$-cluster tilting subcategory $\cC_{\La}$.

Notice that $\gldim(\La_1)=4$ while $\gldim(\La_2)=8$. Hence, following Remark \ref{rem:uniqueness of decomposition in acyclic}, we can write $\La_2$ as a gluing of homogeneous acyclic Nakayama algebras of global dimension equal to $4$ which admit $4\ZZ$-cluster tilting subcategories. Indeed, if $\La_2'=\K A_{[7,11]}/R^2$ and $\La_2''=\K A_{[11,15]}/R^2$, then $\La_2=\La_2'\glue \La_2''$ and $\La_2'$ and $\La_2''$ both admit $4\ZZ$-cluster tilting subcategories. In this way we have $\La = \La_1 \glue \La_2' \glue \La_2''$ and this is the unique decomposition of $\La$ in acyclic homogeneous Nakayama algebras of global dimension equal to $4$ such that each of them admits a $4\ZZ$-cluster tilting subcategory. Also note that $\gldim(\La) = 12$.
\end{example}

\subsection{Non-homogeneous relations: cyclic case}

To give the classification in this case, we first need to recall the notion of self-gluing from \cite{Vas2}.

\begin{definition}\label{def:self-gluing of acyclic Nakayama}
Let $m\ge 2$. Note that the arrows in $A_{[0,m]}$ and $\tilde{A}_{m}$ have the same labels. The correspondence given by this labelling extends to a multiplicative linear map $R_{A_{[0,m]}} \to R_{\tilde{A}_{m}}$, which induces a bijection $R_{A_{[0,m]}} \to R_{\tilde{A}_{m}}/\langle \alpha_m\alpha_1\rangle$. Thus every admissible ideal $I \subseteq \K A_{[0,m]}$ corresponds bijectively to an admissible ideal $\tilde{I} \subseteq \K \tilde{A}_{m}$ such that $\alpha_m\alpha_1 \in \tilde{I}$. We define the \emph{self-gluing of $\La = \K A_{[0,m]}/I$} to be the cyclic Nakayama algebra $\tilde{\La}=\K \tilde{A}_{m}/\tilde{I}$, where $\tilde{I}$ is the ideal generated by $I\cup\{\alpha_m\alpha_1\}$. 
\end{definition}

We immediately have the following lemma.

\begin{lemma}\label{lem:ungluing at a simple, cyclic}
Let $m\geq 2$ and $\tilde{\La}$ be a cyclic Nakayama algebra with $m$ vertices. Assume that $M(-1,1)\not\in\mo{\tilde{\La}}$. Then $\tilde{\La}$ is the self-gluing of an acyclic Nakayama algebra $\La=\K A_{[0,m]}/I$, i.e., $\tilde{\La}=\K \tilde{A}_{m}/\tilde{I}$.
\end{lemma}

\begin{proof}
The condition $M(-1,1)\not\in\mo{\tilde{\La}}$ means precisely $\alpha_{m}\alpha_{1} = 0$ in $\tilde{\La}$, and so the claim follows directly from Definition~\ref{def:self-gluing of acyclic Nakayama}.
\end{proof}

Definition~\ref{def:self-gluing of acyclic Nakayama} is a special case of the methods described in \cite[Section 5.2]{Vas2}. In this case, no full and faithful embedding exists between the module category of an acyclic Nakayama algebra and its self-gluing. However, we can still compute syzygies, cosyzygies and Auslander--Reiten translations in $\mo{\tilde{\La}}$ using the corresponding concepts in $\mo{\La}$. 

In the following, for an indecomposable $\La$-module $M(i,j)$ we write $\tilde{M}(i,j)$ or $M(i,j)^{\sim}$ for the corresponding $\tilde{\La}$-module. This is useful for distinguishing the modules in $\mo\La$ and $\mo{\tilde{\La}}$.

\begin{proposition}\label{prop:basic self-gluing results}
Let $m\geq 2$ and $\La=\K A_{[0,m]}/I$ be an acyclic Nakayama algebra. Let $\tilde{\La}$ be the self-gluing of $\La$. Then
	\begin{enumerate}[label=(\alph*)]
		\item The set
		\[\ind(\tilde{\La}) \coloneqq \{ \tilde{M}(i,j) \mid M(i,j)\in \ind(\La) \text{ and } 0\leq i \leq m\}\]
		is a complete and irredundant set of representatives of isomorphism classes of indecomposable $\tilde{\La}$-modules.
	\end{enumerate}
	If moreover $i\in\ZZ$ with $0\leq i \leq m$, then
	\begin{enumerate}
		\item[(b)] $\tilde{M}(i,j)$ is a projective $\tilde{\La}$-module if and only if $M(i,j)$ is a projective $\La$-module different from $M(0,0)$.
		\item[(c)] $\tilde{M}(i,j)$ is an injective $\tilde{\La}$-module if and only if $M(i,j)$ is an injective $\La$-module different from $M(m,m)$.
		\item[(d)] If $M(i,j)\in\mo{\La}$ with $M(i,j)\neq M(0,0)$, then $\tau(\tilde{M}(i,j))= \left(\tau(M(i,j))\right)^{\sim}$ and $\Omega(\tilde{M}(i,j))=\left(\Omega(M(i,j))\right)^{\sim}$. 
		\item[(e)] If $M(i,j)\in\mo{\La}$ with $M(i,j)\neq M(m,m)$, then $\tau^{-}(\tilde{M}(i,j))=\left(\tau^{-}(M(i,j))\right)^{\sim}$ and $\Omega^{-}(\tilde{M}(i,j))= \left(\Omega^{-}(M(i,j))\right)^{\sim}$.
	\end{enumerate}
\end{proposition}

\begin{proof}
	Similar to the proof of Proposition~\ref{prop:basic gluing results}; we refer to \cite[Section 5.2]{Vas2} for more details. 
\end{proof}

\begin{remark}\label{rem:the extra computations in self-gluing}
	\begin{enumerate}[label=(\alph*)]
		\item Proposition~\ref{prop:basic self-gluing results}(a) implies that the set $\ind(\tilde{\La})$ has one less element than the set $\ind(\La)$. This is due to the fact that $\tilde{M}(0,0)=\tilde{M}(m,m)$ but $M(0,0)\neq M(m,m)$.
		\item Using Proposition~\ref{prop:basic self-gluing results}(d) we can also compute $\tau(\tilde{M}(0,0))=\tau(\tilde{M}(m,m))=\tilde{M}(m-1,m-1)$ and $\Omega(\tilde{M}(0,0))=\Omega(\tilde{M}(m,m))=\left(\Omega(M(m,m))\right)^{\sim}$.
		\item Using Proposition~\ref{prop:basic self-gluing results}(e) we can also compute $\tau^-(\tilde{M}(m,m))=\tau^-(M(0,0))=\tilde{M}(1,1)$ and $\Omega^-(\tilde{M}(m,m))=\Omega^-(\tilde{M}(0,0))=\left(\Omega^-(M(1,1))\right)^{\sim}$.
	\end{enumerate}
\end{remark}

With this we can show the following results about $n\ZZ$-cluster tilting subcategories of self-gluings of acyclic Nakayama algebras.

\begin{proposition}\label{prop:self-gluing of nZ-cluster tilting is nZ-cluster tilting}
Let $m\geq 2$ and $\La=\K A_{[0,m]}/I$ be an acyclic Nakayama algebra. Let $\tilde{\La}$ be the self-gluing of $\La$. If $\La$ admits an $n\ZZ$-cluster tilting subcategory $\mathcal{C}_{\La}$, then $\tilde{\La}$ admits an $n\ZZ$-cluster tilting subcategory $\cC_{\tilde{\La}}$ where
	\[
	\cC_{\tilde{\La}} = \add\{\tilde{M}(i,j) \mid  M(i,j) \in \cC_{\La} \text{ and } 0\leq i\leq m\}
	\]
\end{proposition}

\begin{proof}
	That $\cC_{\tilde{\La}}$ is an $n$-cluster tilting subcategory of $\mo{\tilde{\La}}$ follows by \cite[Corollary 5.13]{Vas2}; see also \cite[Corollary 6.9]{Vas2}. To show that $\cC_{\tilde{\La}}$ is $n\ZZ$-cluster tilting, we can use Proposition~\ref{prop:basic self-gluing results} and follow the proof of Proposition~\ref{prop:gluing of nZ-cluster tilting is nZ-cluster tilting}. We leave the details to the reader.
\end{proof}

\begin{proposition}\label{prop:ungluing of nZ-cluster tilting at simple, cyclic}
Let $m\geq 2$ and $\La=\K A_{[0,m]}/I$ be an acyclic Nakayama algebra. Let $\tilde{\La}$ be the self-gluing of $\La$. If $\tilde{\La}$ admits an $n\ZZ$-cluster tilting subcategory $\cC_{\tilde{\La}}$ such that $\tilde{M}(m,m)\in \cC_{\tilde{\La}}$, then 
	\[
	\cC_{\La} = \add\{ M(i,j) \mid \tilde{M}(i,j)\in\cC_{\tilde{\La}} \text{ and } 0\leq i\leq m\} 
	\]
	is an $n\ZZ$-cluster tilting subcategory of $\mo{\La}$.
\end{proposition}

\begin{proof}
	To show that $\cC_{\La}$ is $n\ZZ$-cluster tilting it is enough to show that the statements (a)--(d) in Proposition~\ref{prop:basic nZ-cluster tilting results} and statement (b) in Proposition~\ref{prop:basic nZ-cluster tilting results2} hold. These can be easily verified using Proposition~\ref{prop:basic self-gluing results} and Remark~\ref{rem:the extra computations in self-gluing}, similarly to the proof of Proposition~\ref{prop:ungluing of nZ-cluster tilting at simple, acyclic}. 
\end{proof}

We can now give the main result for this section.

\begin{theorem}\label{thrm:cyclic non-homogeneous case}
	Let $\tilde{\La}$ be a cyclic Nakayama algebra and assume that $\tilde{\La}$ is not homogeneous. Then $\tilde{\La}$ admits an $n\ZZ$-cluster tilting subcategory $\cC_{\tilde{\La}}$ if and only if $\tilde{\La}$ is the self-gluing of an acyclic Nakayama algebra $\La$ which admits an $n\ZZ$-cluster tilting subcategory $\cC_{\La}$. In this case, $\La$ is piecewise homogeneous, the $n\ZZ$-cluster tilting subcategory $\cC_{\tilde{\La}}$ is unique and $\cC_{\tilde{\La}} = \add\{\tilde{M}(i,j) \mid M(i,j) \in \cC_{\La} \text{ and } 0\leq i\leq m\}$.
\end{theorem}

\begin{proof}
	Let $\tilde{\La}=\K \tilde{A}_{m}/\tilde{I}$ for some $m\geq 1$. If $m=1$, then $\tilde{\La}$ is necessarily homogeneous. Hence we have that $m\geq 2$.
	
	Assume first that $\tilde{\La}$ is the self-gluing of an acyclic Nakayama algebra $\La$ which admits an $n\ZZ$-cluster tilting subcategory. Then $\tilde{\La}$ admits an $n\ZZ$-cluster tilting subcategory $\cC_{\tilde{\La}}$ by Proposition~\ref{prop:self-gluing of nZ-cluster tilting is nZ-cluster tilting}. 
	
	Next assume that $\tilde{\La}$ admits an $n\ZZ$-cluster tilting subcategory. By Corollary \ref{cor:not homogeneous implies ungluing simple} there exists a simple module $M(i,i)\in \mo{\tilde{\La}}$, such that $M(i-1,i+1)\not\in\mo{\La}$ and such that $M(i,i)$ belongs to any $n\ZZ$-cluster tilting subcategory of $\mo{\tilde{\La}}$. By possibly relabeling the vertices of $\tilde{A}_m$, we may assume that $i=m$. By Lemma~\ref{lem:ungluing at a simple, cyclic} it then follows that $\tilde{\La}$ is the self-gluing of an acyclic Nakayama algebra $\La$. Hence, there is a bijection between the $n\ZZ$-cluster tilting subcategories of $\mo{\tilde{\La}}$ and $\mo{\La}$ by Proposition~\ref{prop:self-gluing of nZ-cluster tilting is nZ-cluster tilting} and Proposition~\ref{prop:ungluing of nZ-cluster tilting at simple, cyclic}. Since there is a unique $n\ZZ$-cluster tilting subcategory $\cC_{\La}$ in $\mo{\La}$ by Corollary~\ref{cor:unique n-ct for representation-directed}, it follows that $\mo{\tilde{\La}}$ has a unique $n\ZZ$-cluster tilting subcategory $\cC_{\tilde{\La}}$. Finally, by Proposition~\ref{prop:self-gluing of nZ-cluster tilting is nZ-cluster tilting} we get that $\cC_{\tilde{\La}}=\add\{\tilde{M}(i,j)\mid M(i,j)\in \cC_{\La} \text{ and } 0\leq i\leq m\}$.
\end{proof}

\begin{remark}\label{rem:uniqueness of decomposition in cyclic}
	\begin{enumerate}[label=(\alph*)]
		\item Consider $\tilde{\La}$ as in in Theorem~\ref{thrm:cyclic non-homogeneous case}. Similarly to Remark~\ref{rem:uniqueness of decomposition in acyclic}(c) one may compute that the simples in $\cC_{\tilde{\La}}$ form an orbit under $\Omega^{-n}$. In particular, $\gldim \tilde{\La} = \infty$. In fact, we shall see in Corollary~\ref{cor:not Iwanaga--Gorenstein} that $\tilde{\La}$ is not Iwanaga--Gorenstein.
		\item The algebra $\La$ in Theorem~\ref{thrm:cyclic non-homogeneous case} is not necessarily unique, as can be seen in the following example.
	\end{enumerate}

\end{remark}

\begin{example}\label{ex:non-homogeneous cyclic case}
Let $\La_1$, $\La_2$ and $\La=\La_1\glue \La_2$ be as in Example~\ref{ex:non-homogeneous acyclic case}. Let $\tilde{\La}$ be the self-gluing of $\La$. Then $\La$ admits a $4\ZZ$-cluster tilting subcategory and hence $\tilde{\La}$ also admits a $4\ZZ$-cluster tilting subcategory by Theorem \ref{thrm:cyclic non-homogeneous case}. Indeed, the Auslander--Reiten quiver $\Gamma(\tilde{\La})$ of $\tilde{\La}$ is
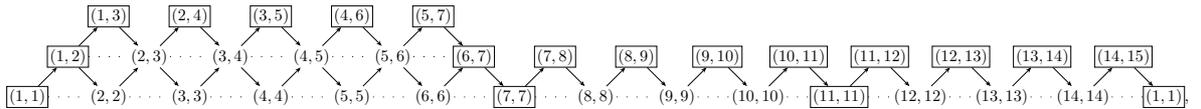
\[\resizebox {\columnwidth} {!} {
    \begin{tikzpicture}[scale=0.9, transform shape, baseline={(current bounding box.center)}]
    
    \tikzstyle{mod}=[rectangle, minimum width=6pt, draw=none, inner sep=1.5pt, scale=0.8]
    \tikzstyle{nct}=[rectangle, minimum width=3pt, draw, inner sep=1.5pt, scale=0.8]
    
    \node[nct] (11) at (0,0) {$(1,1)$};
    \node[mod] (22) at (1.4,0) {$(2,2)$};
    \node[mod] (33) at (2.8,0) {$(3,3)$};
    \node[mod] (44) at (4.2,0) {$(4,4)$};
    \node[mod] (55) at (5.6,0) {$(5,5)$};
    \node[mod] (66) at (7,0) {$(6,6)$};
    \node[nct] (77) at (8.4,0) {$(7,7)$};
    
    \draw[loosely dotted] (11.east) -- (22);
    \draw[loosely dotted] (22.east) -- (33);
    \draw[loosely dotted] (33.east) -- (44);
    \draw[loosely dotted] (44.east) -- (55);
    \draw[loosely dotted] (55.east) -- (66);
    \draw[loosely dotted] (66.east) -- (77);
        
    \node[nct] (12) at (0.7,0.7) {$(1,2)$};    
    \node[mod] (23) at (2.1,0.7) {$(2,3)$};    
    \node[mod] (34) at (3.5,0.7) {$(3,4)$};    
    \node[mod] (45) at (4.9,0.7) {$(4,5)$};    
    \node[mod] (56) at (6.3,0.7) {$(5,6)$};    
    \node[nct] (67) at (7.7,0.7) {$(6,7)$};  

    \draw[-{Stealth[scale=0.5]}] (11) -- (12);
    \draw[-{Stealth[scale=0.5]}] (22) -- (23);
    \draw[-{Stealth[scale=0.5]}] (33) -- (34);
    \draw[-{Stealth[scale=0.5]}] (44) -- (45);
    \draw[-{Stealth[scale=0.5]}] (55) -- (56);
    \draw[-{Stealth[scale=0.5]}] (66) -- (67);
    
    \draw[-{Stealth[scale=0.5]}] (12) -- (22);
    \draw[-{Stealth[scale=0.5]}] (23) -- (33);
    \draw[-{Stealth[scale=0.5]}] (34) -- (44);
    \draw[-{Stealth[scale=0.5]}] (45) -- (55);
    \draw[-{Stealth[scale=0.5]}] (56) -- (66);
    \draw[-{Stealth[scale=0.5]}] (67) -- (77);

    \draw[loosely dotted] (12.east) -- (23);
    \draw[loosely dotted] (23.east) -- (34);
    \draw[loosely dotted] (34.east) -- (45);
    \draw[loosely dotted] (45.east) -- (56);
    \draw[loosely dotted] (56.east) -- (67);
    
    \node[nct] (13) at (1.4,1.4) {$(1,3)$};
    \node[nct] (24) at (2.8,1.4) {$(2,4)$};
    \node[nct] (35) at (4.2,1.4) {$(3,5)$};
    \node[nct] (46) at (5.6,1.4) {$(4,6)$};
    \node[nct] (57) at (7,1.4) {$(5,7)$};
    
    \draw[-{Stealth[scale=0.5]}] (12) -- (13);
    \draw[-{Stealth[scale=0.5]}] (23) -- (24);
    \draw[-{Stealth[scale=0.5]}] (34) -- (35);
    \draw[-{Stealth[scale=0.5]}] (45) -- (46);
    \draw[-{Stealth[scale=0.5]}] (56) -- (57);
    
    \draw[-{Stealth[scale=0.5]}] (13) -- (23);
    \draw[-{Stealth[scale=0.5]}] (24) -- (34);
    \draw[-{Stealth[scale=0.5]}] (35) -- (45);
    \draw[-{Stealth[scale=0.5]}] (46) -- (56);
    \draw[-{Stealth[scale=0.5]}] (57) -- (67);
    
    \node[mod] (88) at (9.8,0) {$(8,8)$};
    \node[mod] (99) at (11.2,0) {$(9,9)$};
    \node[mod] (1010) at (12.6,0) {$(10,10)$};
    \node[nct] (1111) at (14,0) {$(11,11)$};
    \node[mod] (1212) at (15.4,0) {$(12,12)$};
    \node[mod] (1313) at (16.8,0) {$(13,13)$};
    \node[mod] (1414) at (18.2,0) {$(14,14)$};
    \node[nct] (11b) at (19.6,0) {$(1,1)$\nospacepunct{,}};
    
    \draw[loosely dotted] (77.east) -- (88);
    \draw[loosely dotted] (88.east) -- (99);
    \draw[loosely dotted] (99.east) -- (1010);
    \draw[loosely dotted] (1010.east) -- (1111);
    \draw[loosely dotted] (1111.east) -- (1212);
    \draw[loosely dotted] (1212.east) -- (1313);
    \draw[loosely dotted] (1313.east) -- (1414);
    \draw[loosely dotted] (1414.east) -- (11b);
    
    \node[nct] (78) at (9.1,0.7) {$(7,8)$};
    \node[nct] (89) at (10.5,0.7) {$(8,9)$};    
    \node[nct] (910) at (11.9,0.7) {$(9,10)$};    
    \node[nct] (1011) at (13.3,0.7) {$(10,11)$};    
    \node[nct] (1112) at (14.7,0.7) {$(11,12)$};    
    \node[nct] (1213) at (16.1,0.7) {$(12,13)$};    
    \node[nct] (1314) at (17.5,0.7) {$(13,14)$};    
    \node[nct] (1415) at (18.9,0.7) {$(14,15)$};   

    \draw[-{Stealth[scale=0.5]}] (77) -- (78);
    \draw[-{Stealth[scale=0.5]}] (88) -- (89);
    \draw[-{Stealth[scale=0.5]}] (99) -- (910);
    \draw[-{Stealth[scale=0.5]}] (1010) -- (1011);
    \draw[-{Stealth[scale=0.5]}] (1111) -- (1112);
    \draw[-{Stealth[scale=0.5]}] (1212) -- (1213);
    \draw[-{Stealth[scale=0.5]}] (1313) -- (1314);
    \draw[-{Stealth[scale=0.5]}] (1414) -- (1415);
    
    \draw[-{Stealth[scale=0.5]}] (78) -- (88);
    \draw[-{Stealth[scale=0.5]}] (89) -- (99);
    \draw[-{Stealth[scale=0.5]}] (910) -- (1010);
    \draw[-{Stealth[scale=0.5]}] (1011) -- (1111);
    \draw[-{Stealth[scale=0.5]}] (1112) -- (1212);
    \draw[-{Stealth[scale=0.5]}] (1213) -- (1313);
    \draw[-{Stealth[scale=0.5]}] (1314) -- (1414);
    \draw[-{Stealth[scale=0.5]}] (1415) -- (11b);
    \end{tikzpicture}}
    \]
where $(1,1)$ is drawn twice and the rectangles indicate the $4\ZZ$-cluster tilting subcategory $\cC_{\tilde{\La}}$.    

Now let $\La_2'$ and $\La_2''$ be as in Example~\ref{ex:non-homogeneous acyclic case}. Then it is easy to see that $\tilde{\La}$ is also the self-gluing of $\La'=\La_2''\glue \La_1 \glue \La_2'$, up to relabelling the vertices, but $\La\not\isom\La'$.
\end{example}

\section{Singularity categories}\label{sec:singularity categories}
It was shown in \cite{Kva} that if a finite-dimensional algebra $\La$ admits an $n\ZZ$-cluster tilting subcategory, then the singularity category of $\La$ admits an $n\ZZ$-cluster tilting subcategory. Our goal in this section is to describe the singularity category and its $n\ZZ$-cluster tilting subcategories for Nakayama algebras admitting an $n\ZZ$-cluster tilting subcategory. Furthermore, we describe the canonical functor from the module category to the singularity category, and its restriction to the $n\ZZ$-cluster tilting subcategories. To do this we rely heavily on results from \cite{Shen}, where the author studies the singularity category for general Nakayama algebras.

\subsection{Cluster tilting in singularity categories}
We start by recalling the definition of cluster tilting in triangulated categories.
\begin{definition}
Let $\cT$ be a triangulated category. A  subcategory $\cC$ of $\cT$ is called $n$\emph{-cluster tilting} if it is functorially finite and 
\begin{align*}
\cC&=\{T\in \cT \mid \Hom_{\cT}(\cC,T[i])=0 \text{ for all $0<i<n$}\}\\
&=\{T\in\cT \mid \Hom_{\cT}(T,\cC[i])=0 \text{ for all $0<i<n$}\}.
\end{align*}
If moreover $\Hom_{\cT}(\cC,\cC[i])\neq 0$ implies that $i\in n\ZZ$, then we call $\cC$ an \emph{$n\ZZ$-cluster tilting subcategory}.
\end{definition}
Analogously to module categories, we have that an $n$-cluster tilting subcategory $\cC$ of $\cT$ is $n\ZZ$-cluster tilting if and only if $\cC[n]\subseteq \cC$.

Now fix a finite-dimensional algebra $\La$, and let $D^b(\mo \La)$ denote the bounded derived category of finitely generated $\La$-modules. Furthermore, let $\operatorname{perf}\La\subset D^b(\mo \La)$ denote the subcategory of perfect complexes, i.e. complexes which are quasi-isomorphic to bounded complexes of finitely generated projective $\La$-modules. The singularity category of $\La$, denoted $D_{\operatorname{sing}}(\La)$, is the triangulated category defined by the Verdier quotient
\[
D_{\operatorname{sing}}(\La):= D^b(\mo \La)/\operatorname{perf}\La.
\]
Note that we have a canonical functor 
\[\mo \La\to D_{\operatorname{sing}}(\La)
\] given by the composite
$\mo \La\to D^b(\mo \La)\to D_{\operatorname{sing}}(\La)$ where $\mo \La\to D^b(\mo \La)$ sends a module $M$ to the stalk complex which is equal to $M$ in degree $0$, and where $D^b(\mo \La)\to D_{\operatorname{sing}}(\La)$ is the canonical localization functor. By abuse of notation we write $M$ both for a module in $\mo \La$, and for its image in $D_{\operatorname{sing}}(\La)$.

 The following result relates $n\ZZ$-cluster tilting subcategories of $\mo \La$ and $D_{\operatorname{sing}}(\La)$, and motivates our study of the singularity category.

\begin{theorem}\label{thrm:nZ-CTSingCat}
Let $\La$ be a finite-dimensional algebra and let $\cC$ be an $n\ZZ$-cluster tilting subcategory of $\mo \La$. Then the subcategory 
\[
\{X\in D_{\operatorname{sing}}(\La)\mid X\cong  C[ni] \text{ for some }C\in \cC \text{ and }i\in \ZZ\}
\]
is an $n\ZZ$-cluster tilting subcategory of $D_{\operatorname{sing}}(\La)$.
\end{theorem}

\begin{proof}
This follows from \cite[Theorem 1.2]{Kva}.
\end{proof}

\subsection{Singularity categories of Nakayama algebras}\label{subsection:SingCatNakayama}
 We recall the description of the singularity category of a Nakayama algebra obtained in \cite{Shen}. To do this we recall the definition of the resolution quiver, which was first introduced in \cite{Rin}. 

\begin{definition}
Let $\La$ be a Nakayama algebra. The \emph{resolution quiver} $R(\La)$ of $\La$ is given as follows:  
\begin{itemize}
\item The vertices of $R(\La)$ correspond to isomorphism classes of simple $\La$-modules.
\item Let $i$ and $j$ be two vertices of $R(\La)$, corresponding to the simple $\La$-modules $S_i$ and $S_j$, respectively. Then there is an arrow $i\to j$ if and only if $S_j\cong \tau \soc P_i$, where $P_i$ is the projective cover of $S_i$.
\end{itemize}
A simple $\La$-module is called \emph{cyclic} if it is contained in a cycle of $R(\La)$.
\end{definition}  
Following \cite{Shen}, we let $\mathcal{S}_c$ denote the class of simple cyclic $\La$-modules, and we let $\cX_c$ denote the class of indecomposable $\La$-modules $M$ for which $\topp M\in \mathcal{S}_c$ and $\tau \soc M\in \mathcal{S}_c$. Finally, we let $\cF=\add \cX_c$ denote the additive closure of $\cX_c$.

The following theorem is one of the main results in \cite{Shen}. Here we consider the canonical functor 
\[
\cF\to D_{\operatorname{sing}}(\La)
\]
 given by the composite $\cF\to \mo \La\to D_{\operatorname{sing}}(\La)$ where $\cF\to \mo \La$ is the inclusion functor.  

\begin{theorem}\label{thrm:shen}
Let $\La$ be a Nakayama algebra of infinite global dimension. The following hold:
\begin{enumerate}
\item $\cF$ is a wide subcategory of $\mo \La$, i.e. it is closed under extensions, kernels and cokernels.
\item $\cF$ is a Frobenius abelian category. Hence, the stable category $\underline{\cF}$ of the abelian category $\cF$ is triangulated.
\item The functor $\cF\to D_{\operatorname{sing}}(\La)$ kills the projective objects in $\cF$ and induces an equivalence 
\[
\underline{\cF}\cong D_{\operatorname{sing}}(\La)
\]
of triangulated categories.
\end{enumerate}
\end{theorem}
\begin{proof}
This follows from \cite[Proposition 3.5]{Shen}, \cite[Proposition 3.8]{Shen} and \cite[Theorem 3.11]{Shen}.
\end{proof}
 Hence, to describe $D_{\operatorname{sing}}(\La)$, it suffices to determine the objects in $\cF$.
 
\subsection{Singularity categories of Nakayama algebras admitting \texorpdfstring{$n\ZZ$}{nZ}-cluster tilting categories}

Let $\tilde{\La}$ be a Nakayama algebra admitting an $n\ZZ$-cluster tilting subcategory. We want to describe $D_{\operatorname{sing}}(\tilde{\La})$ and its $n\ZZ$-cluster tilting subcategories. If $\tilde{\La}$ is acyclic, then $\tilde{\La}$ has finite global dimension, and hence $D_{\operatorname{sing}}(\tilde{\La})\cong 0$ is trivial. If $\tilde{\La}$ is a cyclic homogeneous Nakayama algebra, then $\tilde{\La}$ is self-injective by Corollary \ref{cor:SelfinjNakayama}, and hence we have an equivalence 
\[
\smo {\tilde{\La}} \cong D_{\operatorname{sing}}(\tilde{\La}).
\]
by Buchweitz theorem \cite{Buc}. It follows that the singularity category of $\tilde{\La}$ is easy to describe in this case. For example, it is well know that there is a bijection between $n\ZZ$-cluster tilting subcategories of $\mo {\tilde{\La}}$ and $\smo {\tilde{\La}}$.

 The remaining case we need to consider is when $\tilde{\La}$ is a cyclic non-homogeneous Nakayama algebra with an $n\ZZ$-cluster tilting subcategory. By Theorem \ref{thrm:acyclic non-homogeneous case}, Remark \ref{rem:uniqueness of decomposition in acyclic} (c), and Theorem~\ref{thrm:cyclic non-homogeneous case} it follows that $\tilde{\La}$ is the self-gluing of an acyclic Nakayama algebra of the form
 \[\La = \La_1 \glue \cdots \glue\La_r,\]
 where $\La_k=\K A_{[m_k,m_{k+1}]}/R^{l_k}$ for $1\leq k\leq r$ is a homogeneous acyclic Nakayama algebra which admits an $n\ZZ$-cluster tilting subcategory and satisfies $\gldim(\La_k)=n$. Fix such algebras $\La$ and $\La_k$ for the remainder of this section. For convenience we also set $l_{r+1}:=l_1$. We make the following observations:
 
 \begin{itemize}
 	\item By Remark~\ref{rem:uniqueness of decomposition in cyclic}, $\gldim\tilde{\La} = \infty$ so Theorem~\ref{thrm:shen} applies.
 	\item If $l_k\geq 3$, then $l_k\divides m_{k+1}-m_k$ and $n=2\frac{m_{k+1}-m_k}{l_k}$ by Proposition \ref{prop:homogeneous acyclic Nakayama with nZ-ct} (b). Since $\tilde{\La}$ is not homogeneous, there must exist a $k$ with $l_k\geq 3$, and therefore $n$ must be even.
 	\item We claim that $l_k\divides m_{k+1}-m_k$ and $n=2\frac{m_{k+1}-m_k}{l_k}$ also hold when $l_k=2$. Indeed, we have  \[
 	n=m_{k+1}-m_k=2\frac{m_{k+1}-m_k}{l_k}
 	\]
 	 since $\gldim(\La_k)=n$. This together with the fact that $n$ is even implies that $l_k\divides m_{k+1}-m_k$.

 \end{itemize}

We start by computing the resolution quiver of the algebras $\La_k$.

\begin{lemma}\label{lem:ResolutionQuiverAcyclicNakayama}
	Let $Q^k$ be the quiver given by the disjoint union    \[
	Q^k=\bigcup_{j=0}^{l_k-1}\tilde{Q}^{k,j}
	\] where $\tilde{Q}^{k,j}$ is a linearly oriented $A_{\frac{n}{2}}$ quiver when $j\neq 0$ and a linearly oriented $A_{\frac{n}{2}+1}$ quiver when $j=0$. Then $Q^k$ is the resolution quiver of $\La_k$, where the $i$th vertex of $\tilde{Q}^{k,j}$ corresponds to the simple $\La_k$-module $M(m_k+(i-1)l_k+j,m_k+(i-1)l_k+j)$.
\end{lemma}

\begin{proof}
	We write a vertex of $Q^k$ as a pair $(i,j)$ where $0\leq j \leq l_k-1$ and $i$ is a vertex of $\tilde{Q}^{k,j}$. With this notation, it is clear that the association
	\[
	(i,j)\mapsto S_{m_k+(i-1)l_k+j}=M(m_k+(i-1)l_k+j,m_k+(i-1)l_k+j)
	\] gives a bijection between the vertices of $Q^k$ and the isomorphism classes of simple $\La_k$-modules. Hence, we only need to check that the arrows of $Q^k$ coincide with the arrows of the resolution quiver under this bijection. If $i>1$ then the projective cover of $S_{m_k+(i-1)l_k+j}$ is 
	\[
	P_{m_k+(i-1)l_k+j}=M(m_k+(i-2)l_k+j+1,m_k+(i-1)l_k+j)
	\] by Proposition \ref{prop:homogeneous Nakayama AR-quiver} (a) and Proposition \ref{prop:basic Nakayama results2} (b). It follows that 
	\begin{align*}
	\tau \soc P_{m_k+(i-1)l_k+j}&= \tau M(m_k+(i-2)l_k+j+1,m_k+(i-2)l_k+j+1) \\
	& = M(m_k+(i-2)l_k+j,m_k+(i-2)l_k+j) \\
	& = S_{m_k+(i-2)l_k+j}
	\end{align*}
	where the first equality follows from Proposition \ref{prop:basic Nakayama results2} (a) and the second equality from Proposition \ref{prop:AR quiver of Nakayama}. Since there is a unique arrow $(i,j)\to (i-1,j)$ in $Q^k$, we see that the arrows of $Q^k$ coincide with the arrows of the resolution quiver when $i>1$.  Finally, if $i=1$ then the projective cover of $S_{m_k+j}$ is 
	\[
	P_{m_k+j}=M(m_k,m_k+j)
	\]
	by Proposition \ref{prop:homogeneous Nakayama AR-quiver} (a) and Proposition \ref{prop:basic Nakayama results2} (b), and so we get
	\[
	\tau \soc P_{m_k+j}= \tau M(m_k,m_k) \cong 0 
	\]
	where the last isomorphism follows from the fact that $M(m_k,m_k)$ is projective. Since $(1,j)$ is a sink in $Q^k$, we conclude that $Q^k$ is the resolution quiver of $\La_k$.
\end{proof}

We now compute the resolution quiver of $\tilde{\La}$.  
\begin{lemma}\label{lem:Resolution quiver}
 Let $Q$ denote the quiver obtained as follows:
 \begin{enumerate}
 \item First take the disjoint union 
 \[
 \bigcup_{k=1}^r\bigcup_{j=0}^{l_k-1} Q^{k,j}
 \] 
 where $Q^{k,j}$ is a linearly oriented $A_{\frac{n}{2}}$-quiver for all $k$ and $j$.
 \item Then	add an arrow from the sink of $Q^{k+1,0}$ to the source of $Q^{k,0}$ for all $1\leq k\leq r$ (where $Q^{r+1,0}:=Q^{1,0}$).
 \item Finally, add an arrow from the sink of $Q^{k+1,j}$ to the source of $Q^{k,l_{k}-1}$ for all $1\leq k\leq r$ and $0<j<l_{k+1}$ (where $Q^{r+1,j}:=Q^{1,j}$).
 \end{enumerate} 
Then $Q$ is the resolution quiver of $\tilde{\La}$, where the $i$th vertex of $Q^{k,j}$ corresponds to the simple $\tilde{\La}$-module $\tilde{M}(m_k+(i-1)l_k+j,m_k+(i-1)l_k+j)$.
\end{lemma}

\begin{proof}
By Lemma \ref{lem:ResolutionQuiverAcyclicNakayama} we know the resolution quiver of $\La_k$. Hence, we only need to determine the effect on resolution quivers for gluing and self-gluing of acyclic Nakayama algebras. To this end, let $\La'$ and $\La''$ be two acyclic Nakayama algebras with resolution quivers $Q'$ and $Q''$, respectively. Then, by Proposition \ref{prop:basic gluing results} the resolution quiver of $\La' \glue \La''$ is given as follows:
\begin{enumerate}
\item First take the disjoint union $Q'\cup Q''$.
\item Then identify the unique vertex of $Q'$ corresponding to a simple injective module with the unique vertex of $Q''$ corresponding to a simple projective module.
\item Finally add an arrow from each non-projective sink vertex of $Q''$ to the vertex of $Q'$ corresponding to the $\tau$-translate of the simple injective $\La'$-module.
\end{enumerate}  
Similarly, by Proposition \ref{prop:basic self-gluing results} the resolution quiver of the self-gluing of $\La'$ is obtained as follows:
\begin{enumerate}
\item First identify the unique vertex of $Q'$ corresponding to a simple injective module with the unique vertex of $Q'$ corresponding to a simple projective module.
\item Then add an arrow from each non-projective sink vertex of $Q'$ to the vertex of $Q'$ corresponding to the $\tau$-translate of the simple injective module.
\end{enumerate} 

We want to apply these constructions to obtain the resolution quiver of the self-gluing $\tilde{\La}$ of $\La_1 \glue \cdots \glue\La_r$ from the resolution quivers $Q^k$ of the algebras $\La_k$. Similar to the proof of Lemma \ref{lem:ResolutionQuiverAcyclicNakayama}, we let a vertex in $Q^k$ be denoted by a pair $(i,j)$ where $0\leq j\leq l_k-1$ and $i$ is a vertex of $\tilde{Q}^{k,j}$. With this notation we see that the unique vertex of $Q^k$ corresponding to the simple projective module $M(m_k,m_k)$ is $(1,0)$, and the unique vertex corresponding to the simple injective module $M(m_{k+1},m_{k+1})$ is $(\frac{n}{2}+1,0)$. Using these observations and the description of the resolution quiver of gluings and self-gluings given above, we get that the resolution quiver of $\tilde{\La}$ is obtained as follows:
\begin{enumerate}
\item First take the disjoint union $\bigcup_{k=1}^rQ^k$.
\item Then identify the vertex $(1,0)$ of $Q^{k+1}$ with the vertex $(\frac{n}{2}+1,0)$ of $Q^{k}$ for each $1\leq k\leq r$ (where $Q^{r+1}:=Q^1$).
\item Finally, add an arrow from the vertex $(1,j)$ of $Q^{k+1}$ to the vertex $(\frac{n}{2},l_k-1)$ of $Q^k$ for each $1\leq k\leq r$ and $0<j<l_{k+1}$. 
\end{enumerate}  
This is precisely the quiver $Q$ described in the lemma, which proves the claim.
\end{proof}

\begin{example}\label{ex:resolution quiver}
Let $\La=\La_1\glue\La_2\glue\La_3$ where $\La_1=\K A_{[1,7]}/R^3$, $\La_2=\K A_{[7,11]}/R^2$ and $\La_3=\K A_{[11,15]}/R^2$ (see also Example~\ref{ex:non-homogeneous acyclic case}). Let $\tilde{\La}$ be the self-gluing of $\La$. Then the resolution quiver $\tilde{Q}$ of $\tilde{\La}$ is the quiver
\[
\begin{tikzpicture}[scale=1, transform shape, baseline={(current bounding box.center)}]
    
    \node (S13) at (0,1) {$S_{13}$};
    \node (S11) at (1,1) {$S_{11}$};
    \node (S9) at (2,1) {$S_9$};
    \node (S7) at (3,1) {$S_7$};
    \node (S4) at (4,1) {$S_4$};
    \node (S1) at (5,1) {$S_1$};
    
    \draw[->] (S13) -- (S11);
    \draw[->] (S11) -- (S9);
    \draw[->] (S9) -- (S7);
    \draw[->] (S7) -- (S4);
    \draw[->] (S4) --(S1);
    \draw[->] (S1) to [out=160,in=20] (S13);
    
    \node (S5) at (0,0) {$S_5$};
    \node (S2) at (1,0) {$S_2$};
    \node (S14) at (2,0) {$S_{14}$};
    \node (S12) at (3,0) {$S_{12}$};
    \node (S10) at (4,0) {$S_{10}$};
    \node (S8) at (5,0) {$S_8$};
    \node (S6) at (6,0) {$S_6$};
    \node (S3) at (7,0) {$S_3$\nospacepunct{.}};
    
    \draw[->] (S5) -- (S2);
    \draw[->] (S2) -- (S14);
    \draw[->] (S14) -- (S12);
    \draw[->] (S12) -- (S10);
    \draw[->] (S10) -- (S8);
    \draw[->] (S8) -- (S6);
    \draw[->] (S6) -- (S3);
    \draw[->] (S3) to [out=160,in=20] (S14);
    \end{tikzpicture}
\]

\end{example}

We obtain the following result which describes the cyclic simple modules for a cyclic non-homogeneous Nakayama algebra with an $n\ZZ$-cluster tilting subcategory.

\begin{lemma}\label{lem:cyclic simple modules}
 Let $\tilde{M}(i,i)$ be a simple $\tilde{\La}$-module where $m_k\leq i<m_{k+1}$ for some $1\leq k\leq r$. Then $\tilde{M}(i,i)$ is cyclic if and only if 
\[
i\equiv m_k \modulo{l_k} \quad \text{or} \quad i\equiv m_k-1 \modulo{l_k}.
\]
\end{lemma}

\begin{proof}
Assume $i\equiv m_k \text{ or } i\equiv m_k-1 \modulo{l_k}$. Then $\tilde{M}(i,i)$ corresponds to a vertex in $\bigcup_{k'}Q^{k',0}\cup\bigcup_{k'}Q^{k',l_{k'}-1}$, using the notation of Lemma \ref{lem:Resolution quiver}. From the description of the resolution quiver $Q$ in Lemma \ref{lem:Resolution quiver}, we see that any vertex in $\bigcup_{k'}Q^{k',0}\cup\bigcup_{k'}Q^{k',l_{k'}-1}$ has a unique predecessor and a unique successor in $\bigcup_{k'}Q^{k',0}\cup\bigcup_{k'}Q^{k',l_{k'}-1}$. Since $\bigcup_{k'}Q^{k',0}\cup\bigcup_{k'}Q^{k',l_{k'}-1}$ has finitely many vertices, it follows that $\bigcup_{k'}Q^{k',0}\cup\bigcup_{k'}Q^{k',l_{k'}-1}$ is a union of cycles. This shows that $\tilde{M}(i,i)$ is cyclic. 

Conversely, assume $i\not\equiv m_k \text{ and } i\not\equiv m_k-1 \modulo{l_k}$. Then $\tilde{M}(i,i)$ corresponds to a vertex in $Q^{k,j}$ where $j\neq 0$ and $j\neq l_k-1$. From the description of the resolution quiver $Q$ in Lemma \ref{lem:Resolution quiver}, we see that any vertex in $Q^{k,j}$ is either a source vertex or has a unique predecessor which is in $Q^{k,j}$. Since $Q^{k,j}$ is a linearly oriented $A_{\frac{n}{2}}$-quiver, it contains no cycles. This shows that $\tilde{M}(i,i)$ is not cyclic.
\end{proof}

\begin{example}\label{ex:cyclic simple modules}
Let $\tilde{\La}$ be as in Example~\ref{ex:resolution quiver}. Notice that $l_2=l_3=2$. By Lemma~\ref{lem:cyclic simple modules} it follows that for every $m_2\leq i \leq m_4$ we have that $\tilde{M}(i,i)$ is cyclic. Since $l_1=3$, again by Lemma~\ref{lem:cyclic simple modules} we conclude that the only non-cyclic simple modules are $\tilde{M}(2,2)$ and $\tilde{M}(5,5)$. The resolution quiver $\tilde{Q}$ of $\tilde{\La}$ computed in Example~\ref{ex:resolution quiver} confirms our computation.
\end{example}

We can now determine the subcategories $\cX_c$ and $\cF=\add \cX_c$ of $\mo {\tilde{\La}}$ described in subsection \ref{subsection:SingCatNakayama}.

\begin{theorem}\label{thm:Indecomp in F}
The subcategory $\cX_c$ consists of the following types of indecomposables $\tilde{\La}$-modules, where we run through all pairs of integers $(k,i)$ where $1\leq k\leq r$ and where $m_k\leq i< m_{k+1}$ and $i\equiv m_k\modulo{l_k}$:
\begin{enumerate}
\item $\tilde{M}(i,i)$.
\item $\tilde{M}(i+1,i+l_k-1)$.
\item $\tilde{M}(i,i+l_k-1)$.
\item $\tilde{M}(i+1,i+l_k)$.
\end{enumerate}
Furthermore, a module $\tilde{M}(i,j)$ is projective in $\cF$ if and only if it is of type $(3)$ or $(4)$.	
\end{theorem}

\begin{proof}
Let $\tilde{M}(i,j)$ be an indecomposable $\tilde{\La}$-module with $m_k\leq i\leq j\leq m_{k+1}$. By definition, we have that $\tilde{M}(i,j)\in \cX_c$ if and only if $\topp \tilde{M}(i,j)$ and $\tau \soc \tilde{M}(i,j)$ are cyclic simple $\tilde{\La}$-modules. Since \[
\topp \tilde{M}(i,j)\cong \tilde{M}(j,j) \quad \text{and} \quad \tau \soc \tilde{M}(i,j)\cong \tilde{M}(i-1,i-1)
\]
by Proposition \ref{prop:AR quiver of Nakayama} and Proposition \ref{prop:basic Nakayama results2} (a), this is equivalent to requiring $\tilde{M}(j,j)$ and $\tilde{M}(i-1,i-1)$ to be cyclic simple $\tilde{\La}$-modules. Now by Lemma \ref{lem:cyclic simple modules} this is equivalent to 
\[
i-1\equiv m_k \text{ or }i-1\equiv m_k-1 \modulo{l_k}.
\]
and
\[
j\equiv m_k  \text{ or }  j\equiv m_k-1 \modulo{l_k}.
\]
Analyzing the different cases for $i$ and $j$, and using that $0\leq j-i\leq l_k-1$, we get the list of indecomposable $\tilde{\La}$-modules given in the theorem.

Finally, we prove that $\tilde{M}(i,j)$ is projective in $\cF$ if and only if is of type $(3)$ or $(4)$. Indeed, since $\rmax i=i+l_k-1$ when $m_k\leq i$ and $i+l_k-1\leq m_{k+1}$, it follows that any module in the set described by $(3)$ or $(4)$ is projective in $\mo \La$ by Proposition \ref{prop:basic Nakayama results2} (b). Hence any such module must also be projective in $\cF$. On the other hand, for any module $\tilde{M}(i+1,i+l_k-1)$ of type $(2)$ we have a non-split epimorphism 
\[
\tilde{M}(i,i+l_k-1)\to \tilde{M}(i+1,i+l_k-1)
\]
where $\tilde{M}(i,i+l_k-1)$ is in the set described by $(3)$. This shows that modules of type $(2)$ are not projective in $\cF$. Similarly, for any module $\tilde{M}(i,i)$ of type $(1)$ where $m_k<i\leq m_{k+1}$ we have a non-split epimorphism 
\[
\tilde{M}(i-l_k+1,i)\to \tilde{M}(i,i)
\]
where $\tilde{M}(i-l_k+1,i)$ is in the set described by $(4)$. This shows that modules of type $(1)$ are not projective in $\cF$.
\end{proof}

We can now describe the singularity category of $\tilde{\La}$, using the subcategory $\cF$ of $\mo {\tilde{\La}}$ which we computed in Theorem \ref{thm:Indecomp in F}.

\begin{corollary}\label{cor:SingCatNonHomog}
There exists an equivalence $\cF\cong \mo \Gamma$ where $\Gamma=k\tilde{A}_m/R^2$ is a homogeneous Nakayama algebra with $m=rn$. In particular, we have an equivalence 
	\[
	D_{\operatorname{sing}}(\tilde{\La})\cong \smo \Gamma
	\]
of triangulated categories.
\end{corollary}

\begin{proof}
Since $\cF$ is an abelian category with enough projectives, we have that $\cF\cong \mo{\cP^{\operatorname{op}}}$ where $\cP$ is the subcategory of projective objects in $\cF$ and $\mo{\cP^{\operatorname{op}}}$ is the category of finitely presented $k$-linear functors from $\cP^{\operatorname{op}}$ to $\mo k$, see \cite[Proposition 2.3]{Kra}. By Theorem \ref{thm:Indecomp in F} we know that $\cP$ is the additive closure of modules of type $(3)$ and $(4)$. We enumerate the modules in $(3)$ and $(4)$ by letting $t_s=2\sum_{k=1}^{s-1}\frac{m_{k+1}-m_k}{l_k}$ for each $1\leq s\leq r$ and setting
\[ 
P_{t_s+2i}=\tilde{M}(m_{s}+il_{s},m_{s}+(i+1)l_{s}-1) \quad \text{and} \quad P_{t_s+2i+1}=\tilde{M}(m_{s}+il_{s}+1,m_{s}+(i+1)l_{s})
\]
where $0\leq i< \frac{m_{s+1}-m_s}{l_s}$. With this notation we have the indecomposable projective modules $P_0,P_1,\ldots, P_{m-1}$ in $\cF$, since 
\[
2\sum_{k=1}^{r}\frac{m_{k+1}-m_k}{l_k}=2\sum_{k=1}^{r}\frac{n}{2}=rn=m
\]
Now for $0\leq i\leq m-1$ we have that
\[
\Hom_{\tilde{\La}}(P_i,P_j) = \begin{cases} 
k &\mbox{if $j=i$ or $j=i+1$},\\ 
0  &\mbox{otherwise}
\end{cases}
\]
where $P_m:=P_0$. Hence $\cP$ is equivalent to the category $\add \Gamma$ of projective $\Gamma$-modules. Therefore we have that \[
\mo {\cP^{\operatorname{op}}}\cong \mo {(\add \Gamma)^{\operatorname{op}}}\cong \mo {\Gamma^{\operatorname{op}}}\cong \mo \Gamma
\]
where the last equivalence follows from the fact that $\Gamma\cong \Gamma^{\operatorname{op}}$. This proves the claim. 
\end{proof}

\begin{remark}
We note that there can be more $n\ZZ$-cluster tilting subcategories in the singularity category than in the module category. Indeed, by Theorem \ref{thrm:cyclic non-homogeneous case} we know that $\tilde{\La}$ has a unique $n\ZZ$-cluster tilting subcategory. On the other hand, by Corollary \ref{cor:SingCatNonHomog} we know that $D_{\operatorname{sing}}(\tilde{\La})\cong \smo \Gamma$, and hence the singularity category of $\tilde{\La}$ has $n$ different $n\ZZ$-cluster tilting subcategories by Remark \ref{rmk:homogeneous cyclic Nakayama with nZ-ct}. These are precisely the subcategories of $\smo \Gamma$ consisting of a simple module $S$ and its $n$-syzygies $\Omega^{k}(S)$ for $0\leq k\leq r-1$. Under the isomorphism $\underline{\cF}\cong D_{\operatorname{sing}}(\tilde{\La})$, the $n\ZZ$-cluster tilting subcategory of  $D_{\operatorname{sing}}(\tilde{\La})$ given in Theorem \ref{thrm:nZ-CTSingCat} corresponds to the subcategory of $\underline{\cF}$ consisting of the simple modules $\tilde{M}(m_k,m_k)$ for all $1\leq k\leq r$.

On the other hand the different $n\ZZ$-cluster tilting subcategories in $D_{\operatorname{sing}}(\tilde{\La})$ are all related by applying some automorphism of $D_{\operatorname{sing}}(\tilde{\La})$. Based on this observation we pose the following question, to which we have no counter example.
\end{remark}

\begin{question}
Given a finite dimensional algebra $A$ with $\cC \subseteq D_{\operatorname{sing}}(A)$ an $n\ZZ$-cluster tilting subcategory. Does there exist a finite dimensional algebra $B$ together with $\cD \subseteq \mo{B}$ an $n\ZZ$-cluster tilting subcategory and an equivalence $D_{\operatorname{sing}}(B) \to D_{\operatorname{sing}}(A)$ such that the $n\ZZ$-cluster tilting subcategory of $D_{\operatorname{sing}}(B)$ corresponding to $\cD$ is sent to $\cC$?
\end{question}

\begin{example}
Let $\tilde{\La}$ be as in Example~\ref{ex:resolution quiver}. In Example~\ref{ex:non-homogeneous cyclic case} we have seen that $\tilde{\La}$ admits a $4\ZZ$-cluster tilting subcategory $\cC_{\tilde{\La}}$. Using Theorem~\ref{thm:Indecomp in F} we can compute $\cX_{c}$. Indeed, the Auslander--Reiten quiver of $\tilde{\La}$ was computed in Example~\ref{ex:non-homogeneous cyclic case} to be
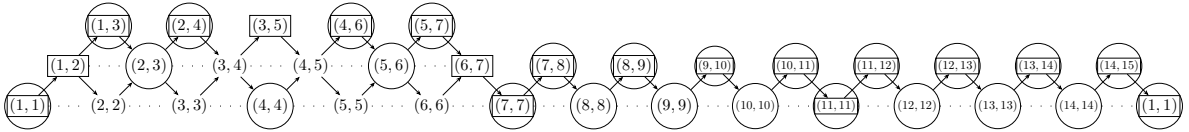
\[\resizebox {\columnwidth} {!} {
    \begin{tikzpicture}[scale=0.9, transform shape, baseline={(current bounding box.center)}]
    
    \tikzstyle{mod}=[rectangle, minimum width=6pt, draw=none, inner sep=1.5pt, scale=0.8]
    \tikzstyle{nct}=[rectangle, minimum width=3pt, draw, inner sep=1.5pt, scale=0.8]
    \tikzstyle{ncts}=[rectangle, minimum width=3pt, draw, inner sep=1.5pt, scale=0.6]
    \tikzstyle{Xc}=[circle, minimum width=3pt, draw, inner sep=1.5pt, scale=0.8]
    \tikzstyle{Xcs}=[circle, minimum width=3pt, draw, inner sep=1.5pt, scale=0.6]
    
    \node[Xc] (11) at (0,0) {$(1,1)$};
    \node[nct] (11) at (0,0) {\phantom{$(1,1)$}};
    \node[mod] (22) at (1.4,0) {$(2,2)$};
    \node[mod] (33) at (2.8,0) {$(3,3)$};
    \node[Xc] (44) at (4.2,0) {$(4,4)$};
    \node[mod] (55) at (5.6,0) {$(5,5)$};
    \node[mod] (66) at (7,0) {$(6,6)$};
    \node[Xc] (77) at (8.4,0) {$(7,7)$};
    \node[nct] (77) at (8.4,0) {\phantom{$(7,7)$}};
    
    \draw[loosely dotted] (11.east) -- (22);
    \draw[loosely dotted] (22.east) -- (33);
    \draw[loosely dotted] (33.east) -- (44);
    \draw[loosely dotted] (44.east) -- (55);
    \draw[loosely dotted] (55.east) -- (66);
    \draw[loosely dotted] (66.east) -- (77);
        
    \node[nct] (12) at (0.7,0.7) {$(1,2)$};   
    \node[Xc] (23) at (2.1,0.7) {$(2,3)$};    
    \node[mod] (34) at (3.5,0.7) {$(3,4)$};    
    \node[mod] (45) at (4.9,0.7) {$(4,5)$};    
    \node[Xc] (56) at (6.3,0.7) {$(5,6)$};    
    \node[mod] (67) at (7.7,0.7) {$(6,7)$};  
    \node[nct] (67) at (7.7,0.7) {\phantom{$(6,7)$}};

    \draw[-{Stealth[scale=0.5]}] (11) -- (12);
    \draw[-{Stealth[scale=0.5]}] (22) -- (23);
    \draw[-{Stealth[scale=0.5]}] (33) -- (34);
    \draw[-{Stealth[scale=0.5]}] (44) -- (45);
    \draw[-{Stealth[scale=0.5]}] (55) -- (56);
    \draw[-{Stealth[scale=0.5]}] (66) -- (67);
    
    \draw[-{Stealth[scale=0.5]}] (12) -- (22);
    \draw[-{Stealth[scale=0.5]}] (23) -- (33);
    \draw[-{Stealth[scale=0.5]}] (34) -- (44);
    \draw[-{Stealth[scale=0.5]}] (45) -- (55);
    \draw[-{Stealth[scale=0.5]}] (56) -- (66);
    \draw[-{Stealth[scale=0.5]}] (67) -- (77);

    \draw[loosely dotted] (12.east) -- (23);
    \draw[loosely dotted] (23.east) -- (34);
    \draw[loosely dotted] (34.east) -- (45);
    \draw[loosely dotted] (45.east) -- (56);
    \draw[loosely dotted] (56.east) -- (67);
    
    \node[Xc] (13) at (1.4,1.4) {$(1,3)$};
    \node[nct] (13) at (1.4,1.4) {\phantom{$(1,3)$}};
    \node[Xc] (24) at (2.8,1.4) {$(2,4)$};
    \node[nct] (24) at (2.8,1.4) {\phantom{$(2,4)$}};
    \node[mod] (35) at (4.2,1.4) {$(3,5)$};
    \node[nct] (35) at (4.2,1.4) {\phantom{$(3,5)$}};
    \node[Xc] (46) at (5.6,1.4) {$(4,6)$};
    \node[nct] (46) at (5.6,1.4) {\phantom{$(4,6)$}};
    \node[Xc] (57) at (7,1.4) {$(5,7)$};
    \node[nct] (57) at (7,1.4) {\phantom{$(5,7)$}};
    
    \draw[-{Stealth[scale=0.5]}] (12) -- (13);
    \draw[-{Stealth[scale=0.5]}] (23) -- (24);
    \draw[-{Stealth[scale=0.5]}] (34) -- (35);
    \draw[-{Stealth[scale=0.5]}] (45) -- (46);
    \draw[-{Stealth[scale=0.5]}] (56) -- (57);
    
    \draw[-{Stealth[scale=0.5]}] (13) -- (23);
    \draw[-{Stealth[scale=0.5]}] (24) -- (34);
    \draw[-{Stealth[scale=0.5]}] (35) -- (45);
    \draw[-{Stealth[scale=0.5]}] (46) -- (56);
    \draw[-{Stealth[scale=0.5]}] (57) -- (67);
    
    \node[Xc] (88) at (9.8,0) {$(8,8)$};
    \node[Xc] (99) at (11.2,0) {$(9,9)$};
    \node[Xcs] (1010) at (12.6,0) {$(10,10)$};
    \node[Xcs] (1111) at (14,0) {$(11,11)$};
    \node[ncts] (1111) at (14,0) {\phantom{$(11,11)$}};
    \node[Xcs] (1212) at (15.4,0) {$(12,12)$};
    \node[Xcs] (1313) at (16.8,0) {$(13,13)$};
    \node[Xcs] (1414) at (18.2,0) {$(14,14)$};
    \node[Xc] (11b) at (19.6,0) {$(1,1)$};
    \node[nct] (11b) at (19.6,0) {\phantom{$(1,1)$}};
    
    \draw[loosely dotted] (77.east) -- (88);
    \draw[loosely dotted] (88.east) -- (99);
    \draw[loosely dotted] (99.east) -- (1010);
    \draw[loosely dotted] (1010.east) -- (1111);
    \draw[loosely dotted] (1111.east) -- (1212);
    \draw[loosely dotted] (1212.east) -- (1313);
    \draw[loosely dotted] (1313.east) -- (1414);
    \draw[loosely dotted] (1414.east) -- (11b);
    
    \node[Xc] (78) at (9.1,0.7) {$(7,8)$};
    \node[nct] (78) at (9.1,0.7) {\phantom{$(7,8)$}};
    \node[Xc] (89) at (10.5,0.7) {$(8,9)$};    
    \node[nct] (89) at (10.5,0.7) {\phantom{$(8,9)$}};
    \node[Xcs] (910) at (11.9,0.7) {$(9,10)$};    
    \node[ncts] (910) at (11.9,0.7) {\phantom{$(9,10)$}};
    \node[Xcs] (1011) at (13.3,0.7) {$(10,11)$};  
    \node[ncts] (1011) at (13.3,0.7) {\phantom{$(10,11)$}}; 
    \node[Xcs] (1112) at (14.7,0.7) {$(11,12)$}; 
    \node[ncts] (1112) at (14.7,0.7) {\phantom{$(11,12)$}};
    \node[Xcs] (1213) at (16.1,0.7) {$(12,13)$};    
    \node[ncts] (1213) at (16.1,0.7) {\phantom{$(12,13)$}};
    \node[Xcs] (1314) at (17.5,0.7) {$(13,14)$};    
    \node[ncts] (1314) at (17.5,0.7) {\phantom{$(13,14)$}};
    \node[Xcs] (1415) at (18.9,0.7) {$(14,15)$};   
    \node[ncts] (1415) at (18.9,0.7) {\phantom{$(14,15)$}};

    \draw[-{Stealth[scale=0.5]}] (77) -- (78);
    \draw[-{Stealth[scale=0.5]}] (88) -- (89);
    \draw[-{Stealth[scale=0.5]}] (99) -- (910);
    \draw[-{Stealth[scale=0.5]}] (1010) -- (1011);
    \draw[-{Stealth[scale=0.5]}] (1111) -- (1112);
    \draw[-{Stealth[scale=0.5]}] (1212) -- (1213);
    \draw[-{Stealth[scale=0.5]}] (1313) -- (1314);
    \draw[-{Stealth[scale=0.5]}] (1414) -- (1415);
    
    \draw[-{Stealth[scale=0.5]}] (78) -- (88);
    \draw[-{Stealth[scale=0.5]}] (89) -- (99);
    \draw[-{Stealth[scale=0.5]}] (910) -- (1010);
    \draw[-{Stealth[scale=0.5]}] (1011) -- (1111);
    \draw[-{Stealth[scale=0.5]}] (1112) -- (1212);
    \draw[-{Stealth[scale=0.5]}] (1213) -- (1313);
    \draw[-{Stealth[scale=0.5]}] (1314) -- (1414);
    \draw[-{Stealth[scale=0.5]}] (1415) -- (11b);
    \end{tikzpicture}}
    \]
where $\cX_{c}$ consists of the encircled modules and $\cC_{\tilde{\La}}$ is the additive closure of all modules inside a rectangle. By Corollary~\ref{cor:SingCatNonHomog} we have that $\cF\cong \mo{\Gamma}$ where $\Gamma=\K \tilde{A}_{12}/R^2$. The $4\ZZ$-cluster tilting subcategory $\cC_{\tilde{\La}}$ of $\mo{\tilde{\La}}$ gives rise to the $4\ZZ$-cluster tilting subcategory $\cC_{\cF}=\cC_{\tilde{\La}}\cap \cF$ of $\cF\cong\mo{\Gamma}$. However, there are $3$ different $4\ZZ$-cluster tilting subcategories inside $\mo{\Gamma}$ which give rise to $3$ different $4\ZZ$-cluster tilting subcategories inside $D_{\operatorname{sing}}(\tilde{\La})\cong \smo \Gamma$.
\end{example}

Finally we describe the functor $\mo {\tilde{\La}}\to D_{\operatorname{sing}}(\tilde{\La})$, using the subcategory $\cF$ of $\mo {\tilde{\La}}$.

\begin{proposition}\label{prop:MapToSing}
 Let $\tilde{M}(i,j)$ be an indecomposable $\tilde{\La}$-module where $m_k\leq i\leq j\leq m_{k+1}$. The following hold:
\begin{enumerate}
\item $\tilde{M}(i,j)$ is nonzero in $D_{\operatorname{sing}}(\tilde{\La})$ if and only if $j-i<l_{k}-1$ and either 
\[
i\equiv m_k+1\modulo{l_k}  \quad \text{or} \quad j\equiv m_{k}\modulo{l_k}
\] 
hold.
\item If $j-i<l_{k}-1$ and $i\equiv m_k+1\modulo{l_k}$, then the inclusion $ \tilde{M}(i,j)\to \tilde{M}(i,i+l_k-2)$ becomes an isomorphism in $D_{\operatorname{sing}}(\tilde{\La})$.
\item If $j-i<l_{k}-1$ and $j\equiv m_{k}\modulo{l_k}$, then the projection $\tilde{M}(i,j)\to \tilde{M}(j,j)$ becomes an isomorphism in $D_{\operatorname{sing}}(\tilde{\La})$.
\end{enumerate}	
\end{proposition}

\begin{remark}
Note that the codomains of the maps in part $(2)$ and $(3)$ of Proposition \ref{prop:MapToSing}  are non-projective objects in $\cF$ by Theorem \ref{thm:Indecomp in F}. Also, any morphism $\tilde{M}(i_1,j_1)\to \tilde{M}(i_2,j_2)$ which is not of the form given in Proposition \ref{prop:MapToSing} $(2)$ or Proposition \ref{prop:MapToSing} $(3)$ must be zero in $D_{\operatorname{sing}}(\tilde{\La})$, since it factors through a module $\tilde{M}(i,j)$ which is not of the form given in Proposition \ref{prop:MapToSing} $(1)$. This implies that Proposition \ref{prop:MapToSing} contains all the information about the behavior of the functor $\mo{\tilde{\La}}\to D_{\operatorname{sing}}(\tilde{\La})\xrightarrow{\cong} \underline{\cF}$ on the objects and morphisms of $\mo{\tilde{\La}}$.
\end{remark}

\begin{proof}[Proof of Proposition \ref{prop:MapToSing}]
Let $\cD_1$ be the set of indecomposable $\tilde{\La}$-modules $\tilde{M}(i,j)$ where 
\[
m_k\leq i\leq j\leq m_{k+1}\quad \text{and} \quad j-i<l_k
\]
and where
\[
i\equiv m_k+1\modulo{l_k}  \quad \text{or} \quad j\equiv m_{k}\modulo{l_k}
\] 
for some $1\leq k\leq r$. Let $\cD_2$ be the set of indecomposable $\tilde{\La}$-modules $\tilde{M}(i,j)$ which are not contained in $\cD_1$. We want to show that the modules in $\cD_1$ (resp $\cD_2$) become non-zero (resp zero) in the singularity category. We first show that up to isomorphism $\cD_1$ and $\cD_2$ are closed under syzygies. To this end, let $\tilde{M}(i,j)$ be a $\tilde{\La}$-module where $m_k\leq i\leq j\leq m_{k+1}$ and where $(i,j)\neq (m_k,m_k)$. If $j\leq m_k+l_k-1$, then by Proposition \ref{prop:basic Nakayama results2} (b) we have that 
\begin{equation}\label{syzygy1}
\Omega \tilde{M}(i,j)\cong \tilde{M}(m_k,i-1)
\end{equation}
Note that $j\equiv m_k\modulo{l_k}$ cannot hold in this case, and $i\equiv m_k+1\modulo{l_k}$ implies that $i=m_k+1$. Hence $\tilde{M}(i,j)$ is contained in $\cD_1$ if and only if $i=m_k+1$, which is equivalent to $\tilde{M}(m_k,i-1)$ being contained in $\cD_1$. This shows that $\cD_1$ and $\cD_2$ are closed under syzygies when $j\leq m_k+l_k-1$. Now assume $j>m_k+l_k-1$. Then by Proposition \ref{prop:basic Nakayama results2} (b) we have an isomorphism 
\begin{equation}\label{syzygy2}
\Omega \tilde{M}(i,j)\cong \tilde{M}(j-l_k+1,i-1)
\end{equation}
Since $i\equiv m_k+1\modulo{l_k}$ if and only if $i-1\equiv m_k\modulo{l_k}$, and $j\equiv m_{k}\modulo{l_k}$ if and only if $j-l_k+1\equiv m_k+1\modulo{l_k}$, we see that $\tilde{M}(i,j)$ is in $\cD_1$ if and only if $\tilde{M}(j-l_k+1,i-1)$ is in $\cD_1$. This shows that $\cD_1$ and $\cD_2$ are closed under syzygies. 

Now let $\tilde{M}(i,j)$ be a $\tilde{\La}$-module where $m_k\leq i \leq j\leq m_{k+1}$. We claim that there exists a syzygy of $\tilde{M}(i,j)$ which is isomorphic to $\tilde{M}(m_k,t)$ for some $t\geq m_k$. Indeed, if $i=m_k$, then we are done, so assume $i>m_k$.  Repeatedly using formula \eqref{syzygy2}, we see that there exists an integer $s\geq 0$ for which  $\Omega^s \tilde{M}(i,j)\cong \tilde{M}(i',j')$ where $m_k\leq i'\leq j'\leq m_k+l_k-1$. Hence, by \eqref{syzygy1} it follows that 
\[
\Omega^{s+1} \tilde{M}(i,j)\cong \tilde{M}(m_k,i'-1).
\]
Setting $t=i'-1$, this proves the claim. 

Now assume $\tilde{M}(i,j)$ is in $\cD_2$, and let $t$ be an integer such that $\tilde{M}(m_k,t)$ is isomorphic to a syzygy of $\tilde{M}(i,j)$. Since $\cD_2$ is closed under syzygies, $\tilde{M}(m_k,t)$ must be in $\cD_2$. Therefore $t>m_k$, which implies that $\tilde{M}(m_k,t)$ is a projective $\tilde{\La}$-module by Proposition \ref{prop:basic gluing results} (b) and Proposition \ref{prop:basic self-gluing results} (b). Hence, $\tilde{M}(m_k,t)$ is zero in $D_{\operatorname{sing}}(\tilde{\La})$, so $\tilde{M}(i,j)$ must also be zero in $D_{\operatorname{sing}}(\tilde{\La})$.  This shows the "only if" direction of part $(1)$ of the proposition.

To prove the remaining part of the proposition, we need to consider the set $\cD_3$ consisting of the identity morphisms between objects in $\cD_1$, the morphisms in part $(2)$ of the proposition, and the morphisms and in part $(3)$ of the proposition. We claim that up to isomorphism $\cD_3$ is closed under syzygies. Clearly this is true for the identity morphisms, since $\cD_1$ is closed under syzygies. Now let $\tilde{M}(i,j)\to \tilde{M}(i,i+l_k-2)$ be a morphism as in $(2)$. We consider two cases separately:
\begin{itemize}
\item If $j\leq m_k+l_k-1$, then taking the syzygy and using \eqref{syzygy1}, we get the identity map 
\[
\tilde{M}(m_k,i-1)\xrightarrow{=}\tilde{M}(m_k,i-1)
\] 
which is in $\cD_3$;
\item If $j> m_k+l_k-1$, then  taking the syzygy and using \eqref{syzygy2}, we get the projection map 
\[
\tilde{M}(j-l_k+1,i-1)\to \tilde{M}(i-1,i-1).
\]
This is of type $(3)$ since $i-1\equiv m_k\modulo{l_k}$, and hence it must be in $\cD_3$.
\end{itemize}
Finally, let $\tilde{M}(i,j)\to \tilde{M}(j,j)$ be a morphism as in $(3)$. We can assume that $(i,j)\neq (m_k,m_k)$ since otherwise we just have the identity map, in which case we already know the claim is true. Therefore $j>m_k$, and since $j\equiv m_k\modulo{l_k}$, we must have that $j>m_k+l_k-1$. Hence, taking the syzygy and using the formula \eqref{syzygy2}, we get the inclusion map 
\[
\tilde{M}(j-l_k+1,i-1)\to \tilde{M}(j-l_k+1,j-1).
\] 
This is of type $(2)$ since $j-l_k+1\equiv m_k+1\modulo{l_k}$, and hence it must be in $\cD_3$. This shows that $\cD_3$ is closed under syzygies. 

Now let $\tilde{M}(i_1,j_j)\to \tilde{M}(i_2,j_2)$ be a map in $\cD_3$ with $m_k\leq i_1 \leq j_1\leq m_{k+1}$ and $m_k\leq i_2 \leq j_2\leq m_{k+1}$. As shown above, we can find an integer $s\geq 0$ such that $\Omega^s\tilde{M}(i_1,j_j)\cong \tilde{M}(m_k,t)$. Now since $\cD_1$ is closed under syzygies, $\tilde{M}(m_k,t)$ must be in $\cD_1$, so $t=m_k$. Furthermore, since $\cD_3$ is closed under syzygies, the map 
\[
\Omega^s\tilde{M}(i_1,j_j)\to \Omega^s\tilde{M}(i_2,j_2)
\]
must be isomorphic to the identity map 
\[
\tilde{M}(m_k,m_k)\xrightarrow{=}\tilde{M}(m_k,m_k)
\]
since this is  the only map in $\cD_3$ with domain isomorphic to $\tilde{M}(m_k,m_k)$. Using that the syzygy functor extends to an autoequivalence on the singularity category, we get that the map $\tilde{M}(i_1,j_1)\to \tilde{M}(i_2,j_2)$ must also be an isomorphism in $D_{\operatorname{sing}}(\tilde{\La})$. This proves part $(2)$ and $(3)$ of the proposition. Finally, the "if" direction of part $(1)$ follows from the fact that the codomain of the maps in $(2)$ and $(3)$ are non-projective objects in $\cF$ by Theorem \ref{thm:Indecomp in F}. 
\end{proof}

\begin{remark}
Let $\cC_{\tilde{\La}}$ denote the $n\ZZ$-cluster tilting subcategory of $\tilde{\La}$. By Theorem  \ref{thrm:acyclic non-homogeneous case}, Remark \ref{rem:uniqueness of decomposition in acyclic} (a) and Theorem \ref{thrm:cyclic non-homogeneous case}, we know that the category $\cC_{\tilde{\La}}$ consists of the projective $\tilde{\La}$-modules, the injective $\tilde{\La}$-modules, and the $\tilde{\La}$-modules of the form $\tilde{M}(m_k,m_k)$ where $1\leq k\leq r$. Note that the modules $\tilde{M}(m_k,m_k)$ are in $\cF$ by Theorem \ref{thm:Indecomp in F}. Also, any injective and non-projective indecomposable module is of the form $\tilde{M}(i,m_{k+1})$ for $m_{k+1}-l_k+1<i<m_{k+1}$. By Proposition \ref{prop:MapToSing} (c) such a module becomes isomorphic in $D_{\operatorname{sing}}(\tilde{\La})$ to its top via the map 
	\[
	\tilde{M}(i,m_{k+1})\to \tilde{M}(m_{k+1},m_{k+1}).
	\]
Since the projective modules vanish in $D_{\operatorname{sing}}(\tilde{\La})$, we get a complete description of the functor \[
\cC_{\tilde{\La}}\to D_{\operatorname{sing}}(\tilde{\La})\cong \underline{\cF}.
\]
\end{remark}

The following corollary shows that $\tilde{\La}$ is not Iwanaga--Gorenstein. As a consequence, $D_{\operatorname{sing}}(\tilde{\La})$ cannot be computed by considering the stable category of Gorenstein projective modules as in Buchweitz theorem, see \cite[Theorem 3.6]{BOJ}.

\begin{corollary}\label{cor:not Iwanaga--Gorenstein}
A non-homogeneous cyclic Nakayama algebra with an $n\ZZ$-cluster tilting subcategory is not Iwanaga--Gorenstein.
\end{corollary}

\begin{proof}
We show this for $\tilde{\La}$. It suffices to prove that there exists an injective $\tilde{\La}$-module which is non-zero in $D_{\operatorname{sing}}(\tilde{\La})$. Since $\tilde{\La}$ is non-homogeneous, there exists a $k$ where $l_k\geq 3$. Hence,  $\tilde{M}(m_{k+1}-1,m_{k+1})$ is a $\tilde{\La}$-module which is injective and not projective. By Proposition \ref{prop:MapToSing} $(3)$ this module is isomorphic to $\tilde{M}(m_{k+1},m_{k+1})$ in $D_{\operatorname{sing}}(\tilde{\La})$. Since $\tilde{M}(m_{k+1},m_{k+1})$ is non-zero in $D_{\operatorname{sing}}(\tilde{\La})$ by Theorem \ref{thm:Indecomp in F}, this proves the claim.
\end{proof}


\begin{thebibliography}{HIMO20}

\bibitem[ASS06]{ASS}
Ibrahim Assem, Daniel Simson, and Andrzej Skowro{\'n}ski.
\newblock {\em {Elements of the Representation Theory of Associative Algebras:
  Volume 1: Techniques of Representation Theory}}.
\newblock {Elements of the Representation Theory of Associative Algebras}.
  Cambridge University Press, 2006.

\bibitem[BOJ15]{BOJ}
Petter~Andreas Bergh, Steffen Oppermann, and David~A. Jorgensen.
\newblock {The Gorenstein defect category}.
\newblock {\em The Quarterly Journal of Mathematics}, 66(2):459--471, 02 2015.

\bibitem[Buc21]{Buc}
Ragnar-Olaf Buchweitz.
\newblock {\em Maximal {C}ohen-{M}acaulay modules and {T}ate cohomology},
  volume 262 of {\em Mathematical Surveys and Monographs}.
\newblock American Mathematical Society, 2021.
\newblock With appendices and an introduction by Luchezar L. Avramov, Benjamin
  Briggs, Srikanth B. Iyengar and Janina C. Letz.

\bibitem[DI20]{DI}
Erik Darp{\"o} and Osamu Iyama.
\newblock {$d$-Representation-finite self-injective algebras}.
\newblock {\em Advances in Mathematics}, 362:106932, 2020.

\bibitem[DJL21]{DJL}
Tobias {Dyckerhoff}, Gustavo {Jasso}, and Yanki {Lekili}.
\newblock {The symplectic geometry of higher Auslander algebras: Symmetric
  products of disks}.
\newblock {\em Forum of Mathematics, Sigma}, 9:e10, 2021.

\bibitem[DJW19]{DJW}
Tobias Dyckerhoff, Gustavo Jasso, and Tashi Walde.
\newblock Simplicial structures in higher {A}uslander-{R}eiten theory.
\newblock {\em Adv. Math.}, 355:106762, 73, 2019.

\bibitem[DW16]{DW}
Will Donovan and Michael Wemyss.
\newblock Noncommutative deformations and flops.
\newblock {\em Duke Math. J.}, 165(8):1397--1474, 2016.

\bibitem[EH08]{EH}
Karin Erdmann and Thorsten Holm.
\newblock Maximal n-orthogonal modules for selfinjective algebras.
\newblock {\em Proceedings of the American Mathematical Society},
  136(9):3069--3078, 2008.

\bibitem[GKO13]{GKO}
Christof Geiss, Bernhard Keller, and Steffen Oppermann.
\newblock {$n$}-angulated categories.
\newblock {\em J. Reine Angew. Math.}, 675:101--120, 2013.

\bibitem[HI11b]{HI1}
Martin Herschend and Osamu Iyama.
\newblock Selfinjective quivers with potential and 2-representation-finite
  algebras.
\newblock {\em Compositio Mathematica}, 147(6):1885--1920, 2011.

\bibitem[HI11a]{HI2}
Martin Herschend and Osamu Iyama.
\newblock {n-representation-finite algebras and twisted fractionally
  Calabi--Yau algebras}.
\newblock {\em Bulletin of the London Mathematical Society}, 43(3):449--466,
  2011.

\bibitem[HIMO20]{HIMO}
Martin Herschend, Osamu Iyama, Hiroyuki Minamoto, and Steffen Oppermann.
\newblock {Representation theory of Geigle--Lenzing complete intersections}.
\newblock {\em arXiv e-prints}, 2020, 1409.0668.

\bibitem[HJr21]{HJ}
Martin Herschend and Peter J\o rgensen.
\newblock Classification of higher wide subcategories for higher {A}uslander
  algebras of type {$A$}.
\newblock {\em J. Pure Appl. Algebra}, 225(5):Paper No. 106583, 22, 2021.

\bibitem[HJS22]{HJS}
Johanne Haugland, Karin~M. Jacobsen, and Sibylle Schroll.
\newblock The role of gentle algebras in higher homological algebra.
\newblock {\em Forum Mathematicum}, 2022.

\bibitem[HLN22]{HLN}
Martin Herschend, Yu~Liu, and Hiroyuki Nakaoka.
\newblock {$n$}-exangulated categories ({II}): {C}onstructions from
  {$n$}-cluster tilting subcategories.
\newblock {\em J. Algebra}, 594:636--684, 2022.

\bibitem[IJ17]{IJ}
Osamu Iyama and Gustavo Jasso.
\newblock {Higher Auslander correspondence for dualizing R-varieties}.
\newblock {\em Algebras and Representation Theory}, 20(2):335--354, Apr 2017.

\bibitem[IO11]{IO1}
Osamu Iyama and Steffen Oppermann.
\newblock {n-representation-finite algebras and n-APR tilting}.
\newblock {\em Transactions of the American Mathematical Society},
  363(12):6575--6614, 2011.

\bibitem[IO13]{IO2}
Osamu Iyama and Steffen Oppermann.
\newblock {Stable categories of higher preprojective algebras}.
\newblock {\em Advances in Mathematics}, 244:23--68, 2013.

\bibitem[IW11]{IW}
Osamu Iyama and Michael Wemyss.
\newblock A new triangulated category for rational surface singularities.
\newblock {\em Illinois J. Math.}, 55(1):325--341, 2011.

\bibitem[IW13]{IW2}
Osamu Iyama and Michael Wemyss.
\newblock {On the noncommutative Bondal--Orlov conjecture}.
\newblock {\em {Journal für die reine und angewandte Mathematik}},
  2013(683):119 -- 128, 01 Oct. 2013.

\bibitem[IW14]{IW3}
Osamu Iyama and Michael Wemyss.
\newblock {Maximal modifications and Auslander--Reiten duality for non-isolated
  singularities}.
\newblock {\em Inventiones mathematicae}, 197(3):521--586, 2014.

\bibitem[Iya07]{Iya1}
Osamu Iyama.
\newblock {Higher-dimensional Auslander--Reiten theory on maximal orthogonal
  subcategories}.
\newblock {\em Advances in Mathematics}, 210(1):22--50, 2007.

\bibitem[Iya11]{Iya2}
Osamu Iyama.
\newblock {Cluster tilting for higher Auslander algebras}.
\newblock {\em Advances in Mathematics}, 226(1):1--61, 2011.

\bibitem[Jas16]{Jas}
Gustavo Jasso.
\newblock {n-Abelian and n-exact categories}.
\newblock {\em Mathematische Zeitschrift}, 283(3):703--759, 2016.

\bibitem[JKPK19]{JK}
Gustavo Jasso, Julian K{\"u}lshammer, Chrysostomos Psaroudakis, and Sondre
  Kvamme.
\newblock {Higher Nakayama algebras I: Construction}.
\newblock {\em Advances in Mathematics}, 351:1139--1200, 2019.

\bibitem[JMK]{JM}
Gustavo Jasso and Fernando Muro.
\newblock Triangulated {A}uslander--{I}yama correspondence.
\newblock {\em To appear.}

\bibitem[Kra99]{Kra}
Henning Krause.
\newblock Functors on locally finitely presented additive categories.
\newblock {\em Colloq. Math.}, 75, 05 1999.

\bibitem[Kup58]{Kup}
Herbert Kupisch.
\newblock {\em {Beiträge zur Theorie nichthalbeinfacher Ringe mit
  Minimalbedingung}}.
\newblock PhD thesis, NA Heidelberg, 1958.

\bibitem[Kva21]{Kva}
Sondre Kvamme.
\newblock {$d\mathbb Z$}-cluster tilting subcategories of singularity categories.
\newblock {\em Math. Z.}, 297(1-2):803--825, 2021.

\bibitem[Miz13]{Miz}
Yuya Mizuno.
\newblock {A Gabriel-type theorem for cluster tilting}.
\newblock {\em Proceedings of the London Mathematical Society},
  108(4):836--868, 09 2013.

\bibitem[OT12]{OT}
Steffen Oppermann and Hugh Thomas.
\newblock Higher-dimensional cluster combinatorics and representation theory.
\newblock {\em Journal of the European Mathematical Society}, 14(6):1679--1737,
  2012.

\bibitem[Rin13]{Rin}
Claus~Michael Ringel.
\newblock The {G}orenstein projective modules for the {N}akayama algebras. {I}.
\newblock {\em J. Algebra}, 385:241--261, 2013.

\bibitem[She15]{Shen}
Dawei Shen.
\newblock The singularity category of a {N}akayama algebra.
\newblock {\em J. Algebra}, 429:1--18, 2015.

\bibitem[Vas19]{Vas1}
Laertis Vaso.
\newblock {$n$-Cluster tilting subcategories of representation-directed
  algebras}.
\newblock {\em Journal of Pure and Applied Algebra}, 223(5):2101--2122, 2019.

\bibitem[{Vas}20]{Vas2}
Laertis {Vaso}.
\newblock {$n$-cluster tilting subcategories from gluing systems of
  representation-directed algebras}.
\newblock {\em arXiv e-prints}, 2020, 2004.02269.

\bibitem[Vas21]{Vas3}
Laertis Vaso.
\newblock Gluing of {$n$}-cluster tilting subcategories for
  representation-directed algebras.
\newblock {\em Algebr. Represent. Theory}, 24(3):715--781, 2021.

\bibitem[Vas23]{Vas4}
Laertis Vaso.
\newblock n-cluster tilting subcategories for radical square zero algebras.
\newblock {\em Journal of Pure and Applied Algebra}, 227(1):107157, 2023.

\bibitem[Wil22]{Wil}
Nicholas~J. Williams.
\newblock New interpretations of the higher {S}tasheff-{T}amari orders.
\newblock {\em Adv. Math.}, 407:Paper No. 108552, 2022.

\end{thebibliography}
\end{document}